\documentclass[twoside,11pt,reqno]{amsart}
\usepackage{amsmath,amssymb,amscd,mathrsfs,epic,wasysym,latexsym,tikz,mathrsfs,cite,hyperref,enumerate}
\usepackage{pb-diagram}
\usepackage{tikz-cd}
\usepackage[all]{xy}
\usepackage{mathdots}

\makeatletter

\numberwithin{equation}{section}

\allowdisplaybreaks

\newtheorem{Proposition}[equation]{Proposition}
\newtheorem{Lemma}[equation]{Lemma}
\newtheorem{Theorem}[equation]{Theorem}
\newtheorem{Corollary}[equation]{Corollary}
\newtheorem{MainTheorem}{Theorem}

\theoremstyle{definition}  

\newtheorem{Remark}[equation]{Remark}

\newtheorem{Example}[equation]{Example}

\let\<\langle
\let\>\rangle

\newcommand\Comment[2][\relax]{\space\par\medskip\noindent%
   \fbox{\begin{minipage}{\textwidth}\textbf{Comment\ifx\relax#1\else---#1\fi}\newline%
        #2\end{minipage}}\medskip
}

\def\bone{\text{\boldmath$1$}}

\def\pmod#1{\text{ }(\text{\rm mod } #1)\,}

\newcommand{\Hom}{\operatorname{Hom}}

\newcommand{\End}{\operatorname{End}}
\newcommand{\ind}{\operatorname{ind}}
\newcommand{\im}{\operatorname{im}}
\newcommand{\id}{\operatorname{id}}
\def\sgn{\mathtt{sgn}}

\newcommand{\res}{\operatorname{res}}
\newcommand{\soc}{\operatorname{soc}}
\newcommand{\head}{\operatorname{head}}

\newcommand{\Z}{\mathbb{Z}}

\def\eps{{\varepsilon}}
\def\phi{{\varphi}}

\newcommand{\F}{{\mathbb F}}

\newcommand{\ga}{\gamma}

\newcommand{\la}{\lambda}

\newcommand{\al}{\alpha}
\newcommand{\be}{\beta}

\newcommand{\si}{\sigma}

\newcommand{\Om}{\Omega}
\newcommand{\vare}{\varepsilon}

\newcommand{\de}{\delta}

\newcommand{\ka}{\kappa}

\newcommand{\Mull}{{\tt M}}

\newcommand{\Ker}{\operatorname{Ker}}
\newcommand{\Aut}{{\mathrm {Aut}}}

\newcommand{\Irr}{{\mathrm {Irr}}}

\def\id{\mathop{\mathrm {id}}\nolimits}

\def\rank{\mathop{\mathrm{ rank}}\nolimits}

\newcommand{\C}{{\mathbb C}}

\newcommand{\EE}{{\mathcal E}}

\newcommand{\SSS}{{\sf S}}
\newcommand{\AAA}{{\sf A}}
\newcommand{\MMM}{{\sf M}}
\newcommand{\CCC}{{\sf C}}

\newcommand{\da}{{\downarrow}}
\newcommand{\ua}{{\uparrow}}

\renewcommand{\mod}{\bmod \,}

\newcommand{\I}{{\mathcal I}}

\newcommand{\Sub}{{\mathcal S}}

\def\Parp{{\mathscr P}_p}
\def\Par{{\mathscr P}}

\def\k{\Bbbk}

\def\im{{\mathrm{im}\,}}

\def\mod#1{#1\!\operatorname{-mod}}

\def\col{{\tt col}}
\def\row{{\tt row}}

{\catcode`\|=\active
  \gdef\set#1{\mathinner{\lbrace\,{\mathcode`\|"8000%
  \let|\midvert #1}\,\rbrace}}
}
\def\midvert{\egroup\mid\bgroup}

\colorlet{darkgreen}{green!50!black}
\tikzset{dots/.style={very thick,loosely dotted},
         greendot/.style={fill,circle,color=darkgreen,inner sep=1.5pt,outer sep=0},
         blackdot/.style={fill,circle,color=black,inner sep=1.5pt,outer sep=0},
         graydot/.style={fill,circle,color=gray,inner sep=1.1pt,outer sep=0}
}
\def\greendot(#1,#2){\node[greendot] at(#1,#2){}}
\def\blackdot(#1,#2){\node[blackdot] at(#1,#2){}}
\def\graydot(#1,#2){\node[graydot] at(#1,#2){}}

\newenvironment{braid}{
  \begin{tikzpicture}[baseline=6mm,black,line width=1pt, scale=0.32,
                      draw/.append style={rounded corners},
                      every node/.append style={font=\fontsize{5}{5}\selectfont}]%
  }{\end{tikzpicture}
}

\def\Grid(#1,#2){
  \draw[very thin,gray,step=2mm] (0,0)grid(#1,#2);
  \draw[very thin,darkgreen,step=10mm] (0,0)grid(#1,#2);
}

\newcommand\Tableau[2][\relax]{
  \begin{tikzpicture}[scale=0.5,draw/.append style={thick,black}]
    \ifx\relax#1\relax%
    \else 
      \foreach\box in {#1} { \filldraw[blue!30]\box+(-.5,-.5)rectangle++(.5,.5); }
    \fi
    \newcount\row\newcount\col
    \row=0
    \foreach \Row in {#2} {
       \col=1
       \foreach\k in \Row {
          \draw(\the\col,\the\row)+(-.5,-.5)rectangle++(.5,.5);
          \draw(\the\col,\the\row)node{\k};
          \global\advance\col by 1
       }
       \global\advance\row by -1
    }
  \end{tikzpicture}
}

\newcommand\YoungDiagram[2][\relax]{
  \begin{tikzpicture}[scale=0.5,draw/.append style={thick,black}]
    \ifx\relax#1\relax%
    \else 
    \foreach\box in {#1} {
      \filldraw[blue!30]\box rectangle ++(1,1);
    }
    \fi
    \newcount\row
    \row=0
    \foreach \col in {#2} {
       \draw(1,\the\row)grid ++(\col,1);
       \global\advance\row by -1
    }
  \end{tikzpicture}
}

\newdimen\hoogte    \hoogte=12pt    
\newdimen\breedte   \breedte=14pt  
\newdimen\dikte     \dikte=0.5pt 

\newenvironment{Young}{\begingroup
       \def\vr{\vrule height0.89\hoogte width\dikte depth 0.2\hoogte}
       \def\fbox##1{\vbox{\offinterlineskip
                    \hrule height\dikte
                    \hbox to \breedte{\vr\hfill##1\hfill\vr}
                    \hrule height\dikte}}
       \vbox\bgroup \offinterlineskip \tabskip=-\dikte \lineskip=-\dikte
            \halign\bgroup &\fbox{##\unskip}\unskip  \crcr }
       {\egroup\egroup\endgroup}
\def\Youngdiagram#1{\relax\ifmmode\vcenter{\,\begin{Young}#1\end{Young}\,}\else%
              $\vcenter{\,\begin{Young}#1\end{Young}\,}$\fi}

\begin{document}

\title[Irreducible restrictions of representations of symmetric groups]{{\bf Irreducible restrictions of representations of symmetric groups in small characteristics: Reduction
theorems}}

\author{\sc Alexander Kleshchev}
\address{Department of Mathematics\\ University of Oregon\\Eugene\\ OR 97403, USA}
\email{klesh@uoregon.edu}

\author{\sc Lucia Morotti}
\address
{Institut f\"{u}r Algebra, Zahlentheorie und Diskrete Mathematik\\ Leibniz Universit\"{a}t Hannover\\ 30167 Hannover\\ Germany} 
\email{morotti@math.uni-hannover.de}

\author{\sc Pham Huu Tiep}
\address
{Department of Mathematics\\ Rutgers University\\ Piscataway\\ NJ~08854, USA} 
\email{tiep@math.rutgers.edu}


\thanks{The first author was supported by the NSF grant DMS-1700905 and the DFG Mercator program through the University of Stuttgart. The second author was supported by the DFG grant MO 3377/1-1 and the DFG Mercator program through the University of Stuttgart. The third author was supported by the NSF grants DMS-1839351 and DMS-1840702.
This work was also supported by the NSF grant DMS-1440140 and Simons Foundation while all three authors were in residence at the MSRI during the Spring 2018 semester.}

\begin{abstract}
We study irreducible restrictions of modules over symmetric groups to subgroups. We get reduction results which substantially restrict the classes of subgroups and modules for which this is possible. Such results are known when the characteristic of the ground field is greater than $3$, but the small characteristics cases require a substantially more delicate analysis and new ideas. This work fits into the Aschbacher-Scott program on maximal subgroups of finite classical groups. 
\end{abstract}


\maketitle

\section{Introduction}
Let $\F$ be an algebraically closed field of characteristic $p\geq 0$, and $H$ be a finite almost quasi-simple group. 
This paper is a contribution to the following

\vspace{2mm}
\noindent
{\bf Irreducible Restriction Problem.}
{\em
Classify the subgroups $G <H$ and $\F H$-modules $V$ of dimension greater than $1$ such that the restriction $V\da_G$ is irreducible.
}
\vspace{2mm}

A major application of the Irreducible Restriction Problem is to the Aschbacher-Scott program on maximal subgroups of finite classical groups, see \cite{Asch,Scott,Magaard,KlL,BDR} for more details on this.  We point out that for the purposes of the applications to the Aschbacher-Scott program we may assume that $G$ is  almost quasi-simple, but  we will not be making this additional assumption. 

Suppose now that $\soc(H/Z(H))=\AAA_n$. We assume that  $n\geq 8$ to avoid small special cases. Then $H$ is one of $\AAA_n,\SSS_n$ or their double covers. 
If $p=0$ and $H$  is a symmetric or alternating group, the Irreducible Restriction Problem has been solved by Saxl \cite{Saxl}. If $p=0$ and $H$ is a double cover of symmetric or alternating groups, the problem was essentially solved by Kleidman and Wales \cite{KlW}. 

Let us assume from now on that $p>0$. We point out that it is the positive characteristic case which is important for the Aschbacher-Scott program. The positive characteristic analogues of the results of Saxl and Kleidman-Wales mentioned in the previous paragraph are currently available only for $p>3$, see  \cite{BK} for symmetric groups, \cite{KSAlt} for the alternating groups, and \cite{KT} for the double covers. It is very important to extend the classification to the case of characteristics $2$ and $3$. 

However, there are formidable technical obstacles which make the small characteristic cases much more complicated. The most  serious difficulty is that the submodule structure of certain important permutation modules over symmetric groups gets very complicated for $p=2$ and $3$. This in turn necessitates a rather  detailed study of branching for symmetric groups. 

The main result of this paper extends reduction theorems obtained in \cite{KS2Tran} and \cite{BK} and strengthens the main results of \cite{KST}. These reduction theorems were crucial for the eventual resolution of the Irreducible Restriction Problem for the cases $p>3$, and their small characteristic analogues will also play a key role in our future work \cite{KMT}. 

To formulate our main result we recall that the irreducible $\F \SSS_n$-modules are labeled by the $p$-regular partitions of $n$. If $\la$ is such a partition, we denote by $D^\la$ the corresponding irreducible $\F\SSS_n$-module, and define $\la^\Mull$ from $D^{\la^\Mull}\cong D^\la\otimes \sgn$. If 
$\la=(\la_1\geq \la_2\geq \dots\geq \la_k>0)$, we write $h(\la)$ for $k$. It is known that $D^\la\da_{\SSS_{n-1}}$ is  irreducible  if and only if $\la$ is in the explicitly defined class of {\em Jantzen-Seitz} (or {\em JS}) partitions which go back to \cite{JS,k2}. There is a special irreducible $\F\SSS_n$-module in characteristic $2$ called the {\em basic spin module $D^{\be_n}$}. Finally, recall that a subgroup of $\SSS_n$ is called {\em $k$-transitive} (resp. {\em $k$-homogeneous}) if  it acts transitively on the set of all ordered (resp. unordered) $k$-tuples of different elements in $\{1,2,\dots,n\}$. We refer the reader to the main body of the paper for more details on all of this. 

It is convenient to formulate our main result for all characteristics, although it is only new for $p=2$ and $3$:

\begin{MainTheorem}\label{TA}
Let $n\geq 8$ and $D^\la$ be an irreducible representation of $\F\SSS_n$ with $\dim D^\la>1$. If $G\leq \SSS_n$ is a subgroup such that the restriction $D^\la\da_G$ is irreducible, then one of the following holds:
\begin{enumerate}
\item[{\rm (i)}] $G$ is $3$-homogeneous.
\item[{\rm (ii)}] $G$ is $2$-transitive and $\min(h(\la),h(\la^\Mull))=2$;

\item[{\rm (iii)}] $G\leq \SSS_{n-1}$ and $\la$ is JS;

\item[{\rm (iv)}] $p=2$, $n$ is even, $G$ is $2$-transitive, $h(\la)\geq 3$ and there exists $1\leq j\leq h(\la)$ with $\la_j=\la_{j+1}+2$ and 
\[\la_1\equiv\ldots\equiv\la_{j-1}\not\equiv\la_j\equiv\la_{j+1}\not\equiv\la_{j+2}\equiv\ldots\equiv\la_{h(\la)}\pmod{2}\]

\item[{\rm (v)}] $p=2$, $n\equiv 2\pmod{4}$, $\la=(n-1,1)$, $G\leq \SSS_{n/2}\wr\SSS_2$ and $G\not\leq\SSS_{n/2}\times\SSS_{n/2}$. 

\item[{\rm (vi)}] $p=2$ and $D^\la$ is the basic spin module. 
\end{enumerate}
\end{MainTheorem}

In case (v) of Theorem~\ref{TA}, we have a complete classification of subgroups giving irreducible restrictions (see Example \ref{imprim} for some examples of such subgroups $G$):

\begin{MainTheorem}\label{TNat}
Let $6 \leq n \equiv 2 (\bmod\ 4)$, $p=2$, and let $G \leq W:=\SSS_{n/2} \wr \SSS_2$. Then $D^{(n-1,1)}\da_G$ is irreducible if and only 
if both of the following two conditions hold.
\begin{enumerate}[\rm(i)]
\item $G$ is transitive on $\{1,2 \ldots,n\}$.
\item If $B = \SSS_{n/2} \times \SSS_{n/2}$ is the base subgroup of $W$, then the projection of $G \cap B$ onto each factor 
$\SSS_{n/2}$ of $B$ induces a $2$-transitive subgroup of $\SSS_{n/2}$ over which $D^{(n/2-1,1)}$ is irreducible, and the restrictions
of the two modules $D^{(n/2-1,1)} \boxtimes \bone_{\SSS_{n/2}}$ and $\bone_{\SSS_{n/2}} \boxtimes D^{(n/2-1,1)}$
to $G \cap B$ are non-isomorphic.
\end{enumerate}
\end{MainTheorem}

In case (vi) of Theorem~\ref{TA}, we can also say much more:

\begin{MainTheorem}\label{TB}
Let $n\geq 5$, $p=2$, $D^{\be_n}$ be the irreducible basic spin module over $\SSS_n$,  and $G< \SSS_n$ be a subgroup of $\SSS_n$ such that $D^{\be_n}\da_G$ is irreducible. Then one of the following happens:

\begin{enumerate}
\item[{\rm (i)}]
$G\leq \SSS_a\wr\SSS_b$ with $n=ab$, $a,b\in\Z_{>1}$ and $a$ is odd. Moreover if $b>2$ then $G\not\leq\SSS_{a}\times\dots\times \SSS_{a}$. In fact,   
$$D^{\be_n}\da_{\SSS_a\wr\SSS_b}\cong D^{\be_a}\wr D^{\be_b}
$$
is indeed irreducible. 

\item[{\rm (ii)}]
$G\leq \SSS_{n-k}\times\SSS_{k}$ with $n-k$ and $k$ odd. In fact, 
$$D^{\be_n}\da_{\SSS_{n-k}\times\SSS_{k}}\cong D^{\be_{n-k}}\boxtimes D^{\be_{k}}$$
is indeed irreducible. 

\item[{\rm (iii)}]
$G$ is primitive, in which case $D^{\be_n}\da_G$ is irreducible if and only if one of the following happens:

\begin{enumerate}
\item[{\rm (a)}] $n\equiv 2\pmod{4}$ and $G=\AAA_n$;

\item[{\rm (b)}] $n=5$, $G=\CCC_5\rtimes \CCC_4$;

\item[{\rm (c)}] $n=6$, $G=\SSS_5$;

\item[{\rm (d)}] $n=10$, $G=\SSS_6$, $\MMM_{10}$ or $\Aut(\AAA_6)$;

\item[{\rm (e)}]  $n=12$, $G=\MMM_{12}$.
\end{enumerate}

\end{enumerate}
\end{MainTheorem}

We give some additional comments on the statements of our main results. First of all, taking into account Theorems~\ref{TNat} and \ref{TB}, let us exclude the cases of the natural and basic spin modules for $p=2$ as appear in parts (v) and (vi) of Theorem~\ref{TA}. Then, we obtain the 
statement that the restriction $D^\la\da_G$ is irreducible only if either (A) $G\leq \SSS_{n-1}$ or (B) $G$ is $2$-transitive.

In case (A), the restriction $D^\la\da_{\SSS_{n-1}}$ must be irreducible, so $\la$ must be JS. Moreover, then $D^\la\da_{\SSS_{n-1}}\cong D^\mu$ for the partition $\mu$ of $n-1$ which is obtained from $\la$ by removing the top removable node. So in this case one can proceed by induction on $n$. 

In case (B), one can use the classification of doubly transitive permutation groups \cite{Cam,Kantor}. In fact, parts (ii) and (iv) of Theorem~\ref{TA} often allow us to assume that $G$ is even $3$ homogeneous, and there are very few such permutation groups. The exceptional cases are mostly related to $2$-row partitions. For example, the exceptions in case (ii) correspond to the cases where either $\la$ is a $2$-row partition or $D^\la\otimes \sgn$ corresponds to a $2$-row partition. 
In a forthcoming paper \cite{KMT} we will analyze case (B) further.

We now outline the proof of the main results and the contents of the paper. Section~\ref{SPrel} is preliminary. In particular, in \S\ref{SSPar} we discuss combinatorics of good and normal nodes which will be crucial for branching results obtained later. In \S\ref{SSIrr}, we discuss irreducible $\F\SSS_n$-modules, and obtain in Lemma~\ref{LBasic2} our main general tool for proving  reducibility of $D^\la\da_G$. Basic facts on Specht, Young and permutation modules are discussed in \S\ref{SSMoreMod}. The information on the $G$-invariant spaces in some dual Specht modules is obtained in \S\ref{SSInv}. 

Section~\ref{SBr} is on branching. After recording the basic branching rules in \S\ref{SSGenBrR}, we study in \S\ref{SSFilt} some important filtrations that arise in the restriction $D^\la\da_{\SSS_{n-1}}$. The technical \S\ref{SSJSBr} is devoted to the study of restrictions of JS modules in characteristic $2$ to the natural subgroups $\SSS_{n-k}$. In \S\ref{SSRecognition} we  obtain characterizations of certain classes of irreducible modules via their branching properties. 

Section~\ref{SPerm} is on the submodule structure of the permutation modules $M_k=M^{(n-k,k)}$ in characteristics $2$ and $3$ for $k=1,2,3$. Section~\ref{SP2NEVEN} is on the submodule structure of the module $\EE(\la):=\End_F(D^\la)\cong D^\la\otimes D^\la$. We show that some quotients of the permutation modules $M_k$ for $k=1,2,3$ arise as submodules of $\EE(\la)$. 
Section~\ref{SSpecialHoms} gives an alternative way of constructing interesting homomorphisms from $M_k$ to $\EE(\la)$, which develops the ideas of \cite[Theorem 3.3]{KS2Tran} and \cite[\S3]{BK}. Finally, in Section~\ref{SMain} we establish the main results. 

\subsection*{Acknowledgement} We are grateful to the anonymous referee and Gunter Malle for careful reading of the paper and multiple useful remarks. 

\section{Preliminaries}\label{SPrel}

\subsection{Groups and modules}
Throughout the paper we work over a fixed algebraically closed ground field $\F$ of characteristic $p>0$. We do not yet assume that $p=2$ or $3$ but will do this when necessary. 

For a finite group $G$, we denote by $\mod{\F G}$ the category of finite dimensional $\F G$-modules. For $U,V\in\mod{\F G}$ we denote by $\Hom_G(U,V)$ the space of all $\F G$-module homomorphisms from $U$ to $V$, and by $\Hom_\F(U,V)$ the space of all linear maps considered as an $\F G$-module via $(g\cdot f)(u)=gf(g^{-1}u)$ for all $f\in \Hom_\F(U,V)$, $u\in U$ and $g\in G$.

We denote by $\bone_G$  the trivial $\F G$-module. Let $G$ be a subgroup of a group $H$, $V$ be an $\F H$-module and $W$ be an $\F G$-module. We denote by $V{\da}_G$ or $V{\da}^H_G$ the {\em restriction} of $V$ from $H$ to $G$, and by $W{\ua}^H$ or $W{\ua}^H_G$ the {\em induction} of $W$ from $G$ to $H$. As a special case, for a subgroup $G\leq \SSS_n$, we will often be using the permutation  module 
\begin{equation}\label{EIG}
\I(G):=\bone_G{\ua}^{\SSS_n}.
\end{equation}
If 
$$\SSS_\mu:= \SSS_{\mu_1}\times\dots\times \SSS_{\mu_a}\leq \SSS_n$$ is a Young subgroup corresponding to a composition $\mu=(\mu_1,\dots,\mu_a)$ of $n$, then we write $M^\mu$ instead of $\I(\SSS_\mu)$. 

Let $V$ be an $\F G$-module. We denote by $V^G$ the set of {\em $G$-invariant vectors} in $V$. We write $\soc V$ and $\head V$ for the socle and head of $V$, respectively.

If $L_1,\dots,L_a$ are irreducible $\F G$-modules, we denote by $L_1|\cdots| L_a$ a {\em uniserial} $\F G$-module with composition factors $L_1,\dots,L_a$ listed from socle to head. If $V$ is an $\F G$-module, we use the notation 
$$
V\cong L_1|\cdots| L_a\ \oplus\,\cdots\, \oplus\ K_1|\cdots| K_b
$$
to indicate that $V$ is (isomorphic to) a direct sum of the uniserial modules $L_1|\cdots| L_a$, \dots, $K_1|\cdots| K_b$. On the other hand, if $V_1,\dots,V_a$ are any $\F G$-modules, we write 
$$
V\sim V_1|\dots|V_a
$$
to indicate that $V$ has a filtration with subquotients $V_1,\dots,V_a$ listed from bottom to top. 
We use the notation 
$$
V\sim V_1|\cdots| V_a\ \oplus\,\cdots\, \oplus\ W_1|\cdots| W_b
$$
to indicate that $V\cong X\oplus\dots\oplus Y$ for $X\sim L_1|\cdots| L_a$, \dots, $Y\sim K_1|\cdots| K_b$.

\begin{Lemma} \label{L200418} 
Let $L$ be an irreducible $\F G$-module, and $M$ be an $\F G$-module with submodules $X\subseteq Y\subseteq M$ such that $\Hom_G(L,Y)=0$ and $\soc(M/X)\subseteq Y/X$. Then $\Hom_G(L,M)=0$. 
\end{Lemma}
\begin{proof}
If $\psi:L\to M$ is a non-zero homomorphism, then $\psi(L)$ is simple and $\psi(L) \not\subseteq Y$. In particular, $\psi(L) \not\subseteq X$, so $(\psi(L)+X)/X$ is a simple submodule of $M/X$ and so $\psi(L)+X \subseteq Y$, a contradiction.
\end{proof}

\subsection{Partitions}\label{SSPar}
We denote by $\Par(n)$ the set of all {\em partitions} of $n$ and by $\Parp(n)$ the set of all {\em $p$-regular} partitions of $n$, see \cite[10.1]{JamesBook}. We identify a partition $\la=(\la_1,\la_2,\dots)$ with its {\em Young diagram} 
$
\{(r,s)\in\Z_{>0}\times\Z_{>0}\mid s\leq\la_r\}.
$
We have a {\em dominance order} $\unrhd$ on partitions, see \cite[3.2]{JamesBook}. 
The number of non-zero parts of a partition $\la$ is denoted by $h(\la)$. 
The following $2$-row partitions will play a special role in this paper:
\begin{eqnarray}\label{ENat}
\al_n&:=&(n-1,1)
\\
\label{ESpin}
\be_n&:=&
\left\{
\begin{array}{ll}
(n/2+1,n/2-1) &\hbox{if $n$ is even,}\\
((n+1)/2,(n-1)/2) &\hbox{if $n$ is odd.}
\end{array}
\right.
\end{eqnarray}

We set
$$
I:=\Z/p\Z
$$
identified with $\{0,1,\dots,p-1\}$. 
Given a node $A=(r,s)$ in row $r$ and column $s$, we consider its {\em residue}
$$
\res A:=s-r\pmod{p}\in I.
$$
The {\em residue content} of a partition $\la$ is the tuple 
$$\operatorname{cont}(\la):=(a_i)_{i\in I}$$ such that $\la$ has exactly $a_i$ nodes of residue $i$ for each $i\in I$. For $j\in I$, let $\ga_j$ be the tuple $(a_i)_{i\in I}$ with $a_i=\de_{i,j}$. We consider the tuples $(a_i)_{i\in I}$ as elements of $\Theta:=\sum_{i\in I}\Z\cdot\ga_i$, the free $\Z$-module with basis $\{\ga_i\mid i\in I\}$. Let 
\begin{equation}\label{ETheta}
\Theta_n:=\left\{\theta=\sum_{i\in I}a_i\ga_i\in\Theta\ \bigl{|}\ a_i\geq 0,\ \sum_{i\in I} a_i=n\right\}.
\end{equation}
Partitions $\la,\mu\in\Par(n)$ have the same residue contents if and only if they have the same $p$-cores, see \cite[2.7.41]{JK}. 

Let $i \in I$ and $\la\in\Par(n)$. 
A node $A 
\in \la$ (resp. $B\not\in\la$) is called {\em $i$-removable} (resp. {\em $i$-addable}) for $\la$ 
if $\res A=i$ and $\la_A:=\la\setminus\{A\}$ (resp. $\la^B:=\la\cup\{B\}$) is a Young diagram of a partition. 
A node is called {\em removable} (resp. {\em addable}) if it is $i$-removable (resp. $i$-addable) for some $i$. Labeling the $i$-addable
nodes of $\la$ by $+$ and the $i$-removable nodes of $\la$ by $-$, the {\em $i$-signature} of 
$\la$ is the sequence of pluses and minuses obtained by going along the 
rim of the Young diagram from bottom left to top right and reading off
all the signs.
The {\em reduced $i$-signature} of $\la$ is obtained 
from the $i$-signature
by successively erasing all neighbouring 
pairs of the form $-+$. 
The nodes corresponding to  $-$'s (resp. $+$'s) in the reduced $i$-signature are
called {\em $i$-normal} (resp. {\em $i$-conormal}) for $\la$. There are equivalent definition of normal (resp. conormal) nodes involving the $i$-removable and $i$-addable nodes above (resp. below) a given node; for example an $i$-removable node $A$ is normal if and only if for any $i$-addable node $B$ above $A$ there exists an $i$-removable node $C_B$ between $A$ and $B$ with the property that if $B_1$ and $B_2$ are distinct $i$-addable nodes above $A$ then $C_{B_1}\not=C_{B_2}$.

The leftmost $i$-normal (resp. rightmost $i$-conormal) node is called {\em $i$-good} (resp. {\em $i$-cogood}) for $\la$. A node is called {\em normal} (resp. {\em conormal}, {\em good}, {\em cogood}) if it is $i$-normal (resp. $i$-conormal, $i$-good, $i$-cogood) for some $i$.
We denote 
\begin{align*}
\eps_i(\la)&:=\sharp\{\text{$i$-normal nodes of $\la$}\},
\\
\phi_i(\la)&:=\sharp\{\text{$i$-conormal nodes of $\la$}\}.
\end{align*}
There exists an $i$-good (resp. $i$-cogood) node for $\la$ if and only if $\eps_i(\la)>0$ (resp. $\phi_i(\la)>0$).  

Let $\la\in\Parp(n)$. If $\eps_i(\la)>0$, we denote by $A$ the $i$-good node of $\la$ and set
$$
\tilde e_i \la:=\la_A.
$$
If $\phi_i(\la)>0$, we denote by $B$ the $i$-cogood node for $\la$ and set
$$
\tilde f_i\la:=\la^B.
$$
We will repeatedly use the known fact that $\tilde e_i \la$ and $\tilde f_i\la$ are $p$-regular, whenever $\la$ is so. 
The following three known statements follow easily from the definitions:

\begin{Lemma}\label{L2.8} {\cite[Lemma 2.8]{M}}
Any partition has one more conormal node than it has normal nodes.
\end{Lemma}

\begin{Lemma}\label{l2}
Let $\lambda\in \Par(n)$ and $i\in I$. Assume that $A$ is $i$-normal and $B$ is $i$-conormal for $\la$. Then $B$ is conormal  for $\la_A$. 
\end{Lemma}

\begin{proof}
Notice first that the set of $i$-removable and $i$-addable nodes of $\la$ is equal to the set of $i$-removable and $i$-addable nodes of $\la_A$. We can obtain the reduced $i$-signature of $\la_A$ as follows: start by deleting a sequence of pairs $-+$ which is deleted from the $i$-signature of $\la$ to obtain the reduced $i$-signature of $\la$. The reduced $i$-signature of $\la$ and the partly reduced $i$-signature of $\la_A$ look as follows:
\[\begin{array}{lcccccc}
&&B&&&A\\
\la:&+\cdots +&+&+\cdots +&-\cdots -&-&-\cdots-,\\
\la_A:&+\cdots +&+&+\cdots +&-\cdots -&+&-\cdots-.
\end{array}\]
It is then easy to see that $B$ is conormal in $\la_A$.
\end{proof}

\begin{Lemma} \label{LGoodCogood} 
Let $i\in I$ and $\la\in\Parp(n)$. 
\begin{enumerate}
\item[{\rm (i)}] If $\eps_i(\la)>0$ then $\phi_i(\tilde e_i\la)>0$ and $\tilde f_i\tilde e_i\la=\la$. 
\item[{\rm (ii)}] If $\phi_i(\la)>0$ then $\eps_i(\tilde f_i\la)>0$ and $\tilde e_i\tilde f_i\la=\la$. 
\end{enumerate}
\end{Lemma}

We will need more results on combinatorics of normal nodes. 

\begin{Lemma}\label{l19}
Let $\la\in\Parp(n)$ and $i\in I$ with $\vare_i(\la),\phi_i(\la)>0$. Let $B=(a,b)$ and $C=(c,d)$ be the $i$-good and $i$-cogood nodes of $\la$, respectively. Then $(\la_B)^C$ is $p$-singular if and only if $c=a+p-1$ and $d=b-1$.
\end{Lemma}

\begin{proof}
Notice that $a<c$ and that
\begin{align*}
\la_B&=(\la_1,\ldots,\la_{a-1},\la_{a}-1,\la_{a+1},\ldots),\\
\la^C&=(\la_1,\ldots,\la_{c-1},\la_{c}+1,\la_{c+1},\ldots),\\
(\la_B)^C&=(\la_1,\ldots,\la_{a-1},\la_{a}-1,\la_{a+1},\ldots,\la_{c-1},\la_{c}+1,\la_{c+1},\ldots).
\end{align*}
Since $\la_B$ and $\la^C$ are $p$-regular, we have that $(\la_B)^C$ is $p$-singular if and only if $c=a+p-1$ and $b-1=\la_{a}-1=\la_{c}+1=d$.
\end{proof}

\begin{Lemma} \label{L6.1} {\rm \cite[Lemma 6.1]{M}} 
Let $p=2$ and $\lambda\in\Par_2(n)$ satisfy $\eps_0(\lambda)+\eps_1(\lambda)=2$. For $1\leq k\leq h(\lambda)$ let $a_k$ be the residue of the removable node in the $k$-th row of $\lambda$. Further let $1<b_1<\ldots<b_t\leq h(\lambda)$ be the set of indices $k$ for which $a_k=a_{k-1}$. Then the normal nodes of $\lambda$ are in rows 1 and $b_1$, while the conormal nodes of $\lambda$ are in rows $b_t-1$, $h(\lambda)$ and $h(\lambda)+1$. Further $a_{b_k}\not=a_{b_{k-1}}$ for all $1<k\leq t$.
\end{Lemma}

\begin{Lemma}\label{L6.2}
Let $p=2$, $\lambda\in\Par_2(n)$ and $i\in I$.
If $\eps_i(\lambda)=2$, $\eps_{1-i}(\lambda)=0$ and $\phi_i(\lambda)=0$ then $n$ is odd. 
\end{Lemma}
\begin{proof}
This follows from Lemma~\ref{L2.8} and \cite[Lemma 6.2]{M}.
\end{proof}

\begin{Lemma}\label{l3}
Let $p=2$, $\la\in\Par_2(n)$ and let $i$ be the residue of the bottom normal node of $\la$. If $\vare_0(\la)=\vare_1(\la)=1$ and $\phi_i(\la)=3$ then $n$ is odd. 
\end{Lemma}

\begin{proof}
Let $j:=1-i\in I$. For $1\leq k\leq h(\la)$ let $a_k$ be the residue of the removable node in the $k$-th row of $\la$. Also let $1<b_1<\ldots<b_t\leq h(\la)$ be the set of indices $k$ for which $a_k=a_{k-1}$. 
The top removable node is always normal, so it must have residue $j$. Moreover, by Lemma~\ref{L6.1}, the removable node in row $b_1$ is $i$-normal, and the conormal nodes for $\la$ are the addable nodes on rows $b_t-1$, $h(\la)$ and $h(\la)+1$. As $\phi_i(\la)=3$ it follows that $a_{b_t-1}=a_{h(\la)}=j$ and $h(\la)\equiv  i\pmod{2}$.

Notice that by definition of $b_1$, the residues $a_k$ alternate for $1\leq k\leq b_1-1$. Also we have that $a_1=j\not=i=a_{b_1}=a_{b_1-1}$. So
\[\la_1\equiv\ldots\equiv\la_{b_1-1}\pmod{2}\]
and $b_1-1$ is even.

For $1\leq m< t$ we similarly have that the residues $a_k$ alternate for $b_m\leq k\leq  b_{m+1}-1$ and by Lemma~\ref{L6.1}, we have  $a_{b_m}\not=a_{b_{m+1}-1}$, so that
\[\la_{b_m}\equiv \ldots\equiv \la_{b_{m+1}-1}\pmod{2}\]
and $b_{m+1}-b_m$ is even.

Further, the residues $a_k$ alternate for $b_t\leq k\leq h(\la)$ and $a_{b_t}=a_{b_{t-1}}=a_{h(\la)}$ by the first paragraph, so
\[\la_{b_t}\equiv \ldots\equiv \la_{h(\la)}\pmod{2}\]
and $h(\la)-b_t+1$ is odd.

It follows that 
$$h(\la)=(b_1-1)+\sum_{m=1}^{t-1}(b_{m+1}-b_m)+(h(\la)-b_t+1)$$
 is odd and then $i=1$, $a_{h(\la)}=0$ by the first paragraph. Hence $\la_{h(\la)}$ is odd. So
\[n\equiv \la_1\cdot(b_1-1)+\sum_{m=1}^{t-1}\la_{b_{m}}\cdot(b_{m+1}-b_m)+\la_{h(\la)}\cdot(h(\la)-b_t+1) \pmod{2},
\]
and we deduce that $n$ is odd.
\end{proof}

\begin{Lemma}\label{l22}
Let $p=2$, $i\in I$, and $\la\in\Par_2(n)$ satisfy $\vare_0(\la)+\vare_1(\la)=2$. Assume that $\vare_i(\la),\phi_i(\la)>0$, and let $B$ be the $i$-good and $C$ be the $i$-cogood nodes of $\la$, respectively. If $(\la_B)^C$ is $2$-singular then one of the following holds: 
\begin{enumerate}
\item[{\rm (a)}]
$h(\la)\geq 3$ and there exists $1\leq j<h(\la)$ such that $\la_j=\la_{j+1}+2$ and 
\[\la_1\equiv\ldots\equiv\la_{j-1}\not\equiv\la_j\equiv\la_{j+1}\not\equiv\la_{j+2}\equiv\ldots\equiv\la_{h(\la)}\pmod{2}.\]

\item[{\rm (b)}]
$\la_1,\dots,\la_{h(\la)-1}$ are odd and $\la_{h(\la)}=2$.
\end{enumerate}
\end{Lemma}

\begin{proof}
By Lemma \ref{l19}, we can write $B=(a,b)$, $C=(a+1,b-1)$. Let $b_1,\dots, b_t$ be as in Lemma~\ref{L6.1}.

Assume first that $B$ is in the first row. Then $C$ is not in the first column, for otherwise $\la=(2)$ which contradicts the assumption $\eps_0(\la)+\eps_1(\la)=2$. If $C$ is in the last row of $\la$ then $h(\la)=2$ and $\la$ is a JS-partition, which again contradicts the assumption $\eps_0(\la)+\eps_1(\la)=2$. 
So we may now assume that $h(\la)\geq 3$ and 
by Lemma~\ref{L6.1} we are in case (1) for $j=1$. 

Assume now that $B=(j,\la_j)$ with $2\leq j\leq h(\la)$. Since $B$ is normal in $\la$ we have by Lemma~\ref{L6.1} that $b_1=j$ and then 
\[\la_1\equiv\ldots\equiv\la_{j-1}\not\equiv\la_j\pmod{2}.\]
If $j=h(\la)$ then we are in case (b).
If $j=h(\la)-1$, then we are in case (a). Finally, if $2\leq j< h(\la)-1$   then by Lemma~\ref{L6.1} we have 
\[\la_{j+1}\not\equiv\la_{j+2}\equiv\ldots\equiv\la_{h(\la)}\pmod{2}.\]
So we are in case (a).
\end{proof}

We now define the Mullineux bijection referring the reader to \cite{FK,BO} for more details. 
Let $\la\in\Parp(n)$. The {\em rim} of $\la$ is defined to be the set of all nodes $(r, s) \in\la$ such that $(r + 1, s + 1)\not\in \la$. The {\em $p$-rim} of $\la$ is the union of the {\em $p$-segments}  which are defined as follows. The first $p$-segment is the first $p$ nodes of the rim, reading from top-right to bottom-left. The next $p$-segment is then obtained by reading off the next $p$ nodes of the rim, but starting from the row immediately below the last node of the first $p$-segment. The remaining p-segments are obtained by repeating this process. All but the last $p$-segment contain exactly $p$ nodes, while the last may contain less. Set $\la^{(1)}=\la$, and define $\la^{(t)}$ to be $\la^{(t-1)}\setminus\{\text{the $p$-rim of $\la^{(t-1)}$}\}$. Let $m$ be the largest number such that $\la^{(m)}\neq \varnothing$. The {\em Mullineux symbol} of $\la$ is defined to be the array
$$G(\la):=
\left(
\begin{matrix}
a_1 & a_2 &\dots & a_m  \\
r_1 & r_2 &\dots & r_m
\end{matrix}
\right),
$$
where $a_t$ is the number of the nodes of the $p$-rim of $\la^{(t)}$ and $r_t:=h(\la^{(t)})$. The $t$th column 
$\left(
\begin{matrix}
a_t   \\
r_t 
\end{matrix}
\right)$ of $G(\la)$ is denoted $G_t(\la)$.  The partition can be uniquely reconstructed from its Mullineux symbol. 
The {\em Mullineux bijection} $\la\mapsto \la^\Mull$ on $\Par_p(n)$ is defined from 
$$G(\la^\Mull):=
\left(
\begin{matrix}
a_1 & a_2 &\dots & a_m  \\
a_1+x_1-r_1 & a_2+x_2-r_2 &\dots & a_m+x_m-r_m
\end{matrix}
\right),
$$
where $x_t:=0$ if $p\mid a_t$ and $x_t:=1$ otherwise.

\subsection{Irreducible modules over symmetric groups}
\label{SSIrr}
We use James' notation 
$$\{D^\la\mid \la\in\Par_p(n)\}$$ for the set of the irreducible $\F\SSS_n$-modules up to isomorphism, see \cite[\S11]{JamesBook}. 
For example, $D^{(n)}\cong\bone_{\SSS_n}$.
By \cite[11.5]{JamesBook}, we have $(D^\la)^*\cong D^\la$ for all $\la\in\Parp(n)$. We denote by $\sgn$ the sign module over $\SSS_n$. Then by \cite{FK} (see also \cite{BO}), we have
$$D^\la \otimes \sgn \cong D^{\la^\Mull}.$$ 

\begin{Lemma} \label{LPlus} {\rm \cite[Lemma 1.11]{BK}} 
If $\la\in\Parp(n)$ and $\mu\in\Parp(m)$ then $D^\la\boxtimes D^\mu$ is a composition factor of $D^{\la+\mu}\da_{\SSS_{n,m}}$, where $\la+\mu$ is the partition $(\la_1+\mu_1,\la_2+\mu_2,\dots)$.  
\end{Lemma}

Recalling \eqref{ENat}, $D^{\al_n}$ is the {\em heart of the natural module} of dimension $n-1-\de_{p|n}$, where we have put $\de_{p|n}:=1$ if $p\mid n$ and $\de_{p|n}:=0$ otherwise. 
Recalling \eqref{ESpin}, $D^{\be_n}$ is the {\em basic spin module} if $p=2$. It often plays a special role as indicated for example by the following result:

\begin{Proposition}\label{p8}
Let $\la\in\Par_2(n)$ with $\dim D^\la>1$. If $2\leq k\leq n/2$ then  $D^\la\da_{\SSS_{n-k,k}}$ is irreducible if and only if $p=2$, $n$ is even, $k$ is odd and $\la=\be_n$. In the exceptional case, we have 
$D^{\be_n}\da_{\SSS_{n-k,k}}\cong D^{\be_{n-k}}\boxtimes D^{\be_k}.$
\end{Proposition}

\begin{proof}
By \cite[Theorem 5.1]{JS} and \cite[Theorem 10]{P}, $D^\la\da_{\SSS_{n-k,k}}$ is irreducible if and only if $p=2$, $n$ is even, $k$ is odd and $\la=\be_n$. The second statement then follows for example from Lemma~\ref{LPlus}.
\end{proof}

For $\la\in\Parp(n)$, we consider the $\F  \SSS_n$-module 
\begin{equation}\label{EE}
\EE(\la):=\End_\F (D^\la). 
\end{equation}
Recall the notation $\I(G)$ from \eqref{EIG}. A fundamental trick that will be used to prove that $D^\la\da_{G}$ is reducible for a subgroup $G<\SSS_n$, is as  follows:

\begin{Lemma} \label{LBasic} 
Let $\la\in\Parp(n)$, and $G\leq \SSS_n$ be a subgroup such that 
$$
\dim\Hom_{\SSS_n}(\I(G),\EE(\la))>1.
$$
Then $D^\la{\da}_G$ is reducible. 
\end{Lemma}
\begin{proof}
This follows from 
$$
\Hom_{\SSS_n}(\I(G),\EE(\la))=\Hom_{\SSS_n}(\bone_G{\ua}^{\SSS_n},\EE(\la))\cong \Hom_{G}(\bone_G,\EE(\la)\da_G)\cong\End_G(D^\la{\da}_G)
$$
and Schur's lemma. 
\end{proof}

\begin{Lemma} \label{LBasic2} 
Let $\la\in\Parp(n)$, and $G\leq \SSS_n$ be a subgroup such that there exists $\psi:\I(G)\to \EE(\la)$ with $\psi$ non-zero and such that $\im\psi\not\cong\bone_{\SSS_n}$. Then $D^\la{\da}_G$ is reducible. 
\end{Lemma}
\begin{proof}
This follows from Lemma \ref{LBasic}, since there always exists a homomorphism $\phi:\I(G)\to\EE(\la)$ with image $\bone_{\SSS_n}$, and so $\phi$ and $\psi$ are linearly independent. 
\end{proof}

We will need one more general result on reducibility of $D^\la\da_G$:

\begin{Lemma} \label{LInvOp} 
Let $n\geq 5$, $H=\SSS_n$ or $\AAA_n$, $L$ be an irreducible 
$\F H$-module of dimension greater than $1$, 
and $G\leq H$ be a subgroup with $O_p(G)\neq 1$. Then $L\da_G$ is reducible. 
\end{Lemma}
\begin{proof}
The assumptions $n\geq 5$ and $\dim L>1$ guarantee that $L$ is faithful. Hence the invariants $L^{O_p(G)}$ form a non-trivial proper submodule of $L\da_G$.
\end{proof}

\subsection{More modules over symmetric groups} 
\label{SSMoreMod}
As in \cite[\S4]{JamesBook}, we have {\em Specht modules}  $S^\la$ and permutation modules $M^\la$ over $\SSS_n$ for all $\la\in\Par(n)$. The module $M^\la$ is the permutation module on the set of $\la$-tabloids $\{t\}$, which are row-equivalence classes of $\la$-tableaux $t$, while $S^\la\subseteq M^\la$ is spanned by the polytabloids 
\begin{equation}\label{EPolyTab}
e_t:=\sum_{\si\in C_t}(\sgn\, \si)\, \si\cdot\{t\}\in M^\la, 
\end{equation}
where $C_t$ denotes the column stabilizer of  the $\la$-tableau $t$. In fact, any $e_t$ generates $S^\la$ as an $\F\SSS_n$-module. 
It is well-known that $(M^\la)^*\cong M^\la$. 

We will also use {\em Young modules $Y^\la$} which can be defined using the following well-known facts  contained for example in \cite{JamesArcata} and \cite[\S4.6]{Martin}:

\begin{Lemma} \label{LYoung} 
There exist indecomposable $\F \SSS_n$-modules $\{Y^\la\mid \la\in\Par(n)\}$ such that $M^\la\cong Y^\la\,\oplus\, \bigoplus_{\mu\rhd\la}(Y^\mu)^{\oplus m_{\mu,\la}}$ for some $m_{\mu,\la}\in\Z_{\geq 0}$. Moreover, $Y^\la$ can be characterized as the unique indecomposable direct summand of $M^\la$ such that  $S^\la\subseteq Y^\la$. Finally, we have $(Y^\la)^*\cong Y^\la$ for all $\la\in\Par(n)$.
\end{Lemma}

\begin{Lemma} \label{LNak} {\rm \cite[6.1.21]{JK}} 
The irreducible $\F\SSS_n$-modules $D^\la$ and $D^\mu$ are in the same block if and only if $\operatorname{cont}(\la)=\operatorname{cont}(\mu)$. All composition factors of $S^\nu$ and $Y^\nu$ are of the form $D^\kappa$ where $\operatorname{cont}(\kappa)=\operatorname{cont}(\nu)$. 
\end{Lemma}

In view of the lemma, blocks of $\F \SSS_n$ are determined by the residue contents of irreducible modules contained in the block, which are elements of $\Theta_n$, see (\ref{ETheta}). The block of $\F \SSS_n$ corresponding to $\theta\in \Theta_n$ will be denoted $B_\theta$. If $\theta\in\Theta_n$ does not arise as a residue content of any $\la\in \Par(n)$, we set $B_\theta:=0$, so that we have 
\begin{equation}\label{EBlocks}
\F\SSS_n=\bigoplus_{\theta\in\Theta_n}B_\theta.
\end{equation}

Two-row partitions will play a special role in this paper, so it is convenient to introduce the following notation. Let $0\leq k\leq n/2$. We denote  
$$
M_k:=M^{(n-k,k)},\quad S_k:=S^{(n-k,k)},\quad D_k:=D^{(n-k,k)}, \quad Y_k:=Y^{(n-k,k)}.
$$
Strictly speaking, when $p=2$ and $n$ is even, $D_k$ is only  defined if $k<n/2$. 
We denote by $\Om_k$ the set of all $k$-elements subsets of $\{1,\dots,n\}$ so that $M_k$ is the permutation module on $\Om_k$. 

For $0\leq k,l\leq n/2$, we will use special homomorphisms between permutation modules:
$$\eta_{k,l}:M_k\to M_l,\ X\mapsto \sum_{Y\in \Om_l,\,Y\, \text{incident to}\, X}Y,
$$
where $Y$ is incident to $X$ means $Y\subseteq X$ or $X\subseteq Y$.

\begin{Lemma} \label{LWilson} {\rm \cite[Theorem 1]{Wil}} 
If $0\leq k\leq l\leq n/2$ then 
$$\rank(\eta_{k,l})=\rank(\eta_{l,k})=\sum \binom{n}{r}-\binom{n}{r-1},$$
where the sum is over all $r=0,\dots,k$ such that $\binom{l-r}{k-r}$ is not divisible by $p$. 
\end{Lemma}

Let $0\leq k\leq n/2$, $G\leq \SSS_n$ and $\la\in\Parp(n)$. We denote by $i_k(G)$ the number of $G$-orbits on $\Om_k$. Note that 
\begin{equation}\label{EItoM}
i_k(G)=\dim M_k^G=\dim\Hom_{\SSS_n}(\I(G),M_k). 
\end{equation}
Define also 
\begin{equation}\label{EMtoE}
m_k(\la):=\dim\Hom_{\SSS_n}(M_k,\EE(\la))=\dim \End_{\SSS_{n-k,k}}(D^\la{\da}_{\SSS_{n-k,k}}). 
\end{equation}
Our main tools are Lemmas~\ref{LBasic} and \ref{LBasic2}, which motivate us to study homomorphisms from $\I(G)$ to $\EE(\la)$. We plan to do it by studying homomorphisms from $\I(G)$ to $M_k$ and then from $M_k$ to $\EE(\la)$ for appropriate small $k$'s. This is why we need dimensions defined in (\ref{EItoM}) and (\ref{EMtoE}). 

\begin{Lemma}\label{L4.14} {\rm \cite[Lemma 4.14]{M}} 
If $p=2$ and $V$ is an $\SSS_n$-module then
\[\dim\End_{\SSS_{n-2}}(V\downarrow_{\SSS_{n-2}})\leq 2\dim\End_{\SSS_{n-2,2}}(V\downarrow_{\SSS_{n-2,2}}).\]
\end{Lemma}

\subsection{Invariants}
\label{SSInv}

In this section, for various transitive $G\leq \SSS_n$, we will study the invariants $(S_1^*)^G$ of the dual Specht module $S_1=S^{(n-1,1)}$. Our goal is to establish that $(S_1^*)^G=0$ in many situations. The following lemma will allow us to reduce to the case $p \mid n$.

\begin{Lemma} \label{LEasy} 
If $p \nmid n$ and $G\leq \SSS_n$ is transitive then $(S_1^*)^G=0$. 
\end{Lemma}
\begin{proof}
Since $G$ is transitive, we have $\dim M_1^G=1$, and the result follows since under the assumption $p \nmid n$ we have $M_1\cong \bone_{\SSS_n}\oplus S_1^*$. 
\end{proof}

If $p\mid n$ we can use the following criteria for $(S_1^*)^G=0$. 

\begin{Lemma} \label{L160817} 
If $G$ is a transitive subgroup of $\SSS_n$ with $G=O^p(G)$ then $(S_1^*)^G=0$. 
\end{Lemma}
\begin{proof}
Since $G$ is transitive, we have $\dim M_1^G=1$. Now the result follows by considering the long exact sequence in cohomology corresponding to the short exact sequence $0\to \bone_G\to M_1\to S_1^*\to 0$ and using $H^1(G,\bone_G)=0$, which comes from the assumption $G=O^p(G)$. 
\end{proof}

\begin{Corollary} \label{COP} 
Let $G$ be a subgroup of $\SSS_n$ such that $O^p(G)$ is transitive. Then $(S_1^*)^G =0$. 
\end{Corollary}
\begin{proof}
Since $O^p(O^p(G))=O^p(G)$, the previous lemma applies to show that $(S_1^*)^{O^p(G)}=0$, which implies the result.
\end{proof}

The following result shows that we can apply Corollary~\ref{COP} to primitive subgroups with non-abelian socle:

\begin{Lemma} \label{LOP} 
Let $G$ be a primitive subgroup of $\SSS_n$ with non-abelian socle $S$. Then $S$ and $O^p(G)$ are transitive. If, in addition, $G$ is $2$-transitive then either $S$ and $O^p(G)$ are $2$-transitive or $(n,G,S) = (28, SL_2(8).3,SL_2(8))$.
\end{Lemma}
\begin{proof}
Since $S$ is normal in $G$, then $G$ permutes the $S$-orbits on $\{1,2,\ldots,n\}$. 
But $G$ is primitive, so there is only one $S$-orbit. 
Further, by inspection of the list of $2$-transitive groups, see \cite[Note 2, p. 9]{Cam}, we see that if $G$ is $2$-transitive then either $S$ is $2$-transitive or $(n,G,S) = (28, SL_2(8).3,SL_2(8))$.

Finally, by the O'Nan-Scott Theorem, see e.g. \cite[Theorem 4.1]{Cam}, $S$ is a direct product of non-abelian simple groups. But 
$$S/(S\cap O^p(G))\cong O^p(G)S/O^p(G)\leq G/O^p(G)$$ 
is a $p$-group, so $O^p(G)\geq S$, and the statements on $O^p(G)$ also follow. 
\end{proof}

\begin{Corollary} \label{C2TrInv} 
If $G$ is a primitive subgroup of $\SSS_n$ with non-abelian socle then $(S_1^*)^G =0$. 
\end{Corollary}
\begin{proof}
Follows from Corollary~\ref{COP} and Lemma~\ref{LOP}. 
\end{proof}

For primitive subgroups with abelian socle we have:

\begin{Lemma} \label{LAbSoc} 
Let $G$ be a primitive subgroup of $\SSS_n$ with abelian socle. Then either $(S_1^*)^G=0$ or $O_p(G)\neq 1$. 
\end{Lemma}

\begin{proof}
By the O'Nan-Scott Theorem, 
$n=r^m$ and $S:=\soc G$ is an elementary abelian $r$-group of order $r^m$ for a prime $r$. If $r=p$ we have $O_p(G)\geq S\neq 1$. Otherwise $p\nmid n$, and we are done by Lemma~\ref{LEasy}.  
\end{proof}

The following result will allow us to assume that $(S_1^*)^G=0$ for primitive subgroups $G\leq \SSS_n$. 

\begin{Corollary} \label{CThree} 
Let $n\geq 5$, $G\leq \SSS_n$ be a primitive subgroup, and $D^\la$ be an irreducible $\F\SSS_n$-module of dimension greater than $1$. If $D^\la\da_G$ is irreducible then $(S_1^*)^G=0$. 
\end{Corollary}
\begin{proof}
This follows from Lemmas~\ref{LInvOp}, \ref{LAbSoc} and Corollary~\ref{C2TrInv}.
\end{proof}

For imprimitive subgroups we will be using the following lemma:

\begin{Lemma}\label{l18}
Let $n=ab$ for some $a,b\in\Z_{>1}$. Then $(S_1^*)^{\SSS_a\wr\SSS_b}=0$ unless $p=b=2$ in which case $\dim(S_1^*)^{\SSS_a\wr\SSS_b}=1$.
\end{Lemma}

\begin{proof}
This is an explicit check. We use the standard basis $v_1,\dots,v_n$ in $M_1$ and the corresponding elements $\bar v_1,\dots\bar v_n\in S_1^*=M_1/\langle \sum_{j=1}^nv_j\rangle$. 
Then $\{\bar v_1,\ldots,\bar v_{n-1}\}$ is a basis of $S_1^*$. Suppose that a non-trivial linear combination $\sum_{i=1}^{n-1}c_i\bar v_i$ is $(\SSS_a\wr\SSS_b)$-invariant. The $(\SSS_a\times\dots\times \SSS_a)$-invariance is equivalent to $c_{ka+1}=\ldots=c_{(k+1)a}$ for all $0\leq k\leq b-2$ and $c_{(b-1)a+1}=\ldots=c_{n-1}=0$. Action of $\SSS_b$ which permutes the blocks of size $a$ leaves such a vector invariant if and only if all $c_1=\dots=c_{n-a}$, $p=2$ and $b=2$. 
\end{proof}

\section{Results on branching}\label{SBr}

\subsection{Modular branching rules}\label{SSGenBrR}
Here we review some results from \cite{KBrII,KBrIII,KBook}. 
Let $V$ be an $\F\SSS_n$-module in a block $B_\theta$ for some $\theta\in \Theta_n$, cf. (\ref{EBlocks}). For any $i\in I$, we define $e_iV$ to be the projection of $V\da_{\SSS_{n-1}}$ to the block $B_{\theta-\ga_i}$ and $f_iV$ to be the projection of $V\ua^{\SSS_{n+1}}$ to the block $B_{\theta+\ga_i}$. We then extend the definition of $e_iV$ and $f_iV$ to arbitrary $\F\SSS_n$-modules additively, yielding the functors 
$$
e_i:\mod{\F\SSS_n}\to \mod{\F\SSS_{n-1}},\quad f_i:\mod{\F\SSS_n}\to \mod{\F\SSS_{n+1}}.
$$

More generally, for any $r\in\Z_{\geq 1}$ we have {\em divided power functors} 
$$e_i^{(r)}:\mod{\F \SSS_n}\rightarrow \mod{\F \SSS_{n-r}},\quad f_i^{(r)}:\mod{\F \SSS_n}\rightarrow\mod{\F \SSS_{n+r}},$$ 
see \cite[\S11.2]{KBook}. 
The following is well-known, see e.g. \cite[Lemma 8.2.2(ii),  Theorems 8.3.2(i), 11.2.7,  11.2.8]{KBook}:

\begin{Lemma}\label{Lemma45}
For any $i\in I$ and $r\in \Z_{\geq 1}$, the functors $e_i^{(r)}$ and $f_i^{(r)}$ are biadjoint and commute with duality. Moreover,  
for any $\F \SSS_n$-module $V$ we have 
\[V\da_{\SSS_{n-1}}\cong e_0V\oplus\ldots\oplus e_{p-1}V\hspace{24pt}\mbox{and}\hspace{24pt}V\ua^{\SSS_{n+1}}\cong f_0V\oplus\ldots\oplus f_{p-1}V.\]
\end{Lemma}

Recall $\tilde e_i, \tilde f_i, \eps_i, \phi_i$ from \S\ref{SSPar}. The following two results are contained in \cite[Theorems 11.2.10,  11.2.11]{KBook}, \cite[Theorem 1.4]{KDec} and \cite[Theorems E(iv), E$'$(iv)]{BrK1}.

\begin{Lemma}\label{Lemma39}
Let $\lambda\in\Parp(n)$, $i\in I$ and $r\in\Z_{\geq 0}$. Then:
\begin{enumerate}
\item[{\rm (i)}] $e_i^rD^\lambda\cong(e_i^{(r)}D^\lambda)^{\oplus r!}$;
\item[{\rm (ii)}]  $e_i^{(r)}D^\lambda\not=0$ if and only if $r\leq \eps_i(\lambda)$, in which case $e_i^{(r)}D^\lambda$ is a self-dual indecomposable module with socle and head both isomorphic to $D^{\tilde e_i^r\la}$.  
\item[{\rm (iii)}]  $[e_i^{(r)}D^\lambda:D^{\tilde e_i^r\la}]=\binom{\eps_i(\lambda)}{r}=\dim\End_{\SSS_{n-r}}(e_i^{(r)}D^\lambda)$;
\item[{\rm (iv)}] if $D^\mu$ is a composition factor of $e_i^{(r)}D^\lambda$ then $\eps_i(\mu)\leq \eps_i(\lambda)-r$, with equality holding if and only if $\mu=\tilde e_i^r\la$;
\item[{\rm (v)}] 
$\dim\End_{\SSS_{n-1}}(D^\lambda\da_{\SSS_{n-1}})=\sum_{j\in I}\eps_j(\lambda)$.
\item[{\rm (vi)}] Let $A$ be a removable node of $\la$ such that $\la_A$ is $p$-regular. Then $D^{\la_A}$ is a composition factor of $e_i D^\la$ if and only if $A$ is $i$-normal, in which case $[e_i D^\la:D^{\la_A}]$ is one more than the number of $i$-normal nodes for $\la$ above $A$. 
\end{enumerate}
\end{Lemma}

\begin{Lemma}\label{Lemma40}
Let $\lambda\in\Parp(n)$, $i\in I$ and $r\in\Z_{\geq 0}$. Then:
\begin{enumerate}
\item[{\rm (i)}] $f_i^rD^\lambda\cong(f_i^{(r)}D^\lambda)^{\oplus r!}$;
\item[{\rm (ii)}] $f_i^{(r)}D^\lambda\not=0$ if and only if $r\leq \phi_i(\lambda)$, in which case $f_i^{(r)}D^\lambda$ is a self-dual indecomposable module with socle and head both isomorphic to $D^{\tilde f_i^r\la}$.  
\item[{\rm (iii)}]  $[f_i^{(r)}D^\lambda:D^{\tilde f_i^r\la}]=\binom{\phi_i(\lambda)}{r}=\dim\End_{\SSS_{n+r}}(f_i^{(r)}D^\lambda)$;
\item[{\rm (iv)}] if $D^\mu$ is a composition factor of $f_i^{(r)}D^\lambda$ then $\phi_i(\mu)\leq \phi_i(\lambda)-r$, with equality holding if and only if $\mu=\tilde f_i^r\la$.
\item[{\rm (v)}]
$\dim\End_{\SSS_{n+1}}(D^\lambda\ua^{\SSS_{n+1}})=\sum_{j\in I}\phi_j(\lambda)$.
\item[{\rm (vi)}] Let $B$ be an addable node for $\la$ such that $\la^B$ is $p$-regular. Then $D^{\la^B}$ is a composition factor of $f_i D^\la$ if and only if $B$ is $i$-conormal, in which case $[f_i D^\la:D^{\la^B}]$ is one more than the number of $i$-conormal nodes for $\la$ below~$B$. 
\end{enumerate}
\end{Lemma}

\begin{Lemma}\label{l4} {\rm \cite[Lemma 8.5.4(ii)]{KBook}}
Let $i\in I$ and $\la\in\Parp(n)$. Then $\soc(f_ie_iD^\la)\cong (D^\la)^{\oplus \vare_i(\la)}$.
\end{Lemma}

\begin{Lemma}\label{l24}
Let $\la\in\Parp(n)$ and $i\in I$. Then 
$$[f_ie_iD^\la:D^\la]=\vare_i(\la)(\phi_i(\la)+1)\quad \text{and}\quad
[e_if_iD^\la:D^\la]=\phi_i(\la)(\vare_i(\la)+1).
$$
\end{Lemma}

\begin{proof}
This follows from \cite[Lemma 8.5.4(i), Corollary 8.5.7]{KBook} since $\eps_i(\tilde f_i\la)=\eps_i(\la)+1$ and $\phi_i(\tilde e_i\la)=\phi_i(\la)+1$.
\end{proof}

\begin{Lemma} \label{L230418} 
Let $p=2$, $n$ be even, and $\la\in\Par_2(n)$ have exactly two normal nodes. If $D^\la$ is a direct summand of $(D^\la\da_{\SSS_{n-1}})\ua^{\SSS_n}$ then $f_0e_0D^\la \oplus f_1e_1D^\la\cong D^\la\oplus X$, where $X$ is a self-dual $\F \SSS_n$-module with socle and head both isomorphic to $D^\la$ with $[X:D^\la]\geq 2$. 
\end{Lemma}
\begin{proof}
By Lemma~\ref{Lemma45}, we have 
$$(D^\la\da_{\SSS_{n-1}})\ua^{\SSS_n}\cong f_0e_0D^\la \oplus f_1e_1D^\la\oplus f_0e_1 D^\la\oplus f_1e_0 D^\la$$ with $f_0e_1 D^\la\oplus f_1e_0 D^\la$ in different blocks from $D^\la$. So $D^\la$ is a direct summand of  $f_0e_0D^\la \oplus f_1e_1D^\la$, and we can write $f_0e_0D^\la \oplus f_1e_1D^\la\cong D^\la\oplus X$ for some self-dual module $X$. By Lemma~\ref{l4}, we only have to check that $\dim\Hom_{\SSS_n}(D^\la,X)=1$ and $[X:D^\la]\geq 2$. The first statement follows from
\begin{align*}
\dim\Hom_{\SSS_n}(D^\la,f_0e_0D^\la \oplus f_1e_1D^\la)
&=
\dim\End_{\SSS_n}(e_0D^\la) + \dim\End_{\SSS_n}(e_1D^\la)\\&=\eps_0(\la)+\eps_1(\la)=2,
\end{align*}
where we have used Lemmas~\ref{Lemma45} and \ref{Lemma39}(iii). To prove the second statement, we show that $[f_0e_0D^\la \oplus f_1e_1D^\la:D^\la]\geq 3$. 

If $\eps_0(\la)=\eps_1(\la)=1$ then, noting that $\phi_i(\la)>0$ for some $i\in I$ the second statement follows from Lemma~\ref{l24}. So we may assume that $\eps_i(\la)=2$ and $\eps_{1-i}(\la)=0$. Then by Lemma~\ref{L6.2}, we have $\phi_i(\la)>0$, and so we  again conclude by Lemma~\ref{l24}. 
\end{proof}

\begin{Lemma} \label{LBrEasy} 
Let $\lambda\in\Parp(n)$ and $i\in I$. If $D^\mu$ is a composition factor of $e_i D^\la$ then there exists a removable node $A$ for $\la$ with $\res A=i$ and $\mu\unrhd \la_A$. In particular, if  $D^\mu$ is a composition factor of $D^\la\da_{\SSS_{n-1}}$ then there exists a removable node $A$ for $\la$ with  $\mu\unrhd \la_A$. 
\end{Lemma}
\begin{proof}
If $D^\mu$ is a composition factor of $e_i D^\la$ then it is a composition factor of $e_i S^\la$. By \cite[9.2]{JamesBook} and Lemma~\ref{LNak}, $e_i S^\la$ has a filtration with subquotients of the form $S^{\la_A}$ for removable nodes $A$ for $\la$ with $\res A=i$. The result now follows from \cite[12.2]{JamesBook}.
\end{proof}

A partition $\lambda\in\Parp(n)$ is called a {\em JS partition} and $D^\la$ is called a {\em JS module} 
if $D^\lambda\da_{\SSS_{n-1}}$ is irreducible. 
JS partitions were first studied in \cite{JS}. 
These can be explicitly classified, see \cite[Theorem D]{k2}. 
It is easy to see that $\la$ is JS if and only if $\la$ has exactly one normal node. In particular:

\begin{Lemma}\label{Lemma55}
Let $p=2$ and $\lambda\in\Par_2(n)$. Then $\lambda$ is JS if and only if all parts of $\lambda$ have the same parity, in which case $D^\la\da_{\SSS_{n-1}}\cong D^{(\la_1-1,\la_2,\la_3,\dots)}$. 
\end{Lemma}

\subsection{Some general branching lemmas}
\label{SSFilt}
We will study some important filtrations that arise in the restriction $D^\la\da_{\SSS_{n-1}}$.

\begin{Lemma}\label{l8}
Let $\la\in\Parp(n)$, $i\in I$ and $\vare_i(\la)>0$. Then, for $1\leq a\leq \vare_i(\la)$, there exist quotients $V_a$ of $e_iD^\la$ such that the following hold:
\begin{enumerate}
\item[{\rm (i)}]
$[V_a:D^{\tilde{e}_i \la}]=a$,

\item[{\rm (ii)}]
$V_a$ has socle and head both isomorphic to $D^{\tilde{e}_i\la}$,

\item[{\rm (iii)}]
$V_a$ is a quotient of $V_{a+1}$ for $1\leq a<\vare_i(\la)$,

\item[{\rm (iv)}]
$V_a$ is self-dual.
\end{enumerate}
\end{Lemma}

\begin{proof}
Set $\eps:=\eps_i(\la)$. By \cite[Theorem 11.2.7(ii)]{KBook}, the  algebra $\End_{\SSS_{n-1}}(e_iD^\la)$ is isomorphic to the truncated polynomial algebra $\F [x]/(x^{\eps})$, so there exists $\psi\in\End_{\SSS_{n-1}}(e_iD^\la)$ with $\psi^{\eps-1}\not=0$ and $\psi^{\eps}=0$. For $1\leq a\leq \eps$ let 
$$V_a:=e_iD^\la/\Ker(\psi^{\eps-a}).$$ 
Clearly such quotients $V_a$ satisfy (iii). Moreover, $\head V_a\cong D^{\tilde{e}_i \la}$ by Lemma~\ref{Lemma39}(ii). 
Since $\psi^{\eps-a}\not=0$ for $1\leq a\leq \eps$ by assumption on $\psi$, we have that
\[0\not=V_a\cong\im(\psi^{\eps-a})\subseteq e_iD^\la.\]
So $\soc V_a\cong D^{\tilde{e}_i\la}$ by Lemma~\ref{Lemma39}(ii), and (ii) holds.

From the assumption $\psi^{\eps-1}\not=0$ and $\psi^{\eps}=0$ we have that $V_a\not=V_{a+1}$ for each $1\leq a<\eps$. By (ii), (iii) and Lemma~\ref{Lemma39}(iii), we then have that
\[1\leq [V_1:D^{\tilde{e}_i \la}]<[V_2:D^{\tilde{e}_i\la}]<\ldots<[V_{\eps}:D^{\tilde{e}_i \la}]=\eps,\]
which implies (i).

We now prove (iv). As $e_iD^\la$ is self-dual by Lemma~\ref{Lemma39}(ii), we identify $e_iD^\la$ and $(e_iD^\la)^*$ so that $\psi$ and $\psi^*$ are both endomorphisms of $e_iD^\la$. Since $\psi$ has nilpotency degree $\eps$ and so does $\psi^*$, we must have 
$$
(\psi^r)^*=c^r\psi^r+(\text{a linear combination of terms $\psi^s$ with $s>r$})
$$
for some non-zero scalar $c$. Hence $\im((\psi^r)^*)=\im(\psi^r)$ for all $r$. Since $\im((\psi^r)^*)\cong (\im (\psi^r))^*$, we conclude that $V_a\cong\im(\psi^{\eps-a})\cong V_a^*$. 
\end{proof}

\begin{Remark} 
{\rm 
(i) Using Lemma~\ref{Lemma39}, one can easily see that we must have $V_1=\head(e_iD^\la)$ and $V_{\vare_i(\la)}=e_iD^\la$ in Lemma~\ref{l8}. 

(ii) In the proof of Lemma~\ref{l8}, we have used the fact that $\psi^*=c\psi+$(higher terms). One can use the explicit construction of $\psi$ in terms of a Murphy element in \cite{KBook} to deduce that $\psi^*=\psi$. 
}
\end{Remark}

A proof similar to that of Lemma~\ref{l8} yields:

\begin{Lemma}\label{l9}
Let $\la\in\Parp(n)$, $i\in I$ and $\phi_i(\la)>0$. Then, for $1\leq a\leq \phi_i(\la)$, there exist quotients $V_a$ of $f_iD^\la$ such that the following hold:
\begin{enumerate}
\item[{\rm (i)}]
$[V_a:D^{\tilde{f}_i \la}]=a$,

\item[{\rm (ii)}]
$V_a$ has socle and head both isomorphic to $D^{\tilde{f}_i\la}$,

\item[{\rm (iii)}]
$V_a$ is a quotient of $V_{a+1}$ for $1\leq a<\phi_i(\la)$,

\item[{\rm (iv)}]
$V_a$ is self-dual.
\end{enumerate}
\end{Lemma}

\begin{Lemma}\label{l12}
Let $p$ divide $n$ and $\la\in\Parp(n)$. Then
\[\dim\Hom_{\SSS_n}(S_1,\EE(\la))\leq \sum_{i\in I} \vare_i(\la).\]
If equality holds then there exists $i$ with $\vare_i(\la)>0$ and $D^\la\subseteq (f_iD^{\tilde{e}_i\la})/D^\la$.
\end{Lemma}

\begin{proof}
By  \cite[Lemma 4.12]{M}, we have
$$
\dim\Hom_{\SSS_n}(S_1,\EE(\la))\leq \sum_{i\in I} \vare_i(\la)-1+m
$$
where
$$
m:=\min\left\{\max_{i\,:\,\eps_i(\la)> 0}[\soc((f_i\tilde e_i D^\la)/D^\la):D^\la],\ \max_{i\,:\,\phi_i(\la)> 0}[\soc((e_i\tilde f_i D^\la)/D^\la):D^\la]
\right\}.
$$
So it is enough to prove that if $i\in I$ with $\vare_i(\la)>0$, then $[\soc((f_iD^{\tilde{e}_i\la})/D^\la):D^\la]\leq 1$. By  Lemmas~\ref{Lemma40}(iii) and \ref{l9},
there exists a quotient $V_{\phi_i(\tilde{e}_i\la)-1}=f_iD^{\tilde{e}_i\la}/X$ such that $\soc V_{\phi_i(\tilde{e}_i\la)-1}\cong D^\la$, $\soc X\cong D^\la$, and $[X:D^\la]=1$. 
The inequality $[\soc((f_iD^{\tilde{e}_i\la})/D^\la):D^\la]\leq 1$ follows.
\end{proof}

\begin{Lemma}\label{l20}
Let $\la\in\Parp(n)$, $i\in I$, $\vare_i(\la)>0$ and $D^\la\subseteq (f_iD^{\tilde{e}_i\la})/D^\la$. Then $\phi_i(\la)>0$ and $(\la_B)^C$ is $p$-singular, where $B$ and $C$ are the $i$-good and $i$-cogood nodes of $\la$ respectively.
\end{Lemma}

\begin{proof}
Set $M:=(f_i D^{\tilde{e}_i \la})/D^\la$. 
It suffices to prove that $D^\la\not\subseteq M$ if  $\phi_i(\la)=0$ or $(\la_B)^C$ is $p$-regular.
If $\phi_i(\la)=0$, then $\phi_i(\tilde{e}_i\la)=1$ and so $f_iD^{\tilde{e}_i\la}\cong D^\la$ by Lemma~\ref{Lemma40}. In particular, $M=0$, and we are done.
So we may assume that $\phi_i(\la)>0$ and $(\la_B)^C$ is $p$-regular. Note that $\tilde{e}_i \la=\la_B$, $B$ is the top $i$-conormal node for $\la_B$, and $C$ is the second  $i$-conormal node for $\la_B$ from the top. 

By \cite[Remark on p.83]{BrK1} and the self-duality of $f_i D^{\tilde{e}_i \la}$, we have that
$$
f_i D^{\tilde{e}_i \la} \sim (\bar S^\la)^*|(\bar S^{(\la_B)^C})^*|\cdots 
$$
where $\bar S^\la$ is a non-zero quotient of $S^\la$ and $\bar S^{(\la_B)^C}$ is a non-zero quotient $S^{(\la_B)^C}$ with $[\bar S^{(\la_B)^C}:D^\la]=1$. 

Let $\hat Z$ be the submodule $(\bar S^\la)^*|\bar (S^{(\la_B)^C})^*$ of $f_iD^{\tilde{e}_i \la}$, and $Z=\hat Z/D_\la$ be the corresponding submodule of $M$. Note that 
$[Z:D^\la]=1$, and 
$\Hom_{\SSS_n}(D^\la,Z)=0$ since $D^\la$ is not a composition factor of $(\bar S^\la)^*/D^\la$ and  $\soc (S^{(\la_B)^C})^*\cong D^{(\la_B)^C}\not\cong D^\la$. 

Let $V:=V_{\phi_i(\tilde{e}_i \la)-1}$ be as in Lemma \ref{l9}. Then $\soc V\cong D^\la$ and 
\[[V:D^\la]=\phi_i(\tilde{e}_i \la)-1=[M:D^\la],\]
where the second equality is by Lemma~\ref{Lemma40}(iii). Let $X\subseteq M$ be a submodule such that $M/X\cong V$. By the last equality, $[X:D^\la]=0$. So, setting $Y:=X+Z$, we now deduce from the previous paragraph that $\Hom_{\SSS_n}(D^\la,Y)=0$. 
Note that $Y\supsetneq X$ since $[X:D^\la]=0$, while $[Y:D^\la]\geq [Z:D^\la]=1$. Since $V$ has simple socle, it follows that $\soc M/X\subseteq Y/X$, and we can now apply Lemma~\ref{L200418}. 
\end{proof}

\subsection{Some branching for JS modules} 
\label{SSJSBr}

In this subsection we will always assume that $p=2$ and $\la$ is a JS partition. 
By definition, the top removable node $A$ of $\la$ is its only normal node, and $D^\la\da_{\SSS_{n-1}}\cong D^{\la_A}$. In this sense JS modules have very simple branching. However, we need to prove some results about their restrictions to other subgroups. 

\begin{Lemma}\label{l26}
Let $p=2$, $\la\in\Par_2(m+n)$, $\mu\in\Par_2(m)$ and $\nu \in\Par(n)$. If $\mu+\nu=\la$ and $(\la_1,\ldots,\la_{h(\nu)})$ is a JS-partition, then $D^\mu$ is a composition factor of $D^\la\da_{\SSS_{m}}$.
\end{Lemma}
\begin{proof}
We apply induction on $n$, the case $n=0$ being clear. Let 
$$\ka:=\la-(1^{h(\nu)})=(\la_1-1,\dots,\la_{h(\nu)}-1,\la_{h(\nu)+1},\dots,\la_{m+n}).$$ 
Note that $\ka$ is $2$-regular, since $\la$ and $\mu$ are $2$-regular and by definition
\[\ka_{h(\nu)}=\la_{h(\nu)}-1\geq\mu_{h(\nu)}>\mu_{h(\nu)+1}
=\la_{h(\nu)+1}=\ka_{h(\nu)+1}.\]
Further $h(\nu-(1^{h(\nu)}))\leq h(\nu)$, $\ka=\mu+(\nu-(1^{h(\nu)}))$ and $(\ka_1,\ldots,\ka_{h(\nu)})$ is a JS-partition, see Lemma \ref{Lemma55}. By the inductive assumption, it suffices to prove that $D^\ka$ is a composition factor of $D^\la\da_{\SSS_{m+n-h(\nu)}}$. Let $B_s:=(s,\la_s)$ be the last node in row $s$ of $\la$, $s=1,2,\dots$. 
Using for example Lemma~\ref{Lemma55}, it is easy to see that 
the node $B_1$ is good for $\la$, $B_2$ is good for $\la_{B_1}$, $B_3$ is good for $(\la_{B_1})_{B_2}$, etc. By Lemma~\ref{Lemma39}(ii), we have that $D^{\la_{B_1}}$ is a composition factor of $D^\la\da_{\SSS_{m+n-1}}$, $D^{(\la_{B_1})_{B_2}}$ is a composition factor of $D^{\la_{B_1}}\da_{\SSS_{m+n-2}}$, etc., and the required result on $D^\ka$ follows since $\ka=(\dots(\la_{B_1})_{B_2}\dots)_{B_{h(\nu)}}$.
%
%
%
\end{proof}

\begin{Lemma}\label{l25}
Let $p=2$, $n$ be even and $\la\in\Par_2(n)$ be a JS-partition with odd parts. Then $D^\la{\da}_{\SSS_{n/2}}$ has at least three  non-isomorphic composition factors, unless one of the following holds:
\begin{itemize}
\item[{\rm (i)}]
$n\geq 4$ and $\la=\al_n$,

\item[{\rm (ii)}]
$n\geq 8$ with $n\equiv 0\pmod{4}$ and $\la=\be_n$,

\item[{\rm (iii)}]
$n\geq 24$ with $n\equiv 0\pmod{8}$ and $\la=(n/4+3,n/4+1,n/4-1,n/4-3)$,

\item[{\rm (iv)}]
$n\geq 22$ with $n\equiv 4\pmod{6}$ and $\la=((n-4)/3+3,(n-4)/3+1,(n-4)/3-1,1)$.
\end{itemize}
\end{Lemma}

\begin{proof}
From Lemma \ref{l26} it is enough to find 
distinct $\mu,\si,\pi\in\Par_2(n/2)$ such that 
$\la-\mu,\la-\si$ and $\la-\pi$ are partitions. 
Notice that $h(\la)$ is even since $n$ is even and $\la$ consists of odd parts. 

{\sf Case 1.} $h(\la)\geq 6$. In this case we can take
\begin{align*}
\mu&=\left(\frac{\la_1+1}{2},\ldots,\frac{\la_{h(\la)/2}+1}{2},\frac{\la_{h(\la)/2+1}-1}{2},\ldots,\frac{\la_{h(\la)}-1}{2}\right),\\
\si&=\left(\frac{\la_1+3}{2},\frac{\la_2+1}{2},\ldots,\frac{\la_{h(\la)/2-1}+1}{2},\frac{\la_{h(\la)/2}-1}{2},\ldots,\frac{\la_{h(\la)}-1}{2}\right).
\end{align*}
If $\la_{h(\la)}\geq 3$ then we can also take
\[\pi=\left(\frac{\la_1+3}{2},\frac{\la_2+1}{2},\ldots,\frac{\la_{h(\la)/2}+1}{2},\frac{\la_{h(\la)/2+1}-1}{2},\ldots,\frac{\la_{h(\la)-1}-1}{2},\frac{\la_{h(\la)}-3}{2}\right),\]
while if $\la_{h(\la)}=1$ we can take
\[\pi=\left(\frac{\la_1+3}{2},\frac{\la_2+1}{2},\ldots,\frac{\la_{h(\la)/2}+1}{2},\frac{\la_{h(\la)/2+1}-1}{2},\ldots,\frac{\la_{h(\la)-2}-1}{2},\frac{\la_{h(\la)-1}-3}{2}\right).\]

{\sf Case 2.} $h(\la)=4$. In this case we can take 
$$\mu=((\la_1+1)/2,(\la_2+1)/2,(\la_3-1)/2,(\la_4-1)/2).$$ 
If $\la_1\geq\la_2+4$ we can also take 
\begin{align*}
\si&=((\la_1-1)/2,(\la_2+1)/2,(\la_3+1)/2,(\la_4-1)/2)
\\ 
\pi&=((\la_1+3)/2,(\la_2-1)/2,(\la_3-1)/2,(\la_4-1)/2).
\end{align*}
If $\la_2\geq\la_3+4$ we can also take 
\begin{align*}
\si&=((\la_1-1)/2,(\la_2-1)/2,(\la_3+1)/2,(\la_4+1)/2)
\\
\pi&=((\la_1+1)/2,(\la_2-1)/2,(\la_3+1)/2,(\la_4-1)/2).
 \end{align*}
We can now assume that $\la_1=\la_2+2=\la_3+4$.

If $\la_3-\la_4=2$, then either we are in the excluded case (iii) or $\la=(7,5,3,1)$. By Lemma \ref{l26}, $D^{(4,3,2,1)}$ is a composition factor of $D^{(7,5,3,1)}\da_{\SSS_{10}}$. 
Since $D^{(4,3,2,1)}\cong S^{(4,3,2,1)}$ (as $(4,3,2,1)$ is a $2$-core), it follows from \cite[9.3, Tables]{JamesBook} that $D^{(4,3,2,1)}\da_{\SSS_8}$ and then also $D^{(7,5,3,1)}\da_{\SSS_8}$ has at least three non-isomorphic composition factors.

If $\la_3-\la_4>2$, then either we are in the excluded case (iv) or $\la_4\geq 3$. In this case we can take  
\begin{align*}
\si&=((\la_1+1)/2,(\la_2-1)/2,(\la_3-1)/2,(\la_4+1)/2)
\\
\pi&=((\la_1+3)/2,(\la_2+1)/2,(\la_3-1)/2,(\la_4-3)/2).
 \end{align*}

{\sf Case 3.} $h(\la)=2$. 
If $\la_2=1$, we are in the exceptional case (i). So from now on we suppose that $\la_2\geq 3$. Moreover, if $\la_1-\la_2=2$, we are in the exceptional case (ii). So from now on we also assume that $\la_1-\la_2\geq 4$. 

Assume first that $\la_1-\la_2\geq 6$. If $\la_2\leq n/4$ we take 
$$\mu=(n/2),\ \si=(n/2-1,1),\ \pi=(n/2-2,2).$$ 
If $\la_2> n/4$ we take 
\begin{align*}
\mu&=(\la_1-\lceil n/4\rceil,\la_2-\lfloor n/4\rfloor), \\ 
\si&=(\la_1-\lceil n/4\rceil-1,\la_2-\lfloor n/4\rfloor+1),\\ 
\pi&=(\la_1-\lceil n/4\rceil-2,\la_2-\lfloor n/4\rfloor+2).
\end{align*}

Assume finally that $\la_1-\la_2=4$, i.e. $\la=(n/2+2,n/2-2)$. Then  $n\equiv 2\pmod{4}$ and we may assume that $n\geq 10$ as for $n=6$ we are in the exceptional case (i). We can take $\mu=((n+6)/4,(n-6)/4)$ and $\si=((n+2)/4,(n-2)/4)$. We complete the proof by showing that $D^{((n+10)/4,(n-10)/4)}$ is also a composition factor of $D^{(n/2+2,n/2-2)}\da_{\SSS_{n/2}}$. By  Lemma~\ref{Lemma55}, we have $D^{(n/2+2,n/2-2)}\da_{\SSS_{n-1}}\cong D^{(n/2+1,n/2-2)}$. 
Further, by Lemma \ref{Lemma39} we have that $D^{(n/2+2,n/2-4)}$ is a composition factor of $D^{(n/2+2,n/2-2)}\da_{\SSS_{n-2}}$. Since $(n/2+2,n/2-4)$ is a JS-partition, it then follows from Lemma \ref{l26} that $D^{((n+10)/4,(n-10)/4)}$ is a composition factor of $D^{(n/2+2,n/2-2)}\da_{\SSS_{n/2}}$.
\end{proof}

\subsection{Branching recognition}
\label{SSRecognition}
In this subsection we  obtain characterizations of certain classes of irreducible modules by their branching properties. 

The following lemma develops \cite[Lemma 2.7]{BK}.

\begin{Lemma} \label{LBr1} 
Let $p=3$, $n>6$ and $\la\in\Par_3(n)$. Suppose that $h(\mu)\leq 2$ or $h(\mu^\Mull)\leq 2$ for all composition factors $D^\mu$ of $D^\la{\da}_{\SSS_{n-1}}$. Then $h(\la)\leq 2$ or $h(\la^\Mull)\leq 2$. 
\end{Lemma}
\begin{proof}
Pick a good node $A$ of $\la$. By \cite{KBrIII}, we have that $(\la_A)^\Mull=(\la^\Mull)_B$ for some good node $B$ of $\la^\Mull$. By Lemma~\ref{Lemma39}, $D^{\la_A}$ is a composition factor of $D^\la{\da}_{\SSS_{n-1}}$ and $D^{(\la^\Mull)_B}$ is a composition factor of $D^{\la^\Mull}{\da}_{\SSS_{n-1}}$. 
If $h(\la)\geq 4$ then $h(\la_A)\geq 3$. If $h(\la^\Mull)\geq 4$ then $h((\la_A)^\Mull)=h((\la^\Mull)_B)\geq 3$. So we cannot have both $h(\la)\geq 4$ and $h(\la^\Mull)\geq 4$. So, tensoring with sign if necessary, we may assume that $h(\la)=3\leq h(\la^\Mull)$. Recall that $G_1(\la)$ denotes the first column of the Mullineux symbol for $\la$. 

\vspace{1 mm}
\noindent
{\sf Claim.} {\em If $B$ is a normal node of $\la$ such that $\la_B$ is $3$-regular, then $G_1(\la_B)\neq G_1(\la)$}.

\vspace{1 mm}
\noindent
Indeed, by Lemma~\ref{Lemma39}(vi), $D^{\la_B}$ is a composition factor of $D^\la{\da}_{\SSS_{n-1}}$. By assumption, we must then have $h(\la_B)\leq 2$ or $h((\la_B)^\Mull)\leq 2$. 
If $h(\la_B)\leq 2$, then $G_1(\la_B)\neq G_1(\la)$ since $h(\la)=3$ and $h(\la)$ is part of the data $G_1(\la)$. If $h((\la_B)^\Mull)\leq 2$, then similarly $G_1((\la_B)^\Mull)\neq G_1(\la^\Mull)$ since $h(\la^\Mull)\geq 3$; hence  $G_1(\la_B)\neq G_1(\la)$. 
The Claim is proved.

The first $3$-segment of $\la$ has one of the following forms, which we will consider case by case (nodes of the first $3$-segment are marked with $\tt x$'s):
$$
\begin{tikzpicture}
\node at (-4.8,0.005){(a)};
\node at (0.55,0.01) {$\begin{Young}
$\tt x$\cr 
\end{Young}$};
\node at (-4,-0.465) {$\begin{Young}
\cr
\cr
\cr
\end{Young}$};
\node at (0.307,-0.46) {$\begin{Young}
$\tt x$& $\tt x$\cr 
\end{Young}$};
\node at (-4,-0.465) {$\begin{Young}
\cr
\cr
\cr
\end{Young}$};
\draw [-] (-3.9,0.25) -- (0.3,0.25);
\node at (-2,-0.9){$\cdots$};
\end{tikzpicture}
$$
$$
\begin{tikzpicture}
\node at (-4.8,0.005){(b)};
\node at (0.3,-0.23) {$\begin{Young}
$\tt x$& $\tt x$\cr $\tt x$ \cr
\end{Young}$};
\node at (-4,-0.465) {$\begin{Young}
\cr
\cr
\cr
\end{Young}$};
\draw [-] (-3.9,0.25) -- (-0.16,0.25);
\node at (-2,-0.9){$\cdots$};
\end{tikzpicture}
$$
$$
\begin{tikzpicture}
\node at (-4.8,0.005){(c)};
\node at (0,0.005) {$\begin{Young}
$\tt x$& $\tt x$& $\tt x$ \cr
\end{Young}$};
\node at (-4,-0.465) {$\begin{Young}
\cr
\cr
\cr
\end{Young}$};
\draw [-] (-3.9,0.25) -- (-0.7,0.25);
\node at (-2,-0.9){$\cdots$};
\end{tikzpicture}
$$

{\sf Case (a).} Let $A$ be the top removable node of $\la$. Then  $G_1(\la)=G_1(\la_A)$, which contradicts the Claim. 

{\sf Case (b).} If $\la_3<\la_2$ and $A$ is the top removable node of $\la$, then  $G_1(\la)=G_1(\la_A)$, which contradicts the Claim. Otherwise, $\la$ is of the form $(k+1,k,k)$. In the exceptional case,  the bottom removable node $A$ is normal for $\la$. By Lemma~\ref{Lemma39}(vi), $D^{\la_A}$ is a composition factor of $D^\la{\da}_{\SSS_{n-1}}$. On the other hand, $h(\la_A)=3$ and it is easy to see that $h((\la_A)^\Mull)\geq 3$ unless $n=7$. For $n=7$ we 
have $\la=(3,2,2)$ and $\la^\Mull=(5,1,1)$. Hence $D^{(4,1,1)}$ is a composition factor of $D^{\la^\Mull}{\da}_{\SSS_{6}}$, and, since $(4,1,1)^\Mull=(4,1,1)$, we deduce that $D^{(4,1,1)}$ is a composition factor of $D^{\la}{\da}_{\SSS_{6}}$ violating our assumptions. 

{\sf Case (c).} If $\la_2<\la_1-2$ and $A$ is the top removable node of $\la$, then  $G_1(\la)=G_1(\la_A)$, which contradicts the Claim. So we may assume that $\la_2=\la_1-2$. Consider the second $3$-segment. We now have the following cases (nodes of the second $3$-segment are marked with $\bullet$'s):
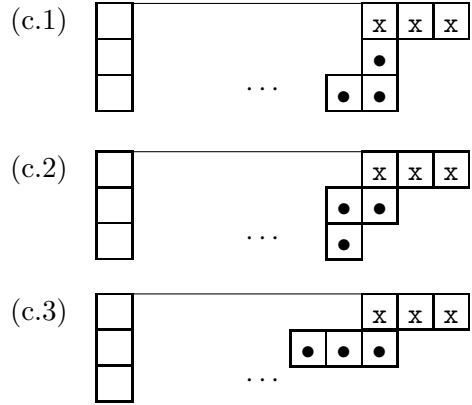
$$
\begin{tikzpicture}
\node at (-5,0.005){(c.1)};
\node at (0,-0.465) {$\begin{Young}
$\tt x$& $\tt x$& $\tt x$ \cr
$\bullet$\cr
$\bullet$\cr
\end{Young}$};
\node at (-4,-0.465) {$\begin{Young}
\cr
\cr
\cr
\end{Young}$};
\draw [-] (-3.9,0.25) -- (-0.7,0.25);
\node at (-2,-0.9){$\cdots$};
\node at (-0.95,-0.945) {$\begin{Young}
$\bullet$
\cr
\end{Young}$};
\end{tikzpicture}
$$
$$
\begin{tikzpicture}
\node at (-5,0.005){(c.2)};
\node at (0,0.013) {$\begin{Young}
$\tt x$& $\tt x$& $\tt x$ \cr
\end{Young}$};
\node at (-4,-0.465) {$\begin{Young}
\cr
\cr
\cr
\end{Young}$};
\draw [-] (-3.9,0.25) -- (-0.7,0.25);
\node at (-2,-0.9){$\cdots$};
\node at (-0.71,-0.705) {$\begin{Young}
$\bullet$&$\bullet$\cr
$\bullet$\cr
\end{Young}$};
\end{tikzpicture}
$$
$$
\begin{tikzpicture}
\node at (-5,0.005){(c.3)};
\node at (0,0.013) {$\begin{Young}
$\tt x$& $\tt x$& $\tt x$ \cr
\end{Young}$};
\node at (-4,-0.465) {$\begin{Young}
\cr
\cr
\cr
\end{Young}$};
\draw [-] (-3.9,0.25) -- (-0.7,0.25);
\node at (-2,-0.9){$\cdots$};
\node at (-0.95,-0.475) {$\begin{Young}
$\bullet$&$\bullet$&
$\bullet$\cr
\end{Young}$};
\end{tikzpicture}
$$
In the case (c.1), the bottom removable node $A$ is normal for $\la$, and $G_1(\la)=G_1(\la_A)$, which contradicts the Claim. In the case (c.2), the second removable node $A$ is normal for $\la$, and, unless $n=7$, we get $G_1(\la)=G_1(\la_A)$, which contradicts the Claim. In the exceptional case $\la=(4,2,1)$ and $\la_A=(4,1,1)$, and so we get a contradiction as in the case (b). In the case (c.3), we have $\la=(k+2,k,l)$ for $1\leq l\leq k-2$. 
The second removable node $A$ is normal for $\la$ and if $l<k-2$ we get $G_1(\la)=G_1(\la_A)$, which contradicts the Claim. Let $l=k-2$. In this case the bottom removable node $B$ is normal for $\la$, and, unless $l\leq 3$, we get $G_1(\la)=G_1(\la_A)$, which contradicts the Claim. In the exceptional cases, the second removable node $A$ is normal for $\la$ which yields a composition factor $D^{(k+2,k-1,l)}$ of $D^\la{\da}_{\SSS_{n-1}}$ which violates the assumptions. 
\end{proof}

\begin{Lemma} \label{LBr2} 
Let $p=2$, $n>6$ and $\la\in\Par_2(n)$. Suppose that $h(\mu)\leq 2$ for all composition factors $D^\mu$ of $D^\la{\da}_{\SSS_{n-1}}$. Then $h(\la)\leq 2$. 
\end{Lemma}
\begin{proof}
Since $D^{\la_A}$ is a composition factor of $D^\la{\da}_{\SSS_{n-1}}$ for any good node $A$ for $\la$, we may assume that $h(\la)=3$. But in this case, the assumption $n>6$ guarantees that there always is a normal node $A$ for $\la$ such that $\la_A$ is $2$-regular and $h(\la_A)=3$. 
\end{proof}

Recall the partition $\be_n$ defined in \eqref{ESpin}. 

\begin{Lemma} \label{L180418_3} 
Let $p=2$, $n\geq 7$ and $\la\in\Par_2(n)$. If all composition factors of $D^\la\da_{\SSS_{n-1}}$ are of the form $\bone_{\SSS_{n-1}}$ or $D^{\be_{n-1}}$ then either $D^\la\cong\bone_{\SSS_n}$ or $\la=\be_n$.  
\end{Lemma}
\begin{proof}
By Lemma~\ref{LBr2}, we may assume that $h(\la)=2$. If $\la_1-\la_2\leq 2$, then $\la=\be_n$. If $\la_1-\la_2>3$, then $(\la_1-1,\la_2)\neq \be_{n-1}$, while $D^{(\la_1-1,\la_2)}$ is a composition factor of $D^\la\da_{\SSS_{n-1}}$ by Lemma~\ref{Lemma39}(vi). Finally, if $\la_1-\la_2=3$, then $\la_2\geq 2$ since $n\geq 7$ and $(\la_1,\la_2-1)\not\in \{(n-1),\be_{n-1}\}$, while  $D^{(\la_1,\la_2-1)}$ is a composition factor of $D^\la\da_{\SSS_{n-1}}$ by Lemma~\ref{Lemma39}(vi).
\end{proof}

\section{Permutation modules}
\label{SPerm}

\subsection{Some general results}

We record two known general results concerning permutation modules $M_k$ and Specht modules $S_k$. 

\begin{Lemma} \label{LFilt} {\rm \cite[17.17]{JamesBook}} 
If $0\leq k\leq n/2$ then $M_k\sim S_k|S_{k-1}|\ldots|S_0$.
\end{Lemma}

Given $a,b\in\Z_{\geq 0}$ with $p$-adic expansions $a=\sum_{t=0}^r a_tp^t$, $b=\sum_{t=0}^s b_tp^t$ such that $a_r\neq0$, $b_s\neq 0$, we say that $a$ {\em contains $b$ to the base  $p$} if $s<r$ and for all $t$ we have $b_t=0$ or $b_t=a_t$. 

\begin{Lemma} \label{L2R} {\rm \cite[24.15]{JamesBook}} 
All composition factors of $S_k$ are of the form $D_j$ with $j\leq k$. Moreover, $[S_k:D_j]=1$ if $n-2j+1$ contains $k-j$ to the base $p$, and   $[S_k:D_j]=0$ otherwise. 
\end{Lemma}

\subsection{The case $p=3$.}
\begin{Lemma} \label{L160817_0} 
Let $p=3$, $n\equiv 0\pmod{3}$ with $n\geq 6$. Then 
$$M_1\cong D_0|D_1|D_0,\quad M_2\cong D_2\oplus M_1\quad\text{and}\quad M_3\sim D_2\oplus ((D_0\oplus S_1^*)|S_3^*).$$
\end{Lemma}
\begin{proof}
The structure of $M_1$ and $M_2$ has been described for example in  \cite[Lemmas 1.1, 1.2]{BK1}. From the same lemmas we also have that $S_1\cong D_0|D_1$ and that $S_2\cong D_2$. From Lemma~\ref{LNak} we have that $D_0$, $D_1$ and $D_3$ are contained in the same block, while $D_2$ is contained in a different block. From Lemma~\ref{LFilt} and from self-duality of $M_3$ and of the simple $\SSS_n$-modules, we then have that
\[M_3\sim S_0^*|S_1^*|S_2^*|S_3^*\sim D_2\oplus (D_0|S_1^*|S_3^*).\]

From Lemma~\ref{LWilson} we have that
$\rank(\eta_{1,3})=n-1=\dim(M_1)-1.$ 
From $M_1\cong D_0|D_1|D_0\sim D_0|S_1^*$, it then follows that $\im(\eta_{1,3})\sim S_1^*$. In particular $S_1^*\subseteq M_3$.
Since $D_2\oplus D_0\subseteq M_3$ and neither $D_2$ nor $D_0$ is contained in $S_1^*\cong D_1|D_0$, 
the fact that $S_1^*\subseteq M_3$ implies that there exists a module $N$ with
\[N\cong D_2\oplus D_0\oplus S_1^*\subseteq M_3.\]
Notice that $N$ does not have any composition factor isomorphic to $D_3\cong \soc(S_3^*)$. Since there exists a quotient of $M_3$ isomorphic to $S_3^*$, it follows that the same holds for $M_3/N$. By comparing dimensions we then have that $M_3/N\cong S_3^*$. In particular, by block decomposition,
\[M_3\sim N|(M_3/N)\sim D_2\oplus ((D_0\oplus S_1^*)|S_3^*).\]
\end{proof}

\begin{Lemma} \label{L160817_1} 
Let $p=3$, $n\equiv 1\pmod{3}$ with $n\geq 7$. Then 
\begin{align*}
&M_1\cong D_0\oplus D_1, &M_2\cong D_1\oplus D_0|D_2|D_0,\\ &S_2\cong D_0|D_2,
&M_3\sim D_1\oplus ((D_0\oplus S_2^*)|S_3^*).
\end{align*}
\end{Lemma}
\begin{proof}
The structure of $M_1$ and $M_2$ has been described for example in \cite[Lemmas 1.1, 1.2]{BK1}. From the same lemmas we also have that $S_1\cong D_1$ and that $S_2\cong D_0|D_2$. From Lemma~\ref{LNak} we have that $D_0$, $D_2$ and $D_3$ are contained in the same block, while $D_1$ is contained in a different block. From Lemma~\ref{LFilt} and from self-duality of $M_3$ and of the simple $\SSS_n$-modules, we then have that
\[M_3\sim S_0^*|S_1^*|S_2^*|S_3^*\sim D_1\oplus (D_0|S_2^*|S_3^*).\]

From Lemma~\ref{LWilson} we have that
\[\rank(\eta_{2,3})=\binom{n}{2}-1=\dim(M_2)-1.\]
As $M_2\cong D_1\oplus (D_0|D_2|D_0)$ and $\dim(D_1)>1$, it then follows that $\im(\eta_{2,3})\sim D_1\oplus S_2^*$. 
Since $D_1\oplus D_0\subseteq M_3$ and neither $D_1$ nor $D_0$ is contained in $S_2^*\cong D_2|D_0$, the fact that $S_2^*\subseteq M_3$ implies that there exists a module $N$ with
\[N\cong D_1\oplus D_0\oplus S_2^*\subseteq M_3.\]
Notice that $N$ does not have any composition factor isomorphic to $D_3\cong \soc(S_3^*)$. Since there exists a quotient of $M_3$ isomorphic to $S_3^*$, it follows that the same holds for $M_3/N$. Again by comparing dimensions we have that $M_3/N\cong S_3^*$, and it follows by block decomposition that
\[M_3\sim N|(M_3/N)\sim D_1\oplus ((D_0\oplus S_2^*)|S_3^*).\]
\end{proof}

\begin{Lemma} \label{L160817_2} 
Let $p=3$, $n\geq 8$ with $n\equiv 2\pmod{3}$. Then 
$$M_1\cong D_0\oplus D_1,\quad M_2\cong D_0\oplus D_1|D_2|D_1,\quad\text{and}\quad M_3\sim M_2|S_3^*.$$
Moreover,
\begin{enumerate}
\item[{\rm (i)}] If $n\equiv 2\pmod{9}$ then
\[M_3\cong (D_0|D_3|D_0)\oplus (D_1|D_2|D_1).\]
\item[{\rm (ii)}] If $n\equiv 5\pmod{9}$ or $n\equiv 8\pmod{9}$ then
\[M_3\cong D_0\oplus D_3\oplus (D_1|D_2|D_1).\]
\end{enumerate}
\end{Lemma}

\begin{proof}
The structure of $M_1$ and $M_2$ follows for example from \cite[Lemmas 1.1, 1.2]{BK1}.
Since $n\equiv 2\pmod{3}$, Lemma~\ref{LNak} shows that $S_0$ and $S_3$ are in the same block, as are $S_1$ and $S_2$, but  $S_0$ and $S_3$ are contained in a different block from $S_1$ and $S_2$. From Lemma~\ref{LFilt} it then follows that
\[M_3\sim S_3|S_2|S_1|S_0\sim (S_3|S_0)\oplus(S_2|S_1).\]
From Lemma~\ref{L2R} it follows that
\[M_3\sim (\overbrace{D_0|D_3}^{S_3}|\overbrace{D_0}^{S_0})\oplus (\overbrace{D_1|D_2}^{S_2}|\overbrace{D_1}^{S_1})\]
if $n\equiv 2\pmod{9}$, while
\[M_3\sim (\overbrace{D_3}^{S_3}|\overbrace{D_0}^{S_0})\oplus (\overbrace{D_1|D_2}^{S_2}|\overbrace{D_1}^{S_1})\]
if $n\equiv 5\mbox{ or }8\pmod{9}$. The lemma now follows from Lemma~\ref{LYoung} and  self-duality of $M_3$.
\end{proof}

\subsection{The case $p=2$}

\begin{Lemma} \label{L150817_7} 
Let $p=2$ and $n\geq 7$ be odd.  Then $M_1\cong D_0\oplus D_1$, $M_2\subseteq M_3$, and $M_3/M_2\cong S_3^*$. Moreover:
\begin{enumerate}
\item[{\rm (i)}] If $n\equiv 1\pmod{4}$, then $M_2\cong D_1\oplus D_0|D_2|D_0$, $M_3\cong D_3\oplus M_2$, and $S_3\cong D_3$. 
\item[{\rm (ii)}]  If $n\equiv 3\pmod{4}$, then  $M_2\cong D_0\oplus D_1\oplus D_2$, $M_3\cong D_0\oplus D_2\oplus D_1|D_3|D_1$, and $S_3\cong D_1|D_3$. 
\end{enumerate}
\end{Lemma}
\begin{proof}
The structure of $M_1$ and $M_2$ follows for example from \cite[Lemmas 1.1, 1.3]{BK1}.
By Lemma~\ref{LNak}, $S_0$ and $S_2$ are in the same block, as are $S_1$ and $S_3$, but  $S_0$ and $S_2$ are contained in a different block from $S_1$ and $S_3$. From Lemma~\ref{LFilt} it then follows that
\[M_3\sim S_3|S_2|S_1|S_0\sim (S_2|S_0)\oplus(S_3|S_1).\]
From Lemma~\ref{L2R} it follows that
\[M_3\sim (\overbrace{D_0|D_2}^{S_2}|\overbrace{D_0}^{S_0})\oplus (\overbrace{D_3}^{S_3}|\overbrace{D_1}^{S_1})\]
if $n\equiv 1\pmod{4}$, while
\[M_3\sim (\overbrace{D_2}^{S_2}|\overbrace{D_0}^{S_0})\oplus (\overbrace{D_1|D_3}^{S_3}|\overbrace{D_1}^{S_1})\]
if $n\equiv 3\pmod{4}$. The lemma now follows from self-duality of $M_3$ and Lemma~\ref{LYoung}.
\end{proof}

\begin{Lemma}\label{l1}\label{l10}
Let $p=2$ and $n\geq 6$ be even. Then
\begin{align*}
&M_1\cong D_0|D_1|D_0\sim D_0|S_1^*,\\
&S_1\cong D_0|D_1,\\
&M_2\sim(D_0\oplus S_1^*)|S_2^*.
\end{align*}
Moreover,
\begin{enumerate}
\item[{\rm (i)}]
If $n\equiv 0\pmod{4}$ then $S_2\cong D_1|D_2$ and $M_2=Y_2\sim S_1^*|D_2|S_1$.

\item[{\rm (ii)}]
If $n\equiv 2\pmod {4}$ then
\begin{align*}
&M_2\cong D_0\oplus Y_2,\\
&Y_2\cong D_1|D_0|D_2|D_0|D_1\sim D_1|D_0|S_2^*,\\
&S_2\cong D_1|D_0|D_2.
\end{align*}
\end{enumerate}
\end{Lemma}

\begin{proof}
The structure of $M_1$ and $S_1$ is well-known, see e.g.  \cite[Lemma 1.1]{BK1}. By Lemma~\ref{LWilson}, we have 
$\rank(\eta_{1,2})=n-1=\dim(M_1)-1.$ 
It then follows that $S_1^*\subseteq M_2$. From Lemma~\ref{LFilt} and self-duality of $M_2$ we have that
\begin{equation}\label{E170418}
M_2\sim S_0^*|S_1^*|S_2^*\sim D_0|S_1^*|S_2^*.
\end{equation}
Since $D_0,S_1^*\subseteq M_2$ and $D_0\not\subseteq S_1^*\cong D_1|D_0$, there exists a module $N$ with
\[N\cong D_0\oplus S_1^*\subseteq M_2.\]
Note that $D_2\cong \soc(S_2^*)$ is not a composition factor of $N$. Since there exists a quotient of $M_2$ isomorphic to $S_2^*$, it follows that the same holds for $M_2/N$. By comparing dimensions we then have that $M_2/N\cong S_2^*$. In particular
\[M_2\sim N|(M_2/N)\sim (D_0\oplus S_1^*)|S_2^*.\]

For $n\equiv 2\pmod{4}$ the structures of $M_2$ and $S_2$ are described in \cite[(1.1), (2.4)]{MO}. So let us assume that $n\equiv 0\pmod{4}$. By \cite[(1.1)]{MO}, we have $M_2\cong Y_2$. 
We also have $S_2\cong D_1|D_2$ by Lemma~\ref{L2R}. To prove that $M_2\sim S_1^*|D_2|S_1$, let $A:=\Ker(\eta_{2,1})$. Since $\eta_{2,1}=\eta_{1,2}^*$ we have that $M_2/A\cong S_1$. 

Let $\{v_i\mid1\leq i\leq n\}$ be the standard permutation basis of $M_1$ and $\{v_{i,j}\mid1\leq i<j\leq n\}$ be the standard permutation basis of $M_2$, so that $\eta_{2,1}(v_{i,j})=v_i+v_j$. 
The only submodule of $M_2$ isomorphic to $D_0$ is $\langle\sum_{i<j}v_{i,j}\rangle$. Note that 
\[\eta_{2,1}\Big(\sum_{i<j}v_{i,j}\Big)=\sum_{i<j}(v_i+v_j)=\sum_i(n-1)v_i\not=0,\]
hence $D_0\not\subseteq A$. 
Since $N\cong D_0\oplus S_1^*\cong D_0\oplus D_1|D_0$ and $M_1\cong D_0|D_1|D_0$, we must have $\eta_{2,1}(N)=D_0$ and $A\cap N\cong S_1^*$ using the Krull-Schmidt Theorem. The composition factors of $A$ are $D_0,D_1,D_2$, so it follows that $A/A\cap N\cong D_2$, completing the proof. 
\end{proof}

\begin{Lemma}\label{l14}
Let $p=2$ and $n\geq 8$ even. Then $M_3\sim S_3|S_2|M_1$. Moreover:
\begin{enumerate}
\item[{\rm (i)}]
If $n\equiv 0\pmod{4}$ then
\[M_3\cong M_1\oplus (\overbrace{D_2|D_1|D_3}^{S_3}|\overbrace{D_1|D_2}^{S_2}).\]

\item[{\rm (ii)}]
If $n\equiv 2\pmod{4}$ then 
$$S_3\cong D_0|D_2|D_3,\quad  
M_3\sim (Y_2/D_1)|S_1^*|S_3^*,$$ and there exists $A\subseteq Y_2/D_1$ with $A\cong D_0|D_2|D_0$ and $M_3\sim A|D_3|S_2|S_1^*$.
\end{enumerate}
\end{Lemma}

\begin{proof}
By Lemma~\ref{LFilt}, we have $M_3\sim S_3|S_2|S_1|S_0$.
Note that $M_3$ has a unique submodule isomorphic to $S_3$, since $M_3$ has a unique composition factor isomorphic to $D_3\cong\head S_3$. Similarly $M_3/S_3$ has a unique submodule isomorphic to $S_2$. So there is a unique submodule  $X\subseteq M_3$ such that $X\sim S_3|S_2$. Moreover, $X$ is the unique minimal submodule of $M_3$ with $[X:D_3]=[M_3:D_3]$ and $[X:D_2]=[M_3:D_2]$.


Since $\rank \eta_{3,1}=n$ by Lemma~\ref{LWilson}, 
we have that $M_1$ is a quotient of $M_3$. 
Since $M_1\cong D_0|D_1|D_0\sim S_1|S_0$ by Lemma \ref{l1}, it follows from the first paragraph by comparing dimensions that $M_3\sim X|M_1\sim S_3|S_2|M_1$.

(i) By \cite[Lemma 5.4(i)]{KST}, $M_3\cong M_1\ \oplus\  D_2|D_1| D_3|D_1|D_2$, and we are done by the first paragraph. 

(ii) Let $n\equiv 2\pmod{4}$. By \cite[Lemmas 5.4(ii), 5.5]{KST}, we 
have that:
\begin{enumerate}
\item[(a)] $\im \eta_{2,3}\cong D_0|D_2|D_0|D_1$;
\item[(b)] the composition factors of $M_3$ are $D_0$ with multiplicity $4$, $D_1$ with multiplicity $2$, $D_2$ with multiplicity $2$, and $D_3$ with multiplicity $1$;
\item[(c)] $\soc M_3\cong D_0$.
\item[(d)] $\im \eta_{1,3}$ is the unique submodule of $M_3$  isomorphic to $M_1$ and $\ker\eta_{3,1}$ is the unique submodule $N$ of $M_3$ such that $M_3/N \cong M_1$.  
\end{enumerate} 
Since $S_3\subseteq M_3$, the structure of $S_3$ follows from (c) and Lemma~\ref{L2R}. By (a),(d) and Lemma \ref{l1} there exist modules $B,C\subseteq M_3$ with $B\cong Y_2/D_1$ and $C\cong M_1$. Moreover, by Lemma \ref{l1}, we have $B/\soc M_3\cong S_2^*$, $C/\soc M_3\cong S_1^*$ and $B\cap C=\soc M_3$. So
\[M_3\sim B|(C/\soc M_3)|D\sim (Y_2/D_1)|S_1^*|D,\]
for a certain quotient $D$ of $M_3$. Since $M_3$ has a quotient of the form $S_3^*$ and $D_3\cong\soc(S_3^*)$ is not a composition factor of neither $B$ nor $C$, it follows that $D$ also has a quotient of the form $S_3^*$ and then by dimensions  $D\cong S_3^*$, so that $M_3\sim (Y_2/D_1)|S_1^*|S_3^*$.

By what has just been proved, $\soc(M_3/D_0)$ is isomorphic to a submodule of
\[\soc(Y_2/(D_1|D_0))\oplus \soc(S_1^*)\oplus \soc(S_3^*)\cong D_2\oplus D_1\oplus D_3.\]
In particular there exists a unique submodule of $M_3/D_0$ of the form $D_2$. So there exists a unique submodule $E\subseteq M_3$ with $E\cong D_0|D_2$. Then $E\subseteq S_3$. Let $A\subseteq B$ be the unique submodule with $A\cong D_0|D_2|D_0$. Again, we have $E\subseteq A$.  It follows that $A+S_3\sim A|D_3$ and $A+S_3\sim S_3|D_0$. 
Since $\soc S_2\cong D_1$ (from Lemma \ref{l1}), we have that
\[\overbrace{((A+S_3)/S_3)}^{D_0}\cap\overbrace{(X/S_3)}^{S_2}=0.\]
and then that
\[(A+X)/(A+S_3)\cong 
(A+S_3+X)/(A+S_3)\cong 
X/((A+S_3)\cap X)\cong 
X/S_3\cong S_2.\]
In particular, 
\[A+X
\sim (A+S_3)|S_2
\sim A|D_3|S_2.\]
Comparing composition factors we have that $M_3/(A+X)$ has composition factors $D_0$ and $D_1$ with multiplicity 1 and no other composition factors. Since $M_3/(A+X)$ is a quotient of
\[M_3/X\cong M_1\cong D_0|D_1|D_0\sim D_0|S_1^*,\]
it follows that $M_3/(A+X)\cong S_1^*$ and so
\[M_3\sim (A+X)/(M_3/(A+X))\sim A|D_3|S_2|S_1^*.\]
\end{proof}

\begin{Lemma}\label{l11}
Let $p=2$ and $n\geq 6$ with $n\equiv 2\pmod{4}$. Then $M^{(n-2,1,1)}\cong M_1\oplus Y^{(n-2,1,1)}$ with
\[Y^{(n-2,1,1)}\cong \underbrace{D_1|D_0}_{S_1^*}|\overbrace{D_2|\underbrace{D_0|D_1}_{S_1}}^{S_2^*}|D_0|\overbrace{D_2|\underbrace{D_0|D_1}_{S_1}}^{S_2^*}.\]
Further $Y_2$ is a submodule and a quotient of $Y^{(n-2,1,1)}$.
\end{Lemma}

\begin{proof}
Since $M^{(n-2,1)}\cong D^{(n-1)}\oplus D^{(n-2,1)}$, we have
$$M^{(n-2,1,1)}\cong D^{(n-1)}{\ua}^{\SSS_n}\,\oplus\,  D^{(n-2,1)}{\ua}^{\SSS_n}\cong M_1\,\oplus\, D^{(n-2,1)}{\ua}^{\SSS_n}.
$$
By \cite[Lemma 3.13]{M}, 
$$
D^{(n-2,1)}{\ua}^{\SSS_n}\cong D_1|D_0|D_2|D_0|D_1|D_0|D_2|D_0|D_1. 
$$
In particular, $D^{(n-2,1)}{\ua}^{\SSS_n}\cong Y^{(n-2,1,1)}$, see Lemma~\ref{LYoung}. The rest comes from \cite[Lemmas 3.5, 3.12]{M}.  
\end{proof}

\begin{Remark}
The following diagrams give information on the structures of $M_2$ and $M_3$ in the cases the structures were not completely determined, but will not be used in the proofs. Edges indicate existence of uniserial subquotients; see \cite{Alp,BC} for precise meaning of the pictures.
\begin{enumerate}
\item[{\rm (i)}] If $p=3$ and $n\equiv 0\pmod{9}$ then
$$
\begin{tikzpicture}
\node at (-1.5,0){$M_3\cong D_2\,\,\oplus$};
\node at (0,1) {$D_0$};
\node at (1,1) {$D_1$};
\node at (1,-0) {$D_3$};
\node at (1,-1) {$D_1$};
\node at (2,-1) {$D_0$};
\draw [-] (0,0.7) -- (0.7,-0.7);
\draw [-] (1,0.7) -- (1,0.3);
\draw [-] (1,-0.7) -- (1,-0.3);
\draw [-] (1.3,0.7) -- (2,-0.7);
\end{tikzpicture}
$$
\item[{\rm (ii)}] If $p=3$ and $n\equiv 3\pmod{9}$ then
\[M_3\cong D_0\oplus D_2\oplus(D_1|D_0|D_3|D_0|D_1).\]

\item[{\rm (iii)}] If $p=3$ and $n\equiv 6\pmod{9}$ then
\[\xymatrix@R=6pt@C=10pt{
&&D_1\ar@{-}[d]\ar@{-}[dl]\\
M_3\cong D_0\oplus D_2\oplus&D_0\ar@{-}[dr]&D_3\ar@{-}[d]\\
&&D_1
}\]
\item[{\rm (iv)}] If $p=3$ and $n\equiv 1\pmod{9}$ then
$$
\begin{tikzpicture}
\node at (-1.5,0){$M_3\cong D_1\,\,\oplus$};
\node at (0,1) {$D_0$};
\node at (1,1) {$D_2$};
\node at (1,-0) {$D_3$};
\node at (1,-1) {$D_2$};
\node at (2,-1) {$D_0$};
\draw [-] (0,0.7) -- (0.7,-0.7);
\draw [-] (1,0.7) -- (1,0.3);
\draw [-] (1,-0.7) -- (1,-0.3);
\draw [-] (1.3,0.7) -- (2,-0.7);
\end{tikzpicture}
$$

\item[{\rm (v)}] If $p=3$ and $n\equiv 4\pmod{9}$ then
\[M_3\cong D_0\oplus D_1\oplus (D_2|D_0|D_3|D_0|D_2).\]

\item[{\rm (vi)}] If $p=3$ and $n\equiv 7\pmod{9}$ then
\[\xymatrix@R=6pt@C=10pt{
&&D_2\ar@{-}[d]\ar@{-}[dl]\\
M_3\cong D_0\oplus D_1\oplus &D_0\ar@{-}[dr]&D_3\ar@{-}[d]\\
&&D_2
}\]
\item[{\rm (vii)}] If $p=2$ and $n\equiv 0\pmod{4}$ then
$$
\begin{tikzpicture}
\node at (-1.2,0){$M_2\cong $};
\node at (0,1) {$D_0$};
\node at (1,1) {$D_1$};
\node at (1,-0) {$D_2$};
\node at (1,-1) {$D_1$};
\node at (2,-1) {$D_0$};
\draw [-] (0,0.7) -- (0.7,-0.7);
\draw [-] (1,0.7) -- (1,0.3);
\draw [-] (1,-0.7) -- (1,-0.3);
\draw [-] (1.3,0.7) -- (2,-0.7);
\end{tikzpicture}
$$
\item[{\rm (vii)}] If $p=2$ and $n\equiv 2\pmod{4}$ then
\[\xymatrix@R=6pt@C=10pt{
&&&&D_0\ar@{-}[dl]\ar@{-}[dr]&&\\
&&&D_1\ar@{-}[dl]&&D_2\ar@{-}[dl]\ar@{-}[dr]&\\
M_3&\cong &D_0\ar@{-}[dr]&&D_3\ar@{-}[dl]&&D_0\ar@{-}[dl]&.\\
&&&D_2\ar@{-}[dr]&&D_1\ar@{-}[dl]&\\
&&&&D_0&&
}\]
\end{enumerate}
\end{Remark}

\vspace{2mm}

\section{Results on the module $\EE(\la)$}
\label{SP2NEVEN}

In this section we study the submodule structure of the module $$\EE(\la)=\End_F(D^\la)\cong D^\la\otimes D^\la.$$ We try to show that some quotients of small permutation modules $M_k$ arise as submodules of $\EE(\la)$, which is needed to obtain homomorphisms $\psi$ as in Lemma~\ref{LBasic2}.

\begin{Lemma}\label{l7}
Let $p=2$,  $n\geq 6$ be even, and let $\la\in\Par_2(n)$ be not a JS-partition. Then $S_1^*\subseteq\EE(\la)$.
\end{Lemma}

\begin{proof}
It suffices to prove that $\dim\Hom_{\SSS_n}(S_1^*,\EE(\la))\geq 2$ since $S_1^*\cong D_1|D_0$ by Lemma \ref{l1} and $D_0\cong \bone_{\SSS_n}$ is contained exactly once in the socle of $\EE(\la)$ by Schur's Lemma. On the other hand, 
\[\Hom_{\SSS_n}(S_1^*,\EE(\la))\cong 
\Hom_{\SSS_n}(S_1^*,(D^\la)^*\otimes D^\la)\cong 
\Hom_{\SSS_n}(D^\la\otimes S_1^*,D^\la).\]
So it is enough to prove that $$\dim \Hom_{\SSS_n}(D^\la\otimes S_1^*,D^\la)\geq 2.$$

We have a commutative diagram
\[\xymatrix@R=12pt@C=12pt{
0\ar@{->}[r]&D_0\ar@{->}[r]&M_1\ar@{->}[r]&S_1^*\ar@{->}[r]&0\\
0\ar@{->}[r]&D_0\ar@{->}^{\id}[u]\ar@{->}[r]&S_1\ar@{->}[r]\ar@{->}[u]&D_1\ar@{->}[r]\ar@{->}[u]&0\\
&&0\ar@{->}[u]&0\ar@{->}[u]
}\]
whose rows and columns are exact. By tensoring with $D^\la$ we get a commutative diagram
\begin{equation}\label{E230418}
\xymatrix@R=12pt@C=12pt{
0\ar@{->}[r]&D^\la\ar@{->}^{\hspace{-15pt}\iota}[r]&D^\la\otimes M_1\ar@{->}[r]&D^\la\otimes S_1^*\ar@{->}[r]&0\\
0\ar@{->}[r]&D^\la\ar@{->}^{\id}[u]\ar@{->}[r]&D^\la\otimes S_1\ar@{->}[r]\ar@{->}[u]&D^\la\otimes D_1\ar@{->}[r]\ar@{->}[u]&0\\
&&0\ar@{->}[u]&0\ar@{->}[u]
}
\end{equation}
whose rows and columns are exact. 

Applying $\Hom_{\SSS_n}(-,D^\la)$ to the short exact sequence in the first row of (\ref{E230418}) and using the fact that $\Hom_{\SSS_n}(D^\la,D^\la)\cong \F$ by Schur's Lemma, we get an exact sequence
\begin{equation}\label{ESES}
0\to \Hom_{\SSS_n}(D^\la\otimes S_1^*,D^\la)\to 
\Hom_{\SSS_n}(D^\la\otimes M_1,D^\la)
\stackrel{\pi}{\longrightarrow} \F. 
\end{equation}
Furthermore, by Lemma~\ref{Lemma39}(v), we have 
\begin{align*}
 \dim \Hom_{\SSS_n}(D^\la\otimes M_1, D^\la)&=\dim  \Hom_{\SSS_n}(M_1,\End_\F(D^\la))\\
 &=\dim \End_{\SSS_{n-1}}(D^\la\da_{\SSS_{n-1}})
 \\
 &=\eps_0(\la)+\eps_1(\la),
\end{align*}
which is just the number of normal nodes in $\la$. 
By assumption, $\la$ has at least two normal nodes. If it has three, we are now done. Moreover, if $\pi$ is the zero map, we are also done. So we may assume that $\la$ has two normal nodes and $\pi\neq0$. We will show that this leads to a contradiction. 

Since $\pi\neq 0$, there exists a homomorphism $\phi\in \Hom_{\SSS_n}(D^\la\otimes M_1,D^\la)$ with $\pi( \phi)=\phi\circ \iota=\id_{D^\la}$, i.e. the short exact sequence in the first row of (\ref{E230418}) splits. Hence the short exact sequence in the second row of (\ref{E230418}) splits. 

By the the splitting of the first row of (\ref{E230418}), we have 
$$D^\la\otimes M_1\cong D^\la\oplus (D^\la\otimes S_1^*).$$ 
Moreover, by Lemma~\ref{Lemma45}, we have 
\[D^\la\otimes M_1\cong D^\la\da_{\SSS_{n-1}}\ua^{\SSS_n}\cong f_0e_0D^\la\oplus f_1e_1D^\la\oplus f_0e_1D^\la\oplus f_1e_0D^\la.\]
So by Lemma~\ref{L230418}, 
$$
f_0e_0D^\la\oplus f_1e_1D^\la\cong D^\la\oplus X
$$
where $X$ is a self-dual module with socle and head both  isomorphic to $D^\la$ and $[X:D^\la]\geq 2$. Using the Krull-Schmidt Theorem, we deduce that 
$$D^\la\otimes S_1^*\cong X\oplus f_0e_1D^\la\oplus f_1e_0D^\la.$$ 
By dualizing, it follows that 
$$D^\la\otimes S_1\cong X\oplus f_0e_1D^\la\oplus f_1e_0D^\la.
$$ 
But by the splitting of the second row of (\ref{E230418}), we know that $D^\la$ is a direct summand of $D^\la\otimes S_1$ which leads to a contradiction by the structure of $X$ and the fact that $f_0e_1D^\la\oplus f_1e_0D^\la$ is in blocks different from that of $D^\la$. 
\end{proof}

Recall the numbers $m_k(\la)$ from (\ref{EMtoE}). 

\begin{Lemma}\label{l13}
Let $p=2$, $n\geq 6$ be even and $\la\in\Par_2(n)$ have at least three normal nodes. Then
\[\dim\Hom_{\SSS_n}(M^{(n-2,1,1)},\EE(\la))>2m_1(\la)+2\dim\Hom_{\SSS_n}(S_1,\EE(\la))+1.\]
\end{Lemma}

\begin{proof}
In this proof we denote $\eps_i:=\eps_i(\la)$, $\phi_i:=\phi_i(\la)$, and $h:=h(\la)$. Note that the left hand side of the inequality in the lemma equals $\dim\End_{\SSS_{n-2}}(D^\la\da_{\SSS_{n-2}})$, which by \cite[Lemma 4.9]{M} is bounded below by $$2\vare_0(\vare_0-1)+2\vare_1(\vare_1-1)+2\delta_{\vare_0,\vare_1\geq 1}(\vare_0+\vare_1+\vare_0\vare_1).
$$ 
On the other hand, by Lemma \ref{l12}, we have $\dim\Hom_{\SSS_n}(S_1,\EE(\la))\leq \eps_0+\eps_1$, while by Lemma \ref{Lemma39}(v) we have $m_1(\la)=\eps_0+\eps_1$. So it suffices to prove that 
\[2\vare_0(\vare_0-1)+2\vare_1(\vare_1-1)+2\delta_{\vare_0,\vare_1\geq 1}(\vare_0+\vare_1+\vare_0\vare_1)>4(\vare_0+\vare_1)+1.\]

By the assumption that $\la$ has at least three normal nodes, we have $\eps_0+\eps_1\geq 3$. 
If either $\vare_i\geq 2$ and $\vare_{1-i}\geq 1$ 
or $\eps_i\geq 4$ and $\eps_{1-i}=0$ for some $i\in I$  
then the above inequality holds. Thus, we are left with the case where $\eps_i= 3$ and $\eps_{1-i}=0$ for some $i\in I$, which we assume from now on. 

By Lemmas~\ref{Lemma45} and \ref{Lemma39}, we have that $m_1(\la)=3$ and
\begin{align*}
\dim\Hom_{\SSS_n}(M^{(n-2,1,1)},\EE(\la))&=\dim\End_{\SSS_{n-2}}(D^\la\da_{S_{n-2}})\\
&=\dim\End_{\SSS_{n-2}}(e_i^2D^\la)+\dim\End_{\SSS_{n-2}}(e_{1-i}e_iD^\la)\\
&=12+\dim\End_{\SSS_{n-2}}(e_{1-i}e_iD^\la).
\end{align*}
By Lemmas \ref{l12} and \ref{l20}, if $\phi_i=0$, then $\dim\Hom_{\SSS_n}(S_1,\EE(\la))\leq 2$ and so in this case the lemma holds. So we may assume that $\phi_i>0$. 
If $e_{1-i}e_iD^\la$ is non-zero and not simple then by self-duality, $\dim\End_{\SSS_{n-2}}(e_{1-i}e_iD^\la)\geq 2$, and so in this case the lemma holds again by Lemma \ref{l12}. 
So we will complete the proof by establishing the following

\vspace{2mm}
\noindent
{\sf Claim.} If $\eps_i=3$, $\eps_{1-i}=0$ and $\phi_i>0$ then $e_{1-i}e_iD^\la$ is non-zero and not simple.

\vspace{2mm}

Notice that $h\geq 3$ since $\la$ has 3 normal nodes. Also, since the top removable node $A=(1,\la_1)$  is always normal, it has residue $i$. Below we will repeatedly use Lemma~\ref{Lemma39} without further notice. 

{\sf Case 1}. $\la_1\equiv\la_2\pmod{2}$. Then $\la_1\geq \la_2+2$ and $(2,\la_2)$ has residue $1-i$. Since $\la_1\geq \la_2+2$, the partition $\la_A$ is $2$-regular. Further the two top removable nodes of $\la_A$ are $(1,\la_1-1)$ and $(2,\la_2)$ which both have residue $1-i$ and then they are both normal in $\la_A$. Therefore $e_{1-i}e_iD^\la$ is non-zero and not simple.

{\sf Case 2}. $\la_1\not\equiv\la_2\equiv\la_3\pmod{2}$. We have that $B:=(2,\la_2)$ is $i$-normal for $\la$,  $\la_B$ is $2$-regular, $[e_iD^\la:D^{\la_B}]=2$, and $(3,\la_3)$ is normal of residue $1-i$ in $\la_B$. Hence $e_{1-i}e_iD^\la$ is non-zero and not simple.

{\sf Case 3}.  $\la_1\not\equiv\la_2\not\equiv\la_3\pmod{2}$. In this case $(1,\la_1)$, $(2,\la_2)$ and $C:=(3,\la_3)$ are exactly the $i$-normal nodes of $\la$, and $C$ is the $i$-good node of $\la$.

{\sf Case 3.1}.  $h=3$. As $n$ is even, we must have that $\la_1$ and $\la_3$ are odd and $\la_2$ is even. So $i=0$. In this case all addable nodes for $\la$ also have residue $1$, so $\phi_i=0$, which contradicts the assumptions of the claim. 

{\sf Case 3.2}. $h\geq 4$. Then $\la_4\equiv\la_3\pmod{2}$, since otherwise $(4,\la_4)$ would also have residue $i$ and then it would also be normal. Now, since $\la_1+\la_2+\la_3+\la_4$ is odd, we must have $h\geq 5$. If $\la_C$ has a normal node of residue $1-i$, then $e_{1-i}e_iD^\la$ is non-zero and not simple. So we may  assume that $\vare_{1-i}(\la_C)=0$. On the other hand, $\vare_i(\la_C)=2$. So $\la_C$ has exactly two normal nodes. 
For $1\leq k\leq h$ let $a_k$ be the residue of the removable node on the $k$-th row of $\la_C$ and let $1<b_1<\ldots<b_t$ be the set of indices $k$ for which $a_k=a_{k-1}$. Note that $b_1=2$ and $b_2=4$. 

{\sf Case 3.2.1}. $t=2$. In this case $((\la_C)_4,\ldots,(\la_C)_{h})=(\la_4,\ldots,\la_{h})$ is a JS-partition. So the only conormal nodes for $\la$ on row 4 or below are the two bottom addable nodes $(h,\la_{h}+1)$ and $(h+1,1)$. Since $\la_1+\la_2+\la_3+\la_4$ is odd and $n$ is even $\la_5+\ldots+\la_{h}$ is odd and then, since $(\la_5,\ldots,\la_{h})$ is also a JS-partition, $h$ and $\la_{h}$ are both odd. From
\[\la_1\not\equiv\la_2\not\equiv\la_3\equiv\la_4\equiv\ldots\equiv\la_{h}\pmod{2}\]
it follows that $\la_1$ is odd and so $i=0$. So the nodes $(h,\la_{h}+1)$ and $(h+1,1)$ both have residue $1$, as have the addable nodes for $\la$ in the first three rows. In particular $\phi_i=0$ giving a contradiction.

{\sf Case 3.2.2}.  $t\geq 3$. By Lemma~\ref{L6.1},  $a_{b_3}\not\equiv a_{b_2}= a_4\equiv 1-i\pmod{2}$, so  $a_{b_3}=i$. By definition of $a_j$, the sequence of residues of the removable nodes of $\la$ in its first $b_3$ rows is given by
\[(a_1,a_2,1-a_3,a_4,\ldots,a_{b_3})=(i,i,i,\overbrace{1-i,i,\ldots,1-i,i}^{1-i\text{ and }i\text{ alternate}},i).\]
By the definition of normal nodes, we then have that $(b_3,\la_{b_3})$ is normal in $\la$, contradicting the assumption that $\la$ has only 3 normal nodes.
\end{proof}

\begin{Lemma}\label{l21}
Let $p=2$, $n\geq 6$ be even and $\la\in\Par_2(n)$ have exactly two normal nodes. 
Then 
$m_2(\la)>m_1(\la)+1=3$ 
and
\[\dim\Hom_{\SSS_n}(M^{(n-2,1,1)},\EE(\la))>m_1(\la)+3.\]
\end{Lemma}

\begin{proof}
By Lemma~\ref{Lemma39} and the assumption that $\la$ has exactly two normal nodes, we have $m_1(\la)=2$, hence the equalities in the lemma. 

{\sf Case 1.} $\vare_i(\la)=2$ and $\vare_{1-i}(\la)=0$ for some $i\in I$. Then by Lemmas~\ref{Lemma45} and \ref{Lemma39}, we have 
\[D^\la\da_{S_{n-2}}\cong e_i^2D^\la\oplus e_{1-i}e_iD^\la\]
and $e_i^2D^\la$ and $e_{1-i}e_iD^\la$ are in different blocks of $\SSS_{n-2}$. Hence we can write 
\[D^\la\da_{S_{n-2,2}}\cong E_{i,i}\oplus E_{1-i,i},\]
where $E_{i,i}\da_{S_{n-2}}\cong e_i^2D^\la$, $E_{1-i,i}\da_{S_{n-2}}\cong e_{1-i}e_iD^\la$, and $E_{i,i}$ and $E_{1-i,i}$ are in different blocks of $\SSS_{n-2,2}$. We deduce that $E_{i,i}$ and $E_{1-i,i}$ are self-dual. 

By Lemma~\ref{Lemma39}, we have $e_i^2D^\la\cong D^{\tilde{e}_i^2 \la}\oplus D^{\tilde{e}_i^2\la}$ and by  \cite[Lemma 6.4]{M} we have that $e_{1-i}e_iD^\la$ is non-zero and not simple. So $E_{i,i}$ and $E_{1-i,i}$ are both non-zero and not simple, since all simple $\F\SSS_2$-modules are 1-dimensional. 
Using self-duality of the modules involved, we now get
\begin{align*}
m_2(\la)&=\dim\End_{\SSS_{n-2,2}}(D^\la\da_{\SSS_{n-2,2}})\\
&=\dim\End_{\SSS_{n-2,2}}(E_{i,i})+\dim\End_{\SSS_{n-2,2}}(E_{1-i,i})\\
&\geq 2+2
\end{align*}
and
\begin{align*}
\dim\Hom_{\SSS_n}(M^{(n-2,1,1)},\EE(\la))&=\dim\End_{\SSS_{n-2}}(D^\la\da_{\SSS_{n-2}})\\
&=\dim\End_{\SSS_{n-2}}(e_i^2D^\la)+\dim\End_{\SSS_{n-2}}(e_{1-i}e_iD^\la)\\
&\geq 4+2.
\end{align*}

{\sf Case 2.} $\vare_0(\la)=\vare_1(\la)=1$. Then 
by Lemmas~\ref{Lemma45} and \ref{Lemma39}, we have 
\[D^\la\da_{S_{n-2}}\cong e_0e_1D^\la\oplus e_1e_0D^\la\cong 
e_0D^{\tilde e_1\la}\oplus e_1D^{\tilde e_0\la}.\] 
So we have  
\begin{align*}
\Hom_{\SSS_n}(M^{(n-2,1,1)},\EE(\la))
\cong\ 
&\End_{\SSS_{n-2}}(D^\la\da_{\SSS_{n-2}})
\\
\cong\ 
&\End_{\SSS_{n-2}}(e_0D^{\tilde e_1\la})\oplus\End_{\SSS_{n-2}}(e_1D^{\tilde e_0\la})
\\
&\oplus \Hom_{\SSS_{n-2}}(e_1D^{\tilde e_0\la},e_0D^{\tilde e_1\la})
\\
&\oplus \Hom_{\SSS_{n-2}}(e_0D^{\tilde e_1\la},e_1D^{\tilde e_0\la}).
\end{align*}
By \cite[Lemma 4.8]{M}, the last two Hom-spaces are non-zero, while by Lemma~\ref{Lemma39}, we have
$$
 \dim\End_{\SSS_{n-2}}(e_0D^{\tilde e_1\la})=\eps_0(\tilde e_1\la)\quad \text{and}\quad 
\dim\End_{\SSS_{n-2}}(e_1D^{\tilde e_0\la})=\eps_1(\tilde e_0\la).
$$
Moreover, by \cite[Lemma 4.4]{M}, $\eps_0(\tilde e_1\la)+\eps_1(\tilde e_0\la)\geq 4$. 
Therefore 
$$
\dim\Hom_{\SSS_n}(M^{(n-2,1,1)},\EE(\la))\geq \eps_0(\tilde e_1\la)+\eps_1(\tilde e_0\la)+2\geq  6,
$$
as required. 

By Lemma~\ref{L4.14}, we further have 
\[2m_2(\la)\geq \dim\Hom_{\SSS_n}(M^{(n-2,1,1)},\EE(\la))\geq \eps_0(\tilde e_1\la)+\eps_1(\tilde e_0\la)+2.\]
So if $\eps_0(\tilde e_1\la)+\eps_1(\tilde e_0\la)>4$, the inequality $m_2(\la)>3$ also follows. Thus we may assume that $\eps_0(\tilde e_1\la)+\eps_1(\tilde e_0\la)=4$.

Let $i:=\res(1,\la_1)$. Then $(1,\la_1)$ is the only $i$-normal node of $\la$. By \cite[Lemma 4.4]{M}, we have 
$
\eps_{1-i}(\tilde e_i\la)=3.
$
So $\eps_i(\tilde e_{1-i}\la)=1$. 
Therefore $e_iD^{\tilde e_{1-i}\la}\cong D^{\tilde e_i\tilde e_{1-i}\la}$, thanks to Lemma~\ref{Lemma39}. On the other hand, as we have pointed out above,
$$
\Hom_{\SSS_{n-2}}(e_iD^{\tilde e_{1-i}\la},e_{1-i}D^{\tilde e_{i}\la})\neq 0,
$$
hence $\tilde e_i\tilde e_{1-i}\la=\tilde e_{1-i}\tilde e_{i}\la$ again by Lemma~\ref{Lemma39}. Set $\mu:=\tilde e_i\tilde e_{1-i}\la$.

Notice that
\[(\soc(D^\la\da_{\SSS_{n-2,2}}))\da_{\SSS_{n-2}}\subseteq\soc(D^\la\da_{\SSS_{n-2}})\cong \soc(e_iD^{\tilde e_{1-i}\la}\oplus e_{1-i}D^{\tilde e_{i}\la})\cong D^\mu\oplus D^\mu.\]
Hence either $\soc(D^\la\da_{\SSS_{n-2,2}})\cong D^\mu\boxtimes \bone_{\SSS_2}$ or $\soc(D^\la\da_{\SSS_{n-2,2}})\cong (D^\mu\boxtimes \bone_{\SSS_2})^{\oplus 2}$. 
In the latter case, we have by self-duality that 
$$
m_2(\la)=\dim\End_{\SSS_{n-2,2}}(D^\la\da_{\SSS_{n-2,2}})\geq 4,
$$
as desired. So we may assume that $\soc(D^\la\da_{\SSS_{n-2,2}})\cong D^\mu\boxtimes \bone_{\SSS_2}$. 

By Lemma \ref{l8}, $e_{1-i}D^{\tilde{e}_i \la}$ has a self-dual quotient $V$ with $[V:D^{\mu}]=2$ and $\soc V\cong \head V\cong D^\mu$. In particular, $\dim\End_{\SSS_{n-2}}(V)=2$. 
Writing $\F\SSS_2$ for the regular module over $\SSS_2$, we have 
\begin{align*}
\Hom_{\SSS_{n-2,2}}(D^\la\da_{S_{n-2,2}},V\boxtimes \F\SSS_2)
&\cong \Hom_{\SSS_{n-2,2}}(D^\la\da_{S_{n-2,2}},V\ua^{\SSS_{n-2,2}})
\\
&\cong \Hom_{\SSS_{n-2}}(D^\la\da_{S_{n-2}},V)\\
&\cong \Hom_{\SSS_{n-2}}(e_iD^{\tilde e_{1-i}\la},V)\oplus\Hom_{\SSS_{n-2}}(e_{1-i}D^{\tilde e_{i}\la},V).
\end{align*}
Since $e_iD^{\tilde e_{1-i}\la}\cong D^\mu$, we have 
$
\dim\Hom_{\SSS_{n-2}}(e_iD^{\tilde e_{1-i}\la},V)=1.
$
Since $V$ is a quotient of $e_{1-i}D^{\tilde e_{i}\la}$ and $\dim\End_{\SSS_{n-2}}(V)=2$, we have $\dim \Hom_{\SSS_{n-2}}(e_{1-i}D^{\tilde e_{i}\la},V)\geq 2$. So 
$$\dim\Hom_{\SSS_{n-2,2}}(D^\la\da_{S_{n-2,2}},V\boxtimes \F\SSS_2)\geq 3.$$ 
Since 
$V\boxtimes \F\SSS_2 \sim (V\boxtimes \bone_{\SSS_2})|(V\boxtimes \bone_{\SSS_2})
$
 it follows that
\begin{equation}\label{E240418}
\dim\Hom_{\SSS_{n-2,2}}(D^\la\da_{S_{n-2,2}},V\boxtimes \bone_{\SSS_2})\geq 2.
\end{equation}
A similar argument with $D^\mu$ in place of $V$ shows that  
\begin{equation}\label{E240418_1}
\dim\Hom_{\SSS_{n-2,2}}(D^\la\da_{S_{n-2,2}},D^{\mu}\boxtimes \F\SSS_2)=2.
\end{equation}

Since $\head(D^\la\da_{\SSS_{n-2,2}})\cong D^\mu\boxtimes \bone_{\SSS_2}$,  $\head(V\boxtimes \bone_{\SSS_2})\cong D^\mu\boxtimes \bone_{\SSS_2}$ and $[V\boxtimes \bone_{\SSS_2}: D^\mu\boxtimes \bone_{\SSS_2}]=2$, we conclude from (\ref{E240418}) that $V\boxtimes \bone_{\SSS_2}$ is a quotient of $D^\la\da_{\SSS_{n-2,2}}$. By self-duality, $V\boxtimes \bone_{\SSS_2}$ is also a submodule of $D^\la\da_{\SSS_{n-2,2}}$. A similar argument using (\ref{E240418_1}) instead of (\ref{E240418}), shows that $D^{\mu}\boxtimes \F\SSS_2$ is a quotient and a submodule of $D^\la\da_{\SSS_{n-2,2}}$. Therefore there exist endomorphisms
$\psi_2,\psi_3\in \End_{\SSS_{n-2,2}}(D^\la\da_{\SSS_{n-2,2}})$ with $\im \psi_2\cong D^{\mu}\boxtimes \F\SSS_2$ and $\im\psi_3\cong V\boxtimes \bone_{\SSS_2}$. Let us also define $\psi_4:=\id\in \End_{\SSS_{n-2,2}}(D^\la\da_{\SSS_{n-2,2}})$ and $\psi_1\in \End_{\SSS_{n-2,2}}(D^\la\da_{\SSS_{n-2,2}})$ to be a homomorphism with $\im \psi_1\cong D^\mu\boxtimes \bone_{\SSS_2}$. Note that $D^{\mu}\boxtimes \F\SSS_2\not\cong V\boxtimes \bone_{\SSS_2}$, so $\im\psi_2\neq \im\psi_3$, $\im\psi_1\subseteq \im\psi_2\cap  \im\psi_3$, and $\im\psi_2+\im\psi_3\subsetneq \im\psi_4$. These facts easily imply that $\psi_1,\psi_2,\psi_3,\psi_4$ are linearly independent, completing the proof of $m_2(\la)\geq 4$. 
\end{proof}

\begin{Lemma}\label{l15}
Let $p=2$, $n\geq 8$ with $n\equiv 0\pmod{4}$ and $\la\in\Par_2(n)$ with $\la\not\in\{(n),\be_n\}$. If $\vare_0(\la)+\vare_1(\la)=2$ assume further that $\dim\Hom_{\SSS_n}(S_1,\EE(\la))<2$. Then $D_2\subseteq \EE(\la)$.
\end{Lemma}

\begin{proof}
If $\la$ is a JS partition this holds by \cite[Lemma 7.5]{M}. So we may assume that $\eps_0(\la)+\eps_1(\la)\geq 2$.

Since $S_1^*$ is a quotient of $M_1$, we have that 
$$\dim\Hom_{\SSS_n}(S_1^*,\EE(\la))\leq \dim\Hom_{\SSS_n}(M_1,\EE(\la))= m_1(\la).$$
If $\vare_0(\la)+\vare_1(\la)\geq 3$ then, by Lemmas~\ref{L4.14} and \ref{l13}, we have 
\begin{align*}
m_2(\la)&=\dim\End_{\SSS_{n-2,2}}(D^\la\da_{\SSS_{n-2,2}})
\\&\geq 
(\dim\End_{\SSS_{n-2}}(D^\la\da_{\SSS_{n-2}}))/2
\\&=
(\dim\Hom_{\SSS_{n}}(M^{(n-2,1,1)},\EE(\la)))/2
\\&
> m_1(\la)+\dim\Hom_{\SSS_n}(S_1,\EE(\la))
\\
&\geq \dim\Hom_{\SSS_n}(S_1^*,\EE(\la))+\dim\Hom_{\SSS_n}(S_1,\EE(\la)).
\end{align*}
On the other hand, if $\vare_0(\la)+\vare_1(\la)=2$ and $\dim\Hom(S_1,\EE(\la))<2$, then by Lemma \ref{l21}, we get
\begin{align*}
m_2(\la)>m_1(\la)+1\geq \dim\Hom_{\SSS_n}(S_1^*,\EE(\la))+\dim\Hom_{\SSS_n}(S_1,\EE(\la)).
\end{align*}
By Lemma \ref{l10}, we have $M_2\sim S_1^*|D_2|S_1$, so the inequality 
$$
\dim\Hom_{\SSS_n}(M_2,\EE(\la))=m_2(\la)>\dim\Hom_{\SSS_n}(S_1^*,\EE(\la))+\dim\Hom_{\SSS_n}(S_1,\EE(\la))
$$
implies that $\dim\Hom_{\SSS_n}(D_2,\EE(\la))>0$, which yields  the lemma. 
\end{proof}

\begin{Lemma}\label{l16}
Let $p=2$, $n\geq 6$ with $n\equiv 2\pmod{4}$ and $\la\in\Par_2(n)$. If $\eps_0(\la)+\eps_1(\la)\geq 3$  then $S_2^*\subseteq\EE(\la)$.
\end{Lemma}

\begin{proof}
From Lemma \ref{l11} it is enough to prove that
\begin{align*}\dim\Hom_{\SSS_n}(M^{(n-2,1,1)},\EE(\la))>&\,\,m_1(\la)+2\dim\Hom_{\SSS_n}(S_1,\EE(\la))
\\
&+\dim\Hom_{\SSS_n}(S_1^*,\EE(\la))+1.
\end{align*}
This follows from Lemma \ref{l13} since 
$\dim\Hom_{\SSS_n}(S_1^*,\EE(\la))\leq m_1(\la)$.
\end{proof}

\begin{Lemma}\label{l17}
Let $p=2$, $n\geq 6$ with $n\equiv 2\pmod{4}$ and $\la\in\Par_2(n)$. Assume that $\la\not\in\{(n),\be_n\}$ is a JS-partition or that $\vare_0(\la)+\vare_1(\la)=2$ and $\dim\Hom_{\SSS_n}(S_1,\EE(\la))<2$. Then $S_2^*$ or $Y_2/D_1$ is contained in $\EE(\la)$.
\end{Lemma}

\begin{proof}
If $\la$ is a JS-partition with $\la\not\in\{(n),\be_n\}$, this holds by  \cite[Lemma 7.4]{M} and Lemmas~\ref{l10},\ref{l11} since $D^{(n-2,1)}\ua^{\SSS_n}\cong Y^{(n-2,1,1)}$. 

If $\vare_0(\la)+\vare_1(\la)=2$ and $\dim\Hom_{\SSS_n}(S_1,\EE(\la))<2$, then by Lemma \ref{l21} we have
\[\dim\Hom_{\SSS_n}(M^{(n-2,1,1)},\EE(\la))>m_1(\la)+3\dim\Hom_{\SSS_n}(S_1,\EE(\la)).\]
Since $D_1$ is a quotient of $S_1$, from Lemma \ref{l11} we then also have that
\[\dim\Hom_{\SSS_n}(Y^{(n-2,1,1)},\EE(\la))>2\dim\Hom_{\SSS_n}(S_1,\EE(\la))+\dim\Hom_{\SSS_n}(D_1,\EE(\la)).\]
From Lemmas \ref{l1} and \ref{l11} we also have that
\[Y^{(n-2,1,1)}\cong D_1|\overbrace{D_0|\overbrace{D_2|\underbrace{D_0|D_1}_{S_1}}^{S_2^*}}^{Y_2/D_1}|\overbrace{D_0|\overbrace{D_2|\underbrace{D_0|D_1}_{S_1}}^{S_2^*}}^{Y_2/D_1},\]
from which the lemma follows.
\end{proof}

\section{Special homomorphisms $M_k\to \EE(\la)$}
\label{SSpecialHoms}

\subsection{The homomorphism $\zeta_k$}
Let $1\leq k\leq n/2$ and $J\in\Om_k$. We denote by $\SSS_J$ the subgroup of $\SSS_n$ consisting of all permutations fixing the elements of $\{1,\dots,n\}\setminus J$. Clearly $\SSS_J\cong \SSS_k$. 

Let $\la\in\Parp(n)$. 
Recalling that $M_k$ denotes the permutation module on $\Om_k$, we define the homomorphism $\zeta_k\in\Hom_{\SSS_n}(M_k,\EE(\la))$ via 
$$
\big(\zeta_k(J)\big)(v)=\sum_{g\in\SSS_J} gv\qquad(J\in\Om_k,\ v\in D^\la). 
$$
Let $t$ be the $(n-k,k)$-tableau
$$
\begin{matrix}
k+1 & k+2& \cdots & 2k & 2k+1 & \cdots & n  \\
1    &    2  & \cdots & k &   &   &  
\end{matrix}
$$
and $C_t$ be the column stabilizer of $t$. Recalling (\ref{EPolyTab}), the corresponding polytabloid 
$$
e_t=\sum_{\si\in C_t}(\sgn\, \si)\, \si\cdot\{1,\dots,k\}\in M_k
$$
generates the submodule $S_k\subseteq M_k$. Define
$$
x_k:=\sum_{g\in\SSS_k,\,\si\in C_t} (\sgn\, \si)\si g\si^{-1}\in \F \SSS_n.
$$
Note that actually $x_k\in\F\SSS_{\{1,\dots,2k\}}\leq \F\SSS_{n}$. 
It follows from the definitions that for any $v\in D^\la$ we have 
$$
\big(\zeta_k(e_t)\big)(v)=x_kv,
$$
so

\begin{Lemma} \label{L150817} 
The homomorphism $\zeta_k$ is zero on the submodule $S_k\subset M_k$ if and only if $x_k D^\la=0$. 
\end{Lemma}

The elements $x_2$ and $x_3$ will play a special role, so we will spell them out explicitly. We have
$$
x_2=(1,2)-(1,4)-(2,3)+(3,4).
$$
For distinct $a,b,c\in\{1,\dots,n\}$, we consider the sum of $3$-cycles
$$
[abc]:=(a,b,c)+(a,c,b)\in \F \SSS_n.
$$
Then it is easy to see that, after some cancellation, we get
$$
x_3=[123]-[234]-[135]-[126]+[345]+[246]+[156]-[456].
$$

\subsection{The case $k=2$ and $p=2$}

\begin{Lemma} \label{L180418} 
Let $p=2$. Then $x_2D^{(4,1)}\neq 0$ and $x_2D^{(3,2,1)}\neq 0$. 
\end{Lemma}
\begin{proof}
We have $D^{(4,1)}=S^{(4,1)}$, so the module has a basis $\{\eps_r+\eps_{r+1}\mid r=1,\dots, 4\}$ with the action of $\SSS_5$ on the indices. An easy computation now shows that $x_2(\eps_1+\eps_2)=\eps_3+\eps_4\neq 0$. 

We also have $D^{(3,2,1)}=S^{(3,2,1)}$. 
Recalling (\ref{EPolyTab}), we realize $S^{(3,2,1)}$ as a submodule of $M^{(3,2,1)}$ spanned by polytabloids. 
For distinct $a,b,c\in\{1,\dots,6\}$, the tabloid corresponding to $a,b$ in the second row and $c$ in the third row will be denoted $ab|c$. Thus $ab|c=ba|c$, and 
$$
\{ab|c\,\mid\, \text{$a,b,c\in\{1,\dots,6\}$ are distinct and $a<b$}\}
$$ 
is a basis of $M^{(3,2,1)}$. Consider the $(3,2,1)$-tableau
$$
t=
\begin{matrix}
1 & 4 & 6  \\
2 &  5 &  \\
3 & &  
\end{matrix}
$$
and the corresponding polytabloid 
\begin{align*}
e_t=&\,25|3+35|2+15|3+15|2+35|1+25|1+24|3
\\&+34|2+14|3+14|2+34|1+24|1
\end{align*}
(since $p=2$ we ignore the signs). Now an explicit calculation shows that the basis element $12|3$ appears in $x_2e_t$ with coefficient $1$, in particular, $x_2e_t\neq 0$.
\end{proof}

\begin{Lemma} \label{L180418_2} 
Let $p=2$, $n\geq 5$, and $\la\in\Par_2(n)$ with $\la\not\in\{(n),\be_n\}$. Then $x_2 D^\la\neq 0$.  
\end{Lemma}
\begin{proof}
We apply induction on $n$. If $n=5$, the only $\la$ that satisfies the assumptions is $(4,1)$, and we can apply Lemma~\ref{L180418}.  If $n=6$, the only partitions that we have to check are $(5,1)$ and $(3,2,1)$. For $(3,2,1)$ see Lemma~\ref{L180418}. As for $(5,1)$, we have $D^{(5,1)}\da_{\SSS_5}\cong D^{(4,1)}$ and so the same lemma applies. 

Let $n>6$. Since $x_2\in\F \SSS_4\leq \F\SSS_{n-1}$, we have $x_2D^\la=0$ only if $x_2(D^\la{\da}_{\SSS_{n-1}})= 0$, which happens only if $x_2 D^\mu=0$ for all composition factors $D^\mu$ of $D^\la{\da}_{\SSS_{n-1}}$. Then by the inductive assumption we have that all of these composition factors are of the form $D^{(n-1)}$ or $D^{\be_{n-1}}$. By Lemma~\ref{L180418_3}, we conclude that  $\la\in\{(n),\be_n\}$.
\end{proof}

\begin{Corollary} \label{C180418}
Let $p=2$, $n\geq 5$, and $\la\in\Par_2(n)$ satisfy $\la\not\in\{(n),\be_n\}$. Then the $\F \SSS_n$-homomorphism $\zeta_2:M_2\to \EE(\la)$ is non-zero on $S_2$.  
\end{Corollary}
\begin{proof}
Apply Lemmas~\ref{L150817} and \ref{L180418_2}.
\end{proof}

\subsection{The case $k=3$ and $p=3$}

\begin{Lemma} \label{L150817_1} 
Let $p=3$. Then $x_3D^{(4,1,1)}\neq 0$. 
\end{Lemma}
\begin{proof}
We use the known fact that $D^{(4,1,1)}$ is the exterior square of $D^{(5,1)}$---this can be seen for example by comparing the Brauer characters of the two modules. 
The module $D^{(5,1)}$ has basis $v_1,\dots,v_4$, where $v_r:=\bar \eps_r-\bar\eps_{r+1}$, where $\{\eps_1,\dots,\eps_6\}$ is the natural basis of the permutation module $M^{(5,1)}$, and for $v\in M^{(5,1)}$, we denote 
$$\bar v:=v+\F \cdot(\eps_1+\dots+\eps_6)\in M^{(5,1)}/\F \cdot(\eps_1+\dots+\eps_6).
$$
We now compute
\begin{alignat*}{3}
(1,2,3)v_1&=v_2, &\quad (1,2,3)v_2&=-v_1-v_2,
\\
(1,3,2)v_1&=-v_1-v_2, & (1,3,2)v_2&=v_1,
\\
(2,3,4)v_1&=v_1+v_2, & (2,3,4)v_2&=v_3,
\\
(2,4,3)v_1&=v_1+v_2+v_3, & (2,4,3)v_2&=-v_2-v_3,
\\
(1,3,5)v_1&=-v_2, & (1,3,5)v_2&=v_2+v_3+v_4,
\\
(1,5,3)v_1&=-v_2-v_3-v_4, & (1,5,3)v_2&=-v_1,
\\
(1,2,6)v_1&=v_1+v_3-v_4, & (1,2,6)v_2&=-v_1+v_2-v_3+v_4,
\\
(1,6,2)v_1&=v_1-v_3+v_4, & (1,6,2)v_2&=v_1+v_2,
\\
(3,4,5)v_1&=v_1, & (3,4,5)v_2&=v_2+v_3,
\\
(3,5,4)v_1&=v_1, & (3,5,4)v_2&=v_2+v_3+v_4,
\\
(2,4,6)v_1&=v_1+v_2+v_3, & (2,4,6)v_2&=-v_3,
\\
(2,6,4)v_1&=-v_1+v_3-v_4, & (2,6,4)v_2&=-v_1+v_2-v_3+v_4,
\\
(1,5,6)v_1&=-v_2-v_3-v_4, & (1,5,6)v_2&=v_2,
\\
(1,6,5)v_1&=-v_1-v_3+v_4, & (1,6,5)v_2&=v_2,
\\
(4,5,6)v_1&=v_1, & (4,5,6)v_2&=v_2,
\\
(4,6,5)v_1&=v_1, & (4,6,5)v_2&=v_2.
\end{alignat*}
Hence
\begin{align*}
x_3(v_1\wedge v_2)=
&
\,\,v_2\wedge (-v_1-v_2)
+(-v_1-v_2)\wedge v_1
\\
&
-(v_1+v_2)\wedge v_3
-(v_1+v_2+v_3)\wedge (-v_2-v_3)
\\
&
-(-v_2)\wedge (v_2+v_3+v_4)
-(-v_2-v_3-v_4)\wedge (-v_1)
\\
&
-(v_1+v_3-v_4)\wedge (-v_1+v_2-v_3+v_4)
-(v_1-v_3+v_4)\wedge (v_1+v_2)
\\
&+v_1\wedge (v_2+v_3)
+v_1\wedge (v_2+v_3+v_4)
\\
&+(v_1+v_2+v_3)\wedge (-v_3)
+(-v_1+v_3-v_4)\wedge (-v_1+v_2-v_3+v_4)
\\
&+(-v_2-v_3-v_4)\wedge v_2
+(-v_1-v_3+v_4)\wedge v_2
\\
&-v_1\wedge v_2
-v_1\wedge v_2
\\
=&
\,\,v_1\wedge v_4- v_2\wedge v_4,
\end{align*}
which is non-zero, completing the proof.
\end{proof}

\begin{Lemma} \label{L150817_2} 
Let $p=3$, $n\geq 6$, and $\la\in\Par_3(n)$ satisfy $h(\la)\geq 3$, $h(\la^\Mull)\geq 3$. Then  $x_3 D^\la\neq 0$.  
\end{Lemma}
\begin{proof}
We apply induction on $n$. If $n=6$, the only $\la$ that satisfies the assumptions $h(\la)\geq 3$, $h(\la^\Mull)\geq 3$ is $(4,1,1)$, and we can apply Lemma~\ref{L150817_1}.  Let $n>6$. Since $x_3\in\F \SSS_6\leq \F\SSS_{n-1}$, we have $x_3D^\la=0$ only if $x_3(D^\la{\da}_{\SSS_{n-1}})= 0$, which happens  only if $x_3 D^\mu=0$ for all composition factors $D^\mu$ of $D^\la{\da}_{\SSS_{n-1}}$. 
Then by the inductive assumption we have that $h(\mu)\leq 2$ or $h(\mu^\Mull)\leq 2$ for all composition factors $D^\mu$ of $D^\la{\da}_{\SSS_{n-1}}$. By Lemma~\ref{LBr1}, we have $h(\la)\leq 2$ or $h(\la^\Mull)\leq 2$, which is a contradiction.
\end{proof}

\begin{Corollary} \label{C160817}
Let $p=3$, $n\geq 6$, and $\la\in\Par_3(n)$ satisfy $h(\la)\geq 3$, $h(\la^\Mull)\geq 3$. Then the $\F \SSS_n$-homomorphism $\zeta_3:M_3\to \EE(\la)$ is non-zero on $S_3$.  
\end{Corollary}
\begin{proof}
Apply Lemmas~\ref{L150817} and \ref{L150817_2}.
\end{proof}

\subsection{The case $k=3$ and $p=2$}

\begin{Lemma} \label{L150817_4} 
Let $p=2$. Then $x_3D^{(3,2,1)}\neq 0$. 
\end{Lemma}
\begin{proof}
Since $(3,2,1)$ is a $2$-core, we have $D^{(3,2,1)}\cong S^{(3,2,1)}$, so we will just prove that $x_3S^{(3,2,1)}\neq 0$. 
We use the same polytabloid basis of $S^{(3,2,1)}$ as in the proof of Lemma~\ref{L180418} and the same 
polytabloid 
\begin{align*}
e_t=&\,25|3+35|2+15|3+15|2+35|1+25|1+24|3
\\&+34|2+14|3+14|2+34|1+24|1.
\end{align*}
Now an explicit calculation shows that the basis element $12|4$ appears in $x_3e_t$ with coefficient $1$, in particular, $x_3e_t\neq 0$.
\end{proof}

\begin{Lemma} \label{L150817_5} 
Let $p=2$, $n\geq 6$, and $\la\in\Par_2(n)$ satisfy $h(\la)\geq 3$. Then $x_3 D^\la\neq 0$.  
\end{Lemma}
\begin{proof}
We apply induction on $n$. If $n=6$, the only $\la$ that satisfies the assumption $h(\la)\geq 3$ is $(3,2,1)$, and we can apply Lemma~\ref{L150817_4}.  Let $n>6$. Since $x_3\in\F \SSS_6\leq \F\SSS_{n-1}$, we have $x_3D^\la=0$ if and only if $x_3(D^\la{\da}_{\SSS_{n-1}})= 0$ only if $x_3 D^\mu=0$ for all composition factors $D^\mu$ of $D^\la{\da}_{\SSS_{n-1}}$. 
Then by the inductive assumption we have that $h(\mu)\leq 2$ for all composition factors $D^\mu$ of $D^\la{\da}_{\SSS_{n-1}}$. By Lemma~\ref{LBr2}, we have $h(\la)\leq 2$,  which is a contradiction.
\end{proof}

\begin{Corollary} \label{C150817_8} 
Let $p=2$, $n\geq 6$, and $\la\in\Par_2(n)$ satisfy $h(\la)\geq 3$. Then the $\F \SSS_n$-homomorphism $\zeta_3:M_3\to \EE(\la)$ is non-zero on $S_3$.  
\end{Corollary}
\begin{proof}
Apply Lemmas~\ref{L150817_4} and \ref{L150817_5}.
\end{proof}

\section{Reduction theorems}
\label{SMain}

\subsection{First reduction theorems}


The reduction results that we need are substantially more difficult to prove in the case $p=2|n$. In this section, we deal with all the other cases. 

\begin{Lemma}  \label{L260418_3} 
Let $p=3$, $n\equiv 0\pmod{3}$ and $n\geq 6$. Suppose that $G$ is a $2$-transitive subgroup of $\SSS_n$ which is not $3$-homogeneous  and such that $(S_1^*)^G=0$. If $\la\in\Par_3(n)$ with $h(\la),h(\la^\Mull)\geq 3$, then $D^\la{\da}_G$ is reducible. 
\end{Lemma}

\begin{proof}
As $G$ is $2$-transitive, we have $i_2(G)=1$, hence 
$\phi(\I(G))\cong D_0$ for every non-zero $\phi\in \Hom_{\SSS_n}(\I(G),M_2)$. Since $D_2$ is a submodule of $M_2$ by Lemma~\ref{L160817_0}, it follows that $D_2$ does not appear in the head of $\I(G)$, i.e. $
\Hom_{\SSS_n}(\I(G),D_2)=0.
$
Moreover, $\Hom_{\SSS_n}(\I(G),S_1^*)=(S_1^*)^G=0$ by assumption.

On the other hand,  
$i_3(G)>1$ means that there is a non-zero homomorphism 
$\psi\in \Hom_{\SSS_n}(\I(G),M_3)$ whose image is not $D_0$. 
So Lemma~\ref{L160817_0} implies that  $D_3$ is a composition factor of $\im \psi$. 

Now we deduce from Corollary~\ref{C160817} that $D_3$ is a composition factor of $\im(\zeta_3\circ\psi)$. So the proof is complete by Lemma~\ref{LBasic2}. 
\end{proof}




\begin{Lemma} \label{L260418_4} 
Let $p=3$, $n\equiv 1\pmod{3}$, $n\geq 7$, and $G$ be a 
transitive subgroup of $\SSS_n$ which is not $3$-homogeneous and such that $(S_2^*)^G=0$. If $\la\in\Par_3(n)$ with $h(\la),h(\la^\Mull)\geq 3$, then $D^\la{\da}_G$ is reducible.
\end{Lemma}
\begin{proof}
Since $G$ is transitive, we have $i_1(G)=1$, we have  
$\phi(\I(G))\cong D_0$ for every non-zero $\phi\in \Hom_{\SSS_n}(\I(G),M_1)$. Since $D_1$ is a submodule of $M_1$, it follows that $D_1$ does not appear in the head of $\I(G)$, i.e. $
\Hom_{\SSS_n}(\I(G),D_1)=0$.

The assumption that $G$ is not $3$-homogeneous means that  $i_3(G)>1$. So there is a non-zero homomorphism 
$\psi\in \Hom_{\SSS_n}(\I(G),M_3)$ whose image is not $D_0$. 
The assumption  $(S_2^*)^G=0$ is equivalent to $\Hom_{\SSS_n}(\I(G),S_2^*)=0$. Taking into account the previous paragraph, we now deduce from Lemma~\ref{L160817_1}, that $D_3$ is a composition factor of $\im \psi$. Now by  Corollary~\ref{C160817}, we have that $D_3$ is a composition factor of $\im(\zeta_3\circ\psi)$. So the proof is complete by Lemma~\ref{LBasic2}. 
\end{proof}




\begin{Lemma} \label{L170418} 
Let $p=3$, $n\equiv 2\pmod{3}$, $n\geq 8$, and $G$ be a $2$-transitive subgroup of $\SSS_n$ which is not $3$-homogeneous. If $\la\in\Par_3(n)$ with $h(\la),h(\la^\Mull)\geq 3$, then $D^\la{\da}_G$ is reducible.
\end{Lemma}
\begin{proof}
By Lemma~\ref{L160817_2}, we have a short exact sequence 
$$0\to M_2\to M_3\to S_3^*\to 0.$$ 
Since $i_2(G)=1<i_3(G)$, we deduce that 
$
\Hom_{\SSS_n}(\I(G),S_3^*)\neq0.
$
So there is an $\F\SSS_n$-homomorphism $\psi:\I(G)\to M_3$ such that $D_3$ is a composition factor of $\im \psi$. Now we deduce from Corollary~\ref{C160817} that $D_3$ is a composition factor of $\im(\zeta_3\circ\psi)$. So the proof is complete by Lemma~\ref{LBasic2}. 
\end{proof}



\begin{Lemma} \label{L160817N} 
Let $p=3$, $n\equiv 1\pmod{3}$ and $G$ be a $2$-transitive subgroup of $\SSS_n$ with $G=O^3(G)$. Then $(S_2^*)^G=0$. 
\end{Lemma}
\begin{proof}
By Lemma~\ref{L160817_1}, we have $M_2 = D_1 \oplus Y$ with $Y\sim D_0|S_2^*$. Since $G$ is $2$-transitive, we have $\dim M_2^G=1$ and $D_1^G=0$, hence $\dim Y^G =1$. 
Now the result follows by considering the long exact sequence in cohomology corresponding to the short exact sequence 
$$
0\to D_0\to Y\to S_2^*\to 0
$$ 
and using $H^1(G,D_0)=0$, which comes from the assumption $G=O^3(G)$. 
\end{proof}

\begin{Corollary} \label{COPN} 
Let $p=3$, $n\equiv 1\pmod{3}$ and $G$ be a $2$-transitive subgroup of $\SSS_n$ with non-abelian socle. Then $(S_2^*)^G=0$, unless possibly $n=28$ and $G= SL_2(8).3$. 
\end{Corollary}
\begin{proof}
This follows from Lemma~\ref{LOP}, and Lemma~\ref{L160817N} applied to $O^3(G)$ in place of $G$.
\end{proof}

The exceptional case in Corollary \ref{COPN} does not create problems:

\begin{Lemma} \label{L281}
Let $G= SL_2(8) \rtimes C_3 < \SSS_{28}$ be a $2$-transitive subgroup, and $D^\la$ be an irreducible $\F\SSS_{28}$-module with $D^\la,D^\la\otimes\sgn\not\cong D^{(28)},D^{(27,1)}$. 
Then $D^\la{\downarrow}_G$ is reducible. 
\end{Lemma}
\begin{proof}
The largest degree of any irreducible $\F G$-module is $\leq 27$, cf. \cite{Atlas}.  On the other hand, by the assumptions on $D^\la$ we have  
$\dim D^\la > 27$ by \cite[Theorem 6]{James}.
\end{proof}

\begin{Lemma}\label{L-S2}
Let $p = 3$, $7 \leq n \equiv 1 \pmod{3}$, and let $G < \SSS_n$ be a $2$-transitive subgroup with abelian socle $S$. Then 
one of the following statements holds.
\begin{enumerate}[\rm(a)]
\item $(S_2^*)^G = 0$.
\item $n = r^d$ for a prime $r$, and either $G \leq A\Gamma L_1(r^d)$ or $d=2$ and $G \leq AGL_2(r)$.
\end{enumerate}
\end{Lemma}

\begin{proof}
By the O'Nan-Scott Theorem \cite[Theorem 4.1]{Cam} (and the remarks after it), $S$ is an elementary abelian 
$r$-group of order $n=r^d$, for a prime $r$, and $G = S \rtimes G_0$ with $G_0 \leq GL_d(r)$. The $2$-transitivity of
$G$ implies that $G_0$ acts transitively on the nonzero vectors of $\F_r^d$. If $d=1$ or $2$, then (b) holds.

Let $d \geq 3$. We apply to the subgroup $G_0$ a version of Hering's theorem as given in \cite[Proposition 3.3]{KT2}. 
Denoting $Z := Z(GL_{d}(r))$, we conclude that one of the following holds:
\begin{enumerate}
\item[{\rm (i)}] $G_0 \rhd  SL_{a}(q_{1})$ with $q_{1}^{a} = r^{n}$ and $a \geq 2$;

\item[{\rm (ii)}] $G_0 \rhd  Sp_{2a}(q_{1})'$ with $q_{1}^{2a} = r^{n}$ and $a \geq 2$;

\item[{\rm (iii)}] $G_0 \rhd  G_{2}(q_{1})'$ with $q_{1}^{6} = r^{n}$ and $2|r$;

\item[{\rm (iv)}] $G_0Z$ is contained in $\Gamma L_{1}(r^{n})$;

\item[{\rm (v)}] $(r^{n},G_0Z)$ is $(3^{4}, \leq 2^{1+4}_{-}\cdot \SSS_{5})$, $(3^{4}, \rhd SL_{2}(5))$, 
$(2^{4},\AAA_{7})$ or $(3^{6},SL_{2}(13))$.
\end{enumerate}
If case (iv) occurs,
then $G_0 \leq \Gamma L_1(r^d)$ and conclusion (b) holds. In all other cases, we see that $G_0$ contains a
perfect subgroup $K$ which is still transitive on the nonzero vectors of $\F_r^d$, unless $(n,G_0) = (2^6,G_2(2))$.
In the exceptional case, we take $K := G_2(2)$ and note that $O^3(K) = K$. Thus in all cases,
$G$ contains the $2$-transitive subgroup $H:=S \rtimes K$ with $O^3(H) = H$, hence we are done by 
Lemma \ref{L160817N}.
\end{proof}

The exceptions in Lemma \ref{L-S2}(b) can be dealt with easily:

\begin{Lemma}\label{L-S21}
Suppose we are in the case {\rm (b)} of Lemma \ref{L-S2}. If $D^\la$ is an irreducible $\F\SSS_n$-module with $D^\la,D^\la\otimes \sgn\not\cong D^{(n)},D^{\al_n}$, then $D^\la\da_G$ is reducible.
\end{Lemma}

\begin{proof}
Assume the contrary. 
If $d=1$, then $n = r$ and $|G| \leq |AGL_1(r)| = r(r-1) < n^2$. If $d=2$, then $n=r^2$ and 
$|G| \leq |AGL_2(r)| < r^6 = n^3$. If $d \geq 3$, then 
$|G| \leq |A\Gamma L_1(r^d)| = n(n-1)d < n^3$. In all cases, $\dim D^\la < |G|^{1/2} < n^{3/2}$. 
On the other hand, the assumption on $D^\la$ implies by \cite{James} that 
$\dim D^\la \geq (n^2-5n+2)/2$, which is larger than $n^{3/2}$ if $n \geq 13$, yielding a contradiction.
The only remaining case is $n=7$, in which case $\dim D^\la \leq 6$, again contradicting the assumption
on $D^\la$. 
\end{proof}

\begin{Theorem} \label{Tp=3}
Let $p=3$, $n\geq 6$, $\la\in\Par_3(n)$ with $h(\la),h(\la^\Mull)\geq 3$,  and $G$ be a $2$-transitive subgroup of $\SSS_n$. If $D^\la{\da}_G$ is irreducible then $G$ is $3$-homogeneous. 
\end{Theorem}
\begin{proof}
If $n\equiv 2\pmod{3}$, the result follows from Lemma~\ref{L170418}. 
If $n\equiv 0\pmod{3}$, the result follows from Lemma~\ref{L260418_3} and Corollary~\ref{CThree}. 
If $n\equiv 1\pmod{3}$, the result follows from Lemmas~\ref{L260418_4}, \ref{L281}, \ref{L-S2}, \ref{L-S21} and Corollary~\ref{COPN}. 
\end{proof}

\begin{Theorem} \label{Tp=2odd}
Let $p=2$, $n\geq 7$ be odd, $\la\in\Par_2(n)$ with $h(\la)\geq 3$,  and $G$ be a $2$-transitive subgroup of $\SSS_n$. If $D^\la{\da}_G$ is irreducible then $G$ is $3$-homogeneous. 
\end{Theorem}
\begin{proof}
The proof is similar to that of Lemma~\ref{L170418}, but uses Lemma~\ref{L150817_7} instead of Lemma~\ref{L160817_2}, and 
Corollary~\ref{C150817_8} instead of Corollary~\ref{C160817}
\end{proof}

\subsection{Reduction theorems for $p=2 \mid n$}
\begin{Lemma}\label{l6}
Let $p=2$, $n\geq 6$ even and $\lambda\in\Par_2(n)\setminus\{(n),\be_n\}$. If 
$$i_2(G)>1+\dim (S_1^*)^G$$ 
then $D^\lambda\da_G$ is reducible.
\end{Lemma}

\begin{proof}
By Lemma \ref{l1} there exists $L\subseteq M_2$ with $L\cong D_0\oplus S_1^*$ and $M_2/L\cong S_2^*$. Note that  
$$\dim\Hom_{\SSS_n}(\I(G), L)=\dim\Hom_{\SSS_n}(\I(G), D_0\oplus S_1^*)=1+\dim (S_1^*)^G,
$$ 
so by assumption, there exists $\psi_G:\I(G)\to M_2$ such that the image of $\psi_G$ is not contained in $L$. So, since $\soc(S_2^*)\cong D_2$, we deduce that  $D_2$ is a composition factor of $\im \psi_G$. 

By Corollary~\ref{C180418}, $\zeta_2:M_2\to\EE(\la)$ is non-zero on $S_2\subseteq M_2$. But $\head S_2\cong D_2$, so 
$D_2$ is a composition factor of $\im\zeta_2$. 
Since $D_2$ appears with multiplicity $1$ in $M_2$ it follows that the image of $\zeta_2\circ\psi_G:\I(G)\to\EE(\la)$ has $D_2$ as composition factor. 
The lemma then holds from Lemma \ref{LBasic2}.
\end{proof}

\begin{Lemma}\label{p1}
Let $p=2 \mid n\geq 6$ and $\lambda\in\Par_2(n)$ not be a JS-partition. If $D^\lambda\da_G$ is irreducible then $G$ is $2$-homogeneous  and $(S_1^*)^G=0$.
\end{Lemma}

\begin{proof}
If $(S_1^*)^G\not=0$ then there is a non-zero homomorphism $\I(G)\to S_1^*$. But $S_1^*$ is a submodule of $\EE(\la)$ by Lemma \ref{l7}, and $\soc S_1^*\cong D_1$, so this yields a non-zero homomorphism $\psi:\I(G)\rightarrow \EE(\la)$ with $\im\psi\not\cong\bone_{\SSS_n}$. By Lemma~\ref{LBasic2}, this contradicts the irreducibility of $D^\lambda\da_G$, thus $(S_1^*)^G=0$. By Lemma \ref{l6} we now have that   
$$0=\dim (S_1^*)^G\geq i_2(G)-1,$$ hence $i_2(G)=1$, i.e. $G$ is $2$-homogeneous.
\end{proof}

\begin{Lemma}\label{p5}
Let $p=2$, $n\geq 8$ with $n\equiv 0\pmod{4}$, and $\la\in\Par_2(n)$. If $D^\la\da_G$ is irreducible and $D_2\subseteq\EE(\la)$ then $D_2^G=0$.
\end{Lemma}

\begin{proof}
If $D_2^G\not=0$ then there is a non-zero homomorphism $\I(G)\to D_2\subseteq \EE(\la)$, which yields a non-zero homomorphism $\psi:\I(G)\rightarrow \EE(\la)$ with $\im\psi\not\cong\bone_{\SSS_n}$. By Lemma~\ref{LBasic2} this contradicts the irreducibility of $D^\lambda\da_G$. 
\end{proof}

\begin{Lemma}\label{p2}
Let $p=2$, $n\geq 8$ with $n\equiv 0\pmod{4}$, and $\la\in\Par_2(n)$. Assume that $D_2\subseteq\EE(\la)$, $(S_1^*)^G=0$ and $D^\la\da_G$ is irreducible. Then:
\begin{enumerate}
\item[{\rm (i)}] $G$ is $2$-homogeneous, 
$(S_2^*)^G\!=0$, $S_2^G\!=0$, $\dim S_3^G\!=i_3(G)-1$, 
and $\dim (S_3^*)^G\!\geq i_3(G)-1$.
\item[{\rm (ii)}] If $h(\la)\geq 3$ then $G$ is $3$-homogeneous.
\end{enumerate}
\end{Lemma}

\begin{proof}
(i) By Lemma \ref{l6}, using the assumption $(S_1^*)^G=0$ we get $i_2(G)=1$, i.e. $G$ is $2$-homogeneous. This also implies that $i_1(G)=1$. 

As $S_1^*\cong D_1|D_0$, the equality $(S_1^*)^G=0$ implies $D_1^G=0$, so $\Hom_{\SSS_n}(\I(G),D_1)=0$, i.e. $D_1$ is not a quotient of $\I(G)$. By \ref{p5}, we have $D_2^G=0$, so by a similar argument,  $D_2$ is also not a quotient of $\I(G)$. By  Lemma \ref{l1} we have that $S_2^*\cong D_2|D_1$ and $S_2\cong D_1|D_2$, hence $(S_2^*)^G=0$ and $S_2^G=0$.

By Lemma \ref{l14} and self-duality of $M_1$ and $M_3$ we have  $M_3\sim M_1\oplus (S_2^*|S_3^*)$ and $M_3\sim M_1\oplus (S_3|S_2)$, so
\[i_3(G)\leq i_1(G)+\dim (S_2^*)^G+\dim (S_3^*)^G=1+\dim (S_3^*)^G.\]
Since $S_2^G=0$, we have  $\dim (S_3|S_2)^G=\dim S_3^G$, hence
\[i_3(G)=i_1(G)+\dim (S_3|S_2)^G=1+\dim S_3^G.\]

(ii) If $G$ is not $3$-homogeneous, then $\dim S_3^G=i_3(G)-1\not=0$. From Lemma \ref{l14} and by self-duality of $M_3$ we have that $S_3\cong D_2|D_1|D_3\sim S_2^*|D_3$. From $(S_2^*)^G=0$ it follows that $S_3\subseteq M_3$ is a quotient of $\I(G)$. In particular there exists $\psi:\I(G)\rightarrow M_3$ with $D_3$ as a composition factor of $\im\psi$. So $D_3$ is a composition factor of $\im(\zeta_3\circ \psi)$ from Corollary \ref{C150817_8}. 
We are now done by Lemma \ref{LBasic2}.
\end{proof}

\begin{Lemma}
Let $p=2$, $n\geq 6$ with $n\equiv 2\pmod{4}$ and $\la\in\Par_2(n)$. Assume that $S_2^*\subseteq\EE(\la)$. If $D^\la\da_G$ is irreducible then $(S_2^*)^G=0$.
\end{Lemma}

\begin{proof}
If $(S_2^*)^G\not=0$ then there is a non-zero homomorphism $\I(G)\to S_2^*$. As $S_2^*\subseteq \EE(\la)$ by assumption, and $\soc S_2^*\cong D_2$, this yields a non-zero homomorphism $\psi:\I(G)\rightarrow \EE(\la)$ with $\im\psi\not\cong\bone_{\SSS_n}$. By Lemma~\ref{LBasic2}, this contradicts the irreducibility of $D^\lambda\da_G$.
\end{proof}

\begin{Lemma}\label{p6}
Let $p=2$, $n\geq 6$ with $n\equiv 2\pmod{4}$ and $\la\in\Par_2(n)$. If $D^\la\da_G$ is irreducible and $S_2^*$ or $Y_2/D_1$ is contained in $\EE(\la)$ then $\dim(Y_2/D_1)^G=1$.
\end{Lemma}

\begin{proof}
By Lemma \ref{l1}, we have $Y_2/D_1\cong D_0|D_2|D_0|D_1\sim D_0|S_2^*$. In particular $(Y_2/D_1)^G\neq 0$. Assume that $\dim(Y_2/D_1)^G\geq 2$. Then there exists a homomorphism $\psi:\I(G)\to Y_2/D_1$ such that $\im\psi$ has $D_2$ as a composition factor. It follows that there also exists a homomorphism $\psi':\I(G)\to S_2^*$ such that $\im\psi$ has $D_2$ as a composition factor. 
By Lemma~\ref{LBasic2}, this contradicts the irreducibility of $D^\lambda\da_G$.
\end{proof}

\begin{Lemma}\label{p3}
Let $p=2$, $n\geq 10$ with $n\equiv 2\pmod{4}$ and $\la\in\Par_2(n)$. Assume that $S_2^*$ or $Y_2/D_1$ is contained in $\EE(\la)$, $(S_1^*)^G=0$ and $D^\la\da_G$ is irreducible. Then:
\begin{enumerate}
\item[{\rm (i)}] $G$ is $2$-homogeneous,  $S_2^G=0$  
and $\dim (S_3^*)^G\geq i_3(G)-1$.
\item[{\rm (ii)}] If $h(\la)\geq 3$ then $G$ is $3$-homogeneous.
\end{enumerate} 
\end{Lemma}

\begin{proof}
(i) By Lemma \ref{l6}, the assumption $(S_1^*)^G=0$ implies $i_2(G)=1$, i.e. $G$ is $2$-homogeneous. This also implies  $i_1(G)=1$.

From Lemma \ref{l1}(ii), we have that $D_0\oplus S_2\subseteq M_2$, so
\[1=i_2(G)=\dim M_2^G\geq\dim D_0^G+\dim S_2^G=1+\dim S_2^G,\]
hence $S_2^G=0$. 

By Lemma \ref{l14}, we have $M_3\sim (Y_2/D_1)|S_1^*|S_3^*$.  So, using Lemma~\ref{p6} we get 
\[i_3(G)=\dim M_3^G\leq \dim (Y_2/D_1)^G+\dim(S_1^*)^G+\dim(S_3^*)^G=1+\dim(S_3^*)^G\]
which completes the proof of (i).

(ii) By Lemma \ref{l14}(ii), there exist submodules $A\subseteq Y_2/D_1$ and $B\subseteq M_3$ such that $A\cong D_0|D_2|D_0$, $B\sim A|D_3$ and $M_3\sim B|S_2|S_1^*$. 
If $G$ is not $3$-homogeneous, i.e. $i_3(G)=\dim M_3^G\geq 2$, then, by  (i) and Lemma \ref{p6}, we have
\begin{align*}
\dim B^G &\geq i_3(G)-\dim S_2^G-\dim (S_1^*)^G
\\&=i_3(G)\geq 2>1=\dim (Y_2/D_1)^G\geq \dim A^G.
\end{align*}
Hence there exists a homomorphism $\psi:\I(G)\to B\subseteq M_3$ with $\im\psi\not\subseteq A$. In particular, $D_3$ is a  composition factor of $\im \psi$. So $D_3$ is a composition factor of $\im(\zeta_3\circ \psi)$ from Corollary \ref{C150817_8}, and we are done by Lemma \ref{LBasic2}.
\end{proof}

\begin{Lemma}\label{p7}
Let $p=2$, $n\geq 8$ be even, $\la\in\Par_2(n)\setminus \{(n),\be_n\}$ be a JS partition, $D^\la\da_G$ be irreducible, and $(S_1^*)^G=0$. Then:
\begin{enumerate}
\item[{\rm (i)}] $G$ is $2$-homogeneous.
\item[{\rm (ii)}] If $h(\la)\geq 3$ then $G$ is $3$-homogeneous.
\end{enumerate} 
\end{Lemma}

\begin{proof}
If $n\equiv 0\pmod{4}$, then $D_2\subseteq \EE(\la)$ by Lemma \ref{l15}, and the result follows from Lemma \ref{p2}. The case $n\equiv 2\pmod{4}$ is handled similarly but using Lemma~\ref{l17} in place of Lemma \ref{l15} and Lemma~\ref{p3} in place of Lemma \ref{p2}.
\end{proof}

\subsection{Wreath products and proofs of Theorems~\ref{TNat} and \ref{TB}}

In this subsection, we assume that $n=ab$ for some $a,b\in\Z_{>1}$ and consider restrictions of irreducible $\F \SSS_n$-modules to the natural subgroup 
$$\SSS_a\wr \SSS_b=(\underbrace{\SSS_a\times\dots\times \SSS_a}_{\text{$b$ times}}) \rtimes \SSS_b.$$ 
A special role will be played by the irreducible $\F (\SSS_a\wr \SSS_b)$-modules of the form $D^\mu\wr D^\nu$ which as a vector space is $(D^\mu)^{\otimes b}\otimes D^\nu$, and the action on $v_1\otimes \dots\otimes v_b\otimes w\in (D^\mu)^{\otimes b}\otimes D^\nu$ is determined from the following requirements: $(g_1,\dots,g_b)\in \SSS_a\times\dots\times \SSS_a$ acts as 
$$
(g_1,\dots,g_b)\cdot (v_1\otimes \dots \otimes v_b\otimes w)=(g_1v_1)\otimes \dots \otimes (g_bv_b)\otimes w
$$
and $h\in\SSS_b$ acts as
$$
h\cdot (v_1\otimes \dots \otimes v_b\otimes w)=(v_{h^{-1}(1)}\otimes \dots \otimes v_{h^{-1}(b)})\otimes (hw).
$$

\begin{Lemma} \label{LBS} 
Let $p=2$ and $n=ab$ for some $a,b\in\Z_{>1}$. Then $D^{\be_n}\da_{\SSS_a\wr\SSS_b}$ is irreducible if and only if $a$ is odd, in which case $D^{\be_n}\da_{\SSS_a\wr\SSS_b}\cong D^{\be_a}\wr D^{\be_b}$. 
\end{Lemma}

\begin{proof}
Recall, see \cite{Wales}, that $\dim D^{\be_n}=2^{\lfloor(n-1)/2\rfloor}$, and furthermore $D^{\be_n}$ can be obtained 
by reducing modulo $2$ a basic spin complex representation $D_{n,\C}$ of a double cover $\hat\SSS_n$ of 
$\SSS_n$.
As in the proof of \cite[Theorem 4.3]{KT}, we let $G$ (resp. $K$, $B$) be the full inverse image in $\hat\SSS_n$ of 
$\SSS_a \wr \SSS_b$ (resp. $\SSS_a^b = \underbrace{\SSS_a\times\dots\times \SSS_a}_{\text{$b$ times}}$, 
$\SSS_a \wr \AAA_b$). It was shown there that $D_{n,\C} \downarrow_G\ \cong V_\C \otimes W_\C$ or
$\ind^G_B(V_\C \otimes W_\C)$. Here, $V_\C$ is a (possibly projective) $\C G$-representation which is irreducible 
over $K$, whose restriction to the full inverse image $\hat\SSS_a$ of $\SSS_a \times 1 \ldots \times 1$ in 
$\hat\SSS_n$ is a sum of basic spin representations. Next, $W_\C$ is a (possibly projective) irreducible representation
of $G$, respectively of $B$, in which $K$ acts trivially, and which gives rise to a basic spin representation of $\SSS_b$,
respectively of $\AAA_b$.

It follows by reducing modulo $2$ that all composition factors 
of the restriction of $D^{\be_n}$ to $\SSS_a \times 1 \ldots \times 1$ are isomorphic to $D^{\beta_a}$. Hence, all  composition factors of $D^{\be_n}\da_{\SSS_a^b}$ are isomorphic to 
$$D_a:=D^{\beta_a} \otimes D^{\beta_a} \otimes \ldots \otimes D^{\beta_a},$$
which can easily be seen to extend to the module $D^{\beta_a} \wr D^{(b)}$ of $\SSS_a \wr \SSS_b$. This implies
that every irreducible $\F(\SSS_a \wr \AAA_b)$-representation $X$ lying above $D_a$ is isomorphic to 
$D^{\beta_a} \wr Y$ for some irreducible $\F\AAA_b$-representation $Y$. A similar statement holds for 
$\SSS_a \wr \SSS_b$. Now, the aforementioned statement about $W_\C$ implies by reducing modulo $2$  
that if such $X$ occurs in $D^{\beta_n} \downarrow_{\SSS_a \wr \AAA_b}$, then $Y$ is basic spin for $\AAA_b$, i.e.
a composition factor of $D^{\beta_b}\da_{\AAA_b}$.
Therefore, all composition factors of $D^{\be_n}\da_{\SSS_a\wr\SSS_b}$ are of the form $D^{\be_a}\wr D^{\be_b}$. Now the result follows by dimension considerations.
\end{proof}

\begin{Lemma}\label{p9}
Let $p=2$, $n$ be even and $\la\in\Par_2(n)$ be a JS-partition with $\la\not\in\{(n),\be_n\}$. Then $D^\la\da_{\SSS_{n/2}\wr\SSS_2}$ is irreducible if and only if $n\geq 6$ with $n\equiv 2\pmod{4}$ and 
$\la=\al_n$, in which case 
\begin{align*}
D^{\al_n}\da_{\SSS_{n/2}\wr \SSS_2}&\cong(D^{\al_{n/2}}\boxtimes D^{(n/2)})\ua^{\SSS_{n/2}\wr \SSS_2}_{\SSS_{n/2,n/2}}.
\end{align*}
\end{Lemma}

\begin{proof}
By Clifford theory (see e.g. \cite[51.7]{CR}), $D^\la\da_{\SSS_{n/2}\wr\SSS_2}$ is irreducible if and only if one of the following conditions holds:
\begin{enumerate}
\item[{\rm (a)}] $D^\la\da_{\SSS_{n/2,n/2}}$ is of the form $D^\mu\boxtimes D^\mu$, in which case $D^\la\da_{\SSS_{n/2}\wr\SSS_2}$ is $D^\mu\wr D^{(2)}$.
\item[{\rm (b)}] $D^\la\da_{\SSS_{n/2,n/2}}$ is of the form $(D^\mu\boxtimes D^\nu)\oplus (D^\nu\boxtimes D^\mu)$ with $\mu\not=\nu$, in which case $D^\la\da_{\SSS_{n/2}\wr\SSS_2}\cong (D^\mu\boxtimes D^\nu)\ua^{\SSS_{n/2}\wr \SSS_2}_{\SSS_{n/2,n/2}}$. 
\end{enumerate}

By dimensions, if $n\equiv 2\pmod{4}$, we have 
$$D^{\al_n}\da_{\SSS_{n/2,n/2}}\cong (D^{\al_{n/2}}\boxtimes D^{(n/2)})\oplus (D^{(n/2)}\boxtimes D^{\al_{n/2}}).$$ If $n\equiv 0\pmod{4}$, then in the Grothendieck group we have 
$$
[D^{\al_n}\da_{\SSS_{n/2,n/2}}]= 
[D^{\al_{n/2}}\boxtimes D^{(n/2)}]+ [D^{(n/2)}\boxtimes D^{\al_{n/2}}]+2[D^{(n/2)}\boxtimes  D^{(n/2)}],
$$
omitting the first two summands if $n=4$. So we may assume that $\la\neq \al_n$. 


If the parts of $\la$ are all even, let $\mu:=(\la_1/2,\ldots,\la_{h(\la)}/2)$. Then $\mu\in\Par_2(n/2)$ and by Lemma~\ref{LPlus} we have that $D^\mu\boxtimes D^\mu$ is a composition factor of $D^\la\da_{\SSS_{n/2,n/2}}$. So  $D^\la\da_{\SSS_{n/2}\wr\SSS_2}$ is irreducible if and only if $D^\la\da_{\SSS_{n/2,n/2}}$ is irreducible. 
By Proposition \ref{p8}, this happens only in the basic spin case, which has already been excluded by assumption.  

So we can now assume that all parts of $\la$ are odd. If $D^\la\da_{\SSS_{n/2}}$ has at least 3 non-isomorphic composition factors then $D^\la\da_{\SSS_{n/2}\wr\SSS_2}$ is not irreducible. So by Lemma \ref{l25} and since the cases $\al_n$ and $\be_n$ have already been excluded, there are only the exceptional cases (iii) and (iv) of Lemma \ref{l25} to consider.

{\sf Case 1.} $n\geq 24$, $n\equiv 0\pmod{8}$ and $\la=(n/4+3,n/4+1,n/4-1,n/4-3)$. Suppose that $D^\la\da_{\SSS_{n/2}\wr\SSS_2}$ is irreducible. Let 
\begin{align*}
\mu&:=(n/8+3,n/8+1,n/8-1,n/8-3),\\
\nu&:=(n/8+2,n/8+1,n/8-1,n/8-2).
\end{align*} 
By Lemma \ref{l26},  $D^\mu$ and $D^\nu$ are composition factors of $D^\la\da_{\SSS_{n/2}}$.  
It then follows that $$D^\la\da_{\SSS_{n/2,n/2}}\cong (D^\mu\boxtimes D^\nu)\oplus (D^\nu\boxtimes D^\mu).$$

Let 
\begin{align*}
\pi&:=(n/8+2,n/8+1,n/8,n/8-1),\\
\psi&:=(n/8+1,n/8,n/8-1,n/8-2).
\end{align*}
From Lemma~\ref{LPlus} we have that $D^\pi\boxtimes D^\psi$ is a composition factor of $D^\la\da_{\SSS_{n/2+2,n/2-2}}$. As $\nu=\tilde e_i^2 \pi$, by Lemma \ref{Lemma39}, we have  that $D^\nu\boxtimes \bone_{\SSS_{1,1}}\boxtimes D^\psi$ is a composition factor of $D^\la\da_{\SSS_{n/2,1,1,n/2-2}}$. So $D^\psi$ is a composition factor of $D^\mu\da_{\SSS_{n/2-2}}$, 
which contradicts Lemma~\ref{LBrEasy}. 

{\sf Case 2.} $n\geq 22$, $n\equiv 4\pmod{6}$, $\la=((n-1)/3+2,(n-1)/3,(n-1)/3-2,1)$. Suppose that $D^\la\da_{\SSS_{n/2}\wr\SSS_2}$ is irreducible.  Let 
\begin{align*}
\mu&:=((n-4)/6+2,(n-4)/6+1,(n-4)/6-1),
\\
\nu&:=((n-4)/6+2,(n-4)/6,(n-4)/6-1,1).
\end{align*}
By Lemma \ref{l26},  $D^\mu$ and $D^\nu$ are composition factors of $D^\la\da_{\SSS_{n/2}}$.  
It then follows that 
$$D^\la\da_{\SSS_{n/2,n/2}}\cong (D^\mu\boxtimes D^\nu)\oplus (D^\nu\boxtimes D^\mu).$$

Let 
\begin{align*}
\pi&:=((n-4)/6+2,(n-4)/6+1,(n-4)/6),
\\
\psi&:=((n-4)/6+1,(n-4)/6,(n-4)/6-1,1).
\end{align*} 
From Lemma~\ref{LPlus} we have that $D^\pi\boxtimes D^\psi$ is a composition factor of $D^\la\da_{\SSS_{n/2+1,n/2-1}}$. By 
Lemma~\ref{Lemma39}, we have 
$$[D^\la\da_{\SSS_{n/2,1,n/2-1}}:D^\mu\boxtimes \bone_{\SSS_1}\boxtimes D^\psi]\geq 3.$$ In particular $[D^\nu\da_{\SSS_{n/2-1}}:D^\psi]\geq 3$, which contradicts Lemma~\ref{Lemma39}(vi).
\end{proof}

\begin{Lemma}\label{c1}
Let $p=2$, $n\geq 8$ even and $\la\in\Par_2(n)$ be a JS partition with $\la\not\in\{(n),\be_n\}$. If $n=ab$ with $a,b\in\Z_{> 1}$ and $b\geq 3$ then $D^\la\da_{\SSS_a\wr\SSS_b}$ is reducible.
\end{Lemma}

\begin{proof}
It follows from Lemmas \ref{p7} and \ref{l18} since $\SSS_a\wr\SSS_b<\SSS_n$ is not a $2$-homogeneous subgroup. 
\end{proof}

\begin{Proposition} \label{P250418} 
Let $n=ab$ with $a,b\in\Z_{> 1}$, $\la\in\Par_p(n)$ and suppose that $\dim D^\la>1$. Then $D^\la\da_{\SSS_a\wr\SSS_b}$ is reducible unless $p=2$ and one of the following holds:
\begin{enumerate}
\item[{\rm (i)}] $\la=\be_n$ and $a$ is odd, in which case $D^{\be_n}\da_{\SSS_a\wr\SSS_b}\cong D^{\be_a}\wr D^{\be_b}$.
\item[{\rm (ii)}] $n\equiv 2\pmod{4}$, 
$\la=\al_n$ and $b=2$, in which case 
$$
D^{\al_n}\da_{\SSS_{n/2}\wr \SSS_2}\cong(D^{\al_{n/2}}\boxtimes D^{(n/2)})\ua^{\SSS_{n/2}\wr \SSS_2}_{\SSS_{n/2,n/2}}.$$
\end{enumerate}
\end{Proposition}
\begin{proof}
The small cases $n=4$ and $6$ are easy to check. So let $n\geq 8$. If either $p>2$, or $p=2 \nmid n$ and $\la\neq \be_n$, then \cite[Theorem 3.10]{KS2Tran} gives the result since our subgroup is transitive but not $2$-transitive. The case where $\la=\be_n$ is considered in Lemma~\ref{LBS}. So we may assume that $p=2 \mid n$ and $\la\not\in\{(n),\be_n\}$. The case where $\la$ is JS is handled in Lemma~\ref{p9} for $b=2$ and Lemma~\ref{c1} for $b>2$. If $\la$ is not JS, we can apply Lemma~\ref{p1}.  
\end{proof}

\subsection*{Proof of Theorem~\ref{TB}}
By Propositions~\ref{p8} and \ref{P250418} we may assume that $G$ is primitive.
If $G=\AAA_n$, the result follows from \cite[Theorem 1.1]{Benson}. So we may assume that $G$ does not contain $\AAA_n$. Since $D^{\be_n}$ is reduction modulo $2$ of the basic spin module $B_0$ in characteristic $0$, if $D^{\be_n}\da_G$ is irreducible then the restriction $B_0\da_{\hat G}$ is also irreducible for the corresponding subgroup $\hat G\leq \hat \SSS_n$. 
The list of such $G$ is available from \cite[Theorem B]{KT}. One easily checks that it is precisely the cases (a),(b),(e),(g) which remain irreducible in characteristic $2$. Those are, respectively, the cases (b),(c),(d),(e) of Theorem~\ref{TB}.

\subsection*{Proof of Theorem~\ref{TNat}}


Let $\varphi$ denote the Brauer character of $D^{\al_n}$
and let $1+\chi$ denote the permutation character of $\SSS_n$ on $\{1,2 \ldots,n\}$. Then $\varphi = \chi^\circ-1$, 
where $\chi^\circ$ denotes the restriction of $\chi$ to $2'$-elements in $\SSS_n$. 
Note that $\varphi\da_B = \varphi_1 + \varphi_2$, where $\varphi_1$ induces the module $D^{\al_{n/2}}$ of 
the first factor $B_1 = \SSS_{n/2} \times \{1\} < B$ and $\varphi_1$ is trivial on the second factor $B_2 = \{1\} \times \SSS_{n/2}$,
and similarly for $\varphi_2$.

\smallskip
(a) Assume first that $\varphi\da_G$ is irreducible. 
 It follows that $G \not\leq B$, $[G:G \cap B]=2$, and the projection of $G \cap B$ onto $B_i$ 
induces a subgroup $X_i \leq \SSS_{n/2}$ over which $D^{\al_{n/2}}$ is irreducible, and $\psi_i:=(\varphi_i)\da_{G \cap B}$ is irreducible. 
Since $2 \nmid n/2 \geq 3$, this 
irreducibility condition implies that $X_i$ is $2$-transitive for $i=1,2$; in particular, $G \cap B$ acts doubly transitively on
$\{1,2, \ldots,n/2\}$ and on $\{n/2+1, \ldots,n-1,n\}$. As $[G:G \cap B] =2$, it also follows that 
$G$ is transitive, i.e. (i) holds. Furthermore, as $\varphi_{G \cap B} = \psi_1+\psi_2$
and $\varphi\da_G$ is irreducible, we must have that $\psi_1 \neq \psi_2$, i.e. (ii) holds.

\smallskip
(b) Assume now that (i) and (ii) hold, and let $X_i$ denote the projection of $G \cap B$ onto $B_i$ for $i = 1,2$. By (ii), 
$G \cap B$ is $2$-transitive on $\{1,2, \ldots,n/2\}$ and on $\{n/2+1, \ldots,n-1,n\}$, and $\psi_i:= (\varphi_i)\da_{G \cap B}$ is 
irreducible. Thus 
\begin{equation}\label{for-wr2}
  \varphi\da_{G \cap B} = \psi_1 +\psi_2.
\end{equation}  
Next, (i) implies again that $G \not\leq B$, and $G = \langle G \cap B,g \rangle$, where $g$ interchanges 
$\{1,2, \ldots,n/2\}$ and $\{n/2+1, \ldots,n-1,n\}$. Now $g$ interchanges $\psi_1$ and $\psi_2$, and $\psi_1 \neq \psi_2$ by (ii).
Hence \eqref{for-wr2} implies that $\varphi\da_G$ is irreducible.

\begin{Example}\label{imprim}
{\rm Let $6 \leq n \equiv 2 (\bmod\,4)$ and let $L \leq \SSS_{n/2}$ be any $2$-transitive subgroup such that 
$D^{(n/2-1,1)}\da_L$ is irreducible. (There are many such pairs $(n,L)$ with $L$ not containing $\AAA_{n/2}$, for instance,
$n = (q^d-1)/(q-1)$ for some odd $d \geq 3$ and some odd prime power $q$, and $PSL_d(q) \lhd L \leq P\Gamma L_d(q)$.) 
Then the subgroup $L \wr \SSS_2$ obviously satisfies the conditions (i) and (ii) of Theorem B.  But not every subgroup $G$
satisfying these two conditions are of this wreath product type, as one can see on the example of 
$(\SSS_{n/2} \wr \SSS_2) \cap \AAA_n$. 

More generally, we claim that any subgroup $G \leq (\SSS_{n/2} \wr \SSS_2)$
with the two properties
\begin{enumerate}[\rm(a)]
\item $G$ is transitive on $\{1,2, \ldots,n\}$, and
\item the projection of $G \cap B$ onto the first factor $\SSS_{n/2}$ of $B$ has nontrivial kernel and induces a $2$-transitive subgroup of 
$\SSS_{n/2}$ over which $D^{(n/2-1,1)}$ is irreducible, 
\end{enumerate}
satisfies the conditions (i) and (ii) of Theorem B. Indeed, (a) implies that $G = \langle G \cap B,g \rangle$ with $g$ interchanging the two factors 
$\SSS_{n/2}$ of $B$, and so (b) also holds for the second factor $\SSS_{n/2}$. In the notation of the proof of Theorem B, the kernel $K$ of
the projection onto $B_1$ is a nontrivial normal subgroup of the image $L$ of the projection onto $B_2 \cong \SSS_{n/2}$. Using the
description of $2$-transitive subgroups of $\SSS_{n/2}$ \cite{Cam} and the assumption $2 \nmid n/2 \geq 3$, it is straightforward to check 
that $K$ acts nontrivially on $\bone_{\SSS_{n/2}} \boxtimes D^{(n/2-1,1)}$, but it clearly acts trivially on 
$D^{(n/2-1,1)} \boxtimes \bone_{\SSS_{n/2}}$. Thus both of the conditions (i) and (ii) of Theorem $B$ are satisfied, as claimed. It remains an
open question whether (i) and (ii) of Theorem B must imply the above condition (b).
}
\end{Example}

\subsection{Main results for $p=2\mid n$ and proof of Theorem~\ref{TA}}

\begin{Theorem}\label{p4}
Let $p=2$, $n\geq 8$ be even, $\la\in\Par_2(n)$ not be a JS partition, and $D^\la\da_G$ be irreducible. Then: 
\begin{enumerate}
\item[{\rm (i)}]
$G$ is $2$-homogeneous and $(S_1^*)^G=0$. 

\item[{\rm (ii)}] $G$ is $3$-homogeneous unless 
$h(\la)\geq 3$ and there exists $1\leq j\leq h(\la)$ with $\la_j=\la_{j+1}+2$ and 
\[\la_1\equiv\ldots\equiv\la_{j-1}\not\equiv\la_j\equiv\la_{j+1}\not\equiv\la_{j+2}\equiv\ldots\equiv\la_{h(\la)}\pmod{2}\]
\end{enumerate}
\end{Theorem}

\begin{proof}
(i) holds by Lemma \ref{p1}. 

(ii) 
Suppose that $\la$ is not of the exceptional form as described in part (ii). 
Assume first that $\eps_0(\la)+\eps_1(\la)=2$. 
Then Lemma~\ref{l22} implies that whenever $\eps_i(\la),\phi_i(\la)>0$ for some $i\in I$, $B$ is $i$-good for $\la$ and $C$ is $i$-cogood for $\la$, then $(\la_B)^C$ is $2$-regular. Hence by Lemma~\ref{l20}, we have $D^\la\not\subseteq (f_i D^{\tilde e_i\la})/D^\la$ whenever $\eps_i(\la)>0$. By Lemma~\ref{l12}, we now conclude that $\dim\Hom_{\SSS_{n}}(S_1,\EE(\la))<2$. 

Now, by Lemma~\ref{l15}, if $n\equiv 0\pmod{4}$ then $D_2\subseteq \EE(\la)$, and by Lemmas~\ref{l16},\ref{l17}, if 
$n\equiv 2\pmod{4}$ then $S_2^*\subseteq \EE(\la)$ or $Y_2/D_1\subseteq \EE(\la)$. Moreover 
$(S_1^*)^G=0$ by (i). Since $p=2$ and $n$ is even all two-row partitions are JS, so we must have $h(\la)\geq 3$. Now, by Lemmas~\ref{p2}(ii) and \ref{p3}(ii), we have that  $G$ is $3$-homogeneous. 
\end{proof}

\begin{Theorem}\label{TJS}
Let $p=2$, $n$ be even, $\la\in\Par_2(n)$ be a JS partition with $\la\not\in\{(n),\al_n,\be_n\}$, $G\not\leq\SSS_{n-1}$, and $D^\la\da_G$ be irreducible. Then:
\begin{enumerate}
\item[{\rm (i)}] $G$ is primitive. 
\item[{\rm (ii)}] If  $(S_1^*)^G=0$ then $G$ is  $2$-homogeneous.
\item[{\rm (iii)}] If  $(S_1^*)^G=0$ and $h(\la)\geq 3$, then $G$ is $3$-homogeneous. 
\end{enumerate}
\end{Theorem}

\begin{proof}
Part (i) follows from Propositions~\ref{p8} and \ref{P250418}. Parts (ii) and (iii) follow from Lemma~\ref{p7}. 
\end{proof}

\begin{Theorem}\label{TAlpha}
Let $p=2$, $n$ be even, $G\not\leq\SSS_{n-1}$, and $D^{\al_n}\da_G$ be irreducible. Then:
\begin{enumerate}
\item[{\rm (i)}]
$G$ is primitive or $n\equiv 2\pmod{4}$, $G\leq \SSS_{n/2}\wr\SSS_2$ and $G\not\leq\SSS_{n/2,n/2}$. Furthermore, in the second case we have 
$$D^{\al_n}\da_{\SSS_{n/2}\wr \SSS_2}\cong(D^{\al_{n/2}}\boxtimes D^{(n/2)})\ua^{\SSS_{n/2}\wr \SSS_2}_{\SSS_{n/2,n/2}}.
$$

\item[{\rm (ii)}] If $(S_1^*)^G=0$ then $G$ is $2$-homogeneous.
\end{enumerate}
\end{Theorem}
\begin{proof}
Part (i) holds by Propositions~\ref{p8} and \ref{P250418}, while part (ii) holds by Lemma~\ref{p7}. 
\end{proof}

%
%

\begin{Theorem}\label{Tp=2even}
Let $p=2$, $n\geq 8$ be even, and $D^\la$ be an irreducible representation of $\F\SSS_n$ with $\dim D^\la>1$. Suppose that $D^\la$ is not basic spin. If $G\leq \SSS_n$ is a subgroup such that the restriction $D^\la\da_G$ is irreducible, then one of the following holds:
\begin{enumerate}
\item[{\rm (i)}] $G\leq \SSS_{n-1}$ and $\la$ is JS.
\item[{\rm (ii)}] $n\equiv 2\pmod{4}$, $\la=\al_n$, $G\leq \SSS_{n/2}\wr\SSS_2$ and $G\not\leq\SSS_{n/2,n/2}$. Moreover, in this case we have that 
$$D^{\al_n}\da_{\SSS_{n/2}\wr \SSS_2}\cong(D^{\al_{n/2}}\boxtimes \bone_{\SSS_{n/2}})\ua^{\SSS_{n/2}\wr \SSS_2}_{\SSS_{n/2,n/2}}
$$
is irreducible.

\item[{\rm (iii)}] $G$ is  $2$-transitive and either $h(\la)=2$ or 
$h(\la)\geq 3$ and there exists $1\leq j\leq h(\la)$ with $\la_j=\la_{j+1}+2$ and 
\[\la_1\equiv\ldots\equiv\la_{j-1}\not\equiv\la_j\equiv\la_{j+1}\not\equiv\la_{j+2}\equiv\ldots\equiv\la_{h(\la)}\pmod{2}.\]

\item[{\rm (iv)}] $G$ is $3$-homogeneous.
\end{enumerate}
\end{Theorem}
\begin{proof}
If $G\leq\SSS_{n-1}$ then $D^\la\da_{\SSS_{n-1}}$ is irreducible and so $\la$ is JS by definition. Let us now assume that $G\not\leq\SSS_{n-1}$. By Corollary~\ref{CThree}, we have that $(S_1^*)^G=0$ if $G$ is primitive. Now the result follows from Theorems~\ref{p4}, \ref{TJS} and \ref{TAlpha} and \cite[Proposition 2.5]{KS2Tran}.
\end{proof}

\subsection*{Proof of Theorem~\ref{TA}}

For $p>3$ the theorem holds by \cite{BK}.

Assume now that either $p=3$ or $p=2$, $n$ is odd and $\la\not=\be_n$. Then by \cite[Theorem 3.10]{KS2Tran} we have $G\leq\SSS_{n-1}$ or $G$ is $2$-transitive. If $G\leq\SSS_{n-1}$ then $\la$ is JS. So we may now assume that this is not the case. For $p=3$ the theorem then holds by Theorem \ref{Tp=3}, while for $p=2$, $n$ odd and $\la\not=\be_n$ the theorem holds by Theorem \ref{Tp=2odd}.

For $p=2$, $n$ even and $\la\not=\be_n$ the theorem holds by Theorem \ref{Tp=2even}. 


\begin{thebibliography}{ABC}

\bibitem[Al]{Alp}
J.L. Alperin, Diagrams for modules, {\em J. Pure Appl. Algebra} {\bf 16} (1980), 111--119. 

\bibitem[A]{Asch}
M. Aschbacher, On the maximal subgroups of the finite classical groups, {\em Invent. Math.} {\bf 76} (1984), 469--514.

\bibitem[B]{Benson} D.J. Benson, Spin modules for symmetric groups, {\em J. London Math. Soc.} {\bf 38} (1988), 250--262.

\bibitem[BC]{BC} D.J. Benson and J.F. Carlson, Diagrammatic methods for modular representations and cohomology, {\em Comm. Algebra} {\bf 15} (1987), 53--121.


\bibitem[BeK]{BK1} C. Bessenrodt and A.S. Kleshchev, On tensor products of modular representations of symmetric groups, {\em Bull. London Math. Soc.} {\bf 32} (2000), 292--296.

\bibitem[BO]{BO} C. Bessenrodt and J.B. Olsson, On residue symbols and the Mullineux conjecture, {\em J. Algebraic Combin.} {\bf 7} (1998), 227--251.



\bibitem[BDR]{BDR}
   J.N. Bray, D.F. Holt, and C.M. Roney-Dougal, `{\it The Maximal Subgroups of the Low-dimensional Finite Classical Groups}',  London Math. Soc. 
Lecture Note Series no. {\bf 407}. Camb. Univ. Press, 2013.


\bibitem[BrK${}_1$]{BrK1} J. Brundan and A.S. Kleshchev, On translation functors for general linear and symmetric groups, {\em Proc. London Math. Soc.} {\bf 80 }(2000), 75--106.

\bibitem[BrK${}_2$]{BK}
  J. Brundan and A.S. Kleshchev, Representations of the symmetric group which 
are irreducible over subgroups, {\it J. Reine Angew. Math.} {\bf 530} (2001), 
145--190.

\bibitem[Atl]{Atlas} J.H. Conway, R.T. Curtis, S.P. Norton, R.A. Parker, and R.A. Wilson, `{\it An ATLAS of Finite Groups}', Clarendon Press, Oxford, 1985.

\bibitem[C]{Cam}
  P.J. Cameron, Finite permutation groups and finite simple groups,  {\em Bull. Lond. Math. Soc.} {\bf 13} (1981), 1--22.
  
\bibitem[CR]{CR} C.W. Curtis and I. Reiner, `{\it Representation Theory of Finite Groups and Associative Algebras}', Interscience Publishers, New York, 1962.

  

\bibitem[FK]{FK} B. Ford and A.S. Kleshchev, A proof of the Mullineux conjecture, {\em  Math. Z.} {\bf 226} (1997), 267--308. 






\bibitem[J${}_1$]{JamesBook}
G.D. James, `{\em The Representation Theory of the Symmetric Groups}', Lecture Notes in Mathematics, vol. {\bf 682}, Springer, NewYork/Heidelberg/Berlin, 1978.

\bibitem[J${}_2$]{James} G.D. James, On the minimal dimensions of irreducible representations of symmetric groups, {\em Math. Proc. Camb. Phil. Soc.} {\bf 94} (1983), 417--424.

\bibitem[J${}_3$]{JamesArcata} G.D. James, The representation theory of the symmetric groups, pp. 111--126 in {\em The Arcata Conference on Representations of Finite Groups (Arcata, Calif., 1986)}, Proc. Sympos. Pure Math., 47, Part 1, Amer. Math. Soc., Providence, RI, 1987.

\bibitem[JK]{JK}
G.D. James and A. Kerber, {\em The Representation Theory of the Symmetric Group}, Encyclopedia of Mathematics and its Applications, vol.~16,
  Addison-Wesley Publishing Co., Reading, Mass., 1981. 


 \bibitem[JaS]{JS}
 J.C. Jantzen and G.M. Seitz, On the representation theory of the symmetric groups, {\em  Proc. London Math. Soc.} {\bf 65} (1992), 475--504.


\bibitem[Ka]{Kantor}
W.M. Kantor, Homogeneous designs and geometric lattices, {\em J. Combin. Th. (A)} {\bf 38} (1985), 66--74.

\bibitem[KlL]{KlL} P.B. Kleidman and M.W. Liebeck, `{\it The Subgroup Structure of the Finite Classical Groups}', London Math. Soc. Lecture Note Ser. no. {\bf 129}, Cambridge University Press, 1990.

\bibitem[KlW]{KlW} P.B. Kleidman and D.B. Wales, The projective characters of the symmetric groups that remain irreducible on subgroups, {\it J. Algebra} {\bf 138} (1991), 440--478.

\bibitem[K${}_1$]{k2} A.S. Kleshchev. On restrictions of irreducible modular representations of semisimple algebraic groups and symmetric groups to some natural subgroups, I. {\em  Proc. London Math. Soc.} {\bf 69} (1994), 515--540.


\bibitem[K${}_2$]{KBrII} A.S. Kleshchev, Branching rules for modular representations of symmetric groups. II, {\em J. Reine Angew. Math.} {\bf 459} (1995), 163--212. 

\bibitem[K${}_3$]{KBrIII} A.S. Kleshchev, Branching rules for modular representations of symmetric groups, III, {\em J. Lond. Math. Soc.} (2) {\bf 54} (1996), 25--38.

\bibitem[K${}_4$]{KDec} A.S. Kleshchev, On decomposition numbers and branching coefficients for symmetric and special linear groups, {\em Proc. London Math. Soc. (3)} {\bf 75} (1997), 497--558. 

\bibitem[K${}_5$]{KBook}
A.S. Kleshchev, {\em Linear and Projective Representations of Symmetric Groups}, Cambridge University Press, Cambridge, 2005. 

\bibitem[KMT]{KMT}
  A.S. Kleshchev, L. Morotti and P.H. Tiep, Irreducible restrictions of representations of symmetric and alternating groups in small characteristics; {\em in preparation}. 


\bibitem[KS${}_1$]{KS2Tran}
A.S. Kleshchev and J. Sheth, Representations of the symmetric group are reducible over simply transitive subgroups, {\em Math. Z.} {\bf 235}  (2000), 99--109.

\bibitem[KS${}_2$]{KSAlt}
  A.S. Kleshchev and J. Sheth, Representations of the alternating group
which are irreducible over subgroups, {\it Proc. London Math. Soc.} 
{\bf 84} (2002), 194--212.

\bibitem[KST]{KST}
  A.S. Kleshchev, P. Sin, and P.H. Tiep, Representations of the alternating group which are irreducible over subgroups. 
II, {\it Amer. J. Math.} {\bf 138} (2016), 1383--1423.

\bibitem[KT${}_1$]{KT} A.S. Kleshchev and P.H. Tiep, On restrictions of modular spin representations of symmetric and alternating groups, {\em Trans. Amer. Math. Soc.} {\bf 356} (2004), 1971--1999.

\bibitem[KT${}_2$]{KT2} A.S. Kleshchev and P.H. Tiep, Representations of general linear groups which are irreducible over subgroups, {\it Amer. J. Math.} {\bf 132} (2010), 425--473.


\bibitem[LPS]{LPS} M.W. Liebeck, C.E. Praeger, and J. Saxl, On the O'Nan-Scott reduction theorem for finite primitive permutation groups, {\em  J. Austral. Math. Soc.} {\bf 44} (1988), 389--396.

\bibitem[Mag]{Magaard}
K. Magaard, 
Some remarks on maximal subgroups of finite classical groups; in: {\em Finite Simple Groups: Thirty Years of the Atlas and Beyond}, 
Contemp. Math., 694, Amer. Math. Soc., Providence, RI, 2017, pp. 123--137.

\bibitem[Ma]{Martin}
S. Martin, {\em Schur Algebras and Representation Theory},  Cambridge Tracts in Mathematics, 112. Cambridge University Press, Cambridge, 1993. 

\bibitem[Mo]{M} L. Morotti, Irreducible tensor products for symmetric groups in characteristic 2, {\em Proc. London Math. Soc. (3)} {\bf 116} (2018), 1553--1598.



\bibitem[MO]{MO} J. M{\"u}ller, J. Orlob, On the structure of the tensor square of the natural module of the symmetric group, {\em Algebra Colloq.} {\bf 18} (2011), 589--610.

\bibitem[P]{P} A.M. Phillips, On $2$-modular representations of the symmetric groups, {\em  J. Algebra} {\bf 290} (2005), 282--294.

\bibitem[S]{Saxl}  J. Saxl, Irreducible characters of the symmetric groups that remain irreducible in subgroups, {\it J. Algebra} {\bf 111} (1987), 210--219.

\bibitem[Sc]{Scott} 
L. L. Scott, Representations in characteristic $p$, in {\em The Santa Cruz Conference on Finite Groups (Univ.
California, Santa Cruz, Calif., 1979)}, Proc. Sympos. Pure Math., vol. 37, Amer. Math. Soc., Providence, RI, 1980, pp. 319--331.





\bibitem[W]{Wales}
  D.B. Wales, Some projective representations of $S_{n}$, {\it J.
Algebra} {\bf 61} (1979), 37--57.

\bibitem[Wi]{Wil} R.M. Wilson, A diagonal form for the incidence matrices of $t$-subsets vs. $k$-subsets, {\it European J. Combin.} {\bf 11} (1990), 609--615.


\end{thebibliography}
\end{document}